\documentclass[10pt, reqno]{amsart}

\makeatletter
\@namedef{subjclassname@2020}{%
	\textup{2020} Mathematics Subject Classification}
\makeatother

\usepackage{amsmath, amsthm, amssymb, cite, eufrak,enumitem,bbm,color,setspace,ragged2e,mathtools, mathabx, amsfonts, latexsym}
\usepackage{hyperref}
\usepackage[T1]{fontenc}
\usepackage{enumerate}
\usepackage[hmargin=3.5cm,vmargin=3.5cm]{geometry}
\usepackage{caladea}

\newtheorem{theorem}{Theorem}[section]

\newtheorem{lemma}[theorem]{Lemma}

\newtheorem{corollary}[theorem]{Corollary}
\numberwithin{equation}{section}

\newcommand{\tuplea}{\textbf{a}}
\newcommand{\tupleb}{\textbf{b}}

\newcommand{\tupleg}{\textbf{g}}
\newcommand{\tupleh}{\textbf{h}}

\newcommand{\tupleu}{\textbf{u}}
\newcommand{\tuplex}{\textbf{x}}
\newcommand{\tupley}{\textbf{y}}
\newcommand{\tuplez}{\textbf{z}}

\newcommand{\tuplealpha}{\boldsymbol{\alpha}}

\newcommand{\tuplebeta}{\boldsymbol{\beta}}
\newcommand{\tuplegamma}{\boldsymbol{\gamma}}

\newcommand{\tuplezeta}{\boldsymbol{\zeta}}
\newcommand{\tupleeta}{\boldsymbol{\eta}}
\newcommand{\tuplelambda}{\boldsymbol {\lambda}}
\newcommand{\tuplemu}{\boldsymbol{\mu}}
\newcommand{\tuplexi}{\boldsymbol{\xi}}

\newcommand{\sinc}{\text{sinc}}

\newcommand{\mmod}[1]{\,\,(\text{mod}\,\,#1)}

\title{Simultaneous equations and inequalities}

\author{Constantinos Poulias}
\address{CP: School of Mathematics, Fry Building, Woodland Road, Clifton, Bristol BS8 1UG, UK.}\email{constantinos.poulias@gmail.com}

\begin{document}
	\date{}
	\maketitle
	\begin{abstract}
		Let $\lambda_i, \mu_j$ be  non-zero real numbers not all of the same sign and let  $a_i, b_k$ be non-zero integers not all of the same sign. We investigate a mixed Diophantine system of the shape
		\begin{equation*}  
			\begin{cases}	
				\left| \lambda_1 x_1^\theta + \cdots + \lambda_\ell x_\ell^\theta + \mu_1 y_1^\theta + \cdots + \mu_m y_m^\theta \right| < \tau \\[10pt]
				a_1 x_1^d + \cdots a_\ell x_\ell^d + b_1 z_1^d + \cdots + b_n z_n^d =0,		
			\end{cases}
		\end{equation*}
		where  $ d\geq 2 $ is an integer, $ \theta > d+1$ is real and non-integral and $ \tau $ is a positive real number. For such systems we obtain an asymptotic formula for the number of positive integer solutions $(\tuplex, \tupley, \tuplez) = (x_1, \ldots, z_n)$ inside a bounded box. Our approach makes use of a two-dimensional version of the classical Hardy--Littlewood circle method and the Davenport--Heilbronn--Freeman method. The proof involves a combination of essentially optimal mean value estimates for the auxiliary exponential sums, together with estimates stemming from the classical Weyl and Weyl--van der Corput inequalities.
	\end{abstract}

\section{Introduction}

In this paper we investigate the simultaneous solubility of inequalities and equations. Here we seek to count the number of positive integer solutions of a mixed system, consisting of a diagonal inequality of fractional degree and a diagonal integral form.

Fix non-zero real numbers $ \lambda_i, \mu_j $ not all of the same sign and non-zero integers  $ a_i, b_k$ not all of the same sign. Suppose that $ d \geq 2 $ is an integer and suppose further that $ \theta > d+ 1$ is real and non-integral. We write 
	\begin{equation}  \label{eq3.1.1}
		\begin{cases}	
			\mathfrak F (\tuplex, \tupley) = \lambda_1 x_1^\theta + \cdots + \lambda_\ell x_\ell^\theta + \mu_1 y_1^\theta + \cdots + \mu_m y_m^\theta \\[15pt]
			\mathfrak D (\tuplex, \tuplez) = a_1 x_1^d + \cdots a_\ell x_\ell^d + b_1 z_1^d + \cdots + b_n z_n^d.		
		\end{cases}
	\end{equation}
We shall write $ s = \ell + m +n$ to denote the total number of variables. Let $ \tau $ be a fixed positive real number. The Diophantine system under investigation is of the shape
	\begin{equation} \label{eq3.1.2}
		\begin{cases}	
			\left| \mathfrak F( \tuplex, \tupley) \right| < \tau \\[10pt] 
			\mathfrak D (\tuplex, \tuplez) =0.
		\end{cases}
	\end{equation}
We ask for the system 
	\begin{equation} \label{eq3.1.3}
		\mathfrak F (\tuplex, \tupley) = \mathfrak D (\tuplex, \tuplez) = 0
	\end{equation}
to admit a non-trivial (i.e. with at least one non-zero component) real solution $(\tuplex, \tupley, \tuplez) \in \mathbb R^s.$ Beyond the indefiniteness of $ \mathfrak F $ and $ \mathfrak D,$ in order to study the solubility of the system (\ref{eq3.1.2}) over the set of natural numbers one has to impose some further conditions. It is apparent that we must ask for the congruence $  \mathfrak D (\tuplex, \tuplez) \equiv 0 \mmod{p^\nu} $ to be soluble for all prime powers $ p^\nu.$ Furthermore, for reasons associated with the application of the circle method, one has to assume that the given local solutions are in fact non-singular. For us a tuple $\tupleeta = (\tuplex^\star, \tupley^\star ,\tuplez^\star) \in \mathbb R^s$ which satisfies the system of equations (\ref{eq3.1.3}) is called a non-singular solution of the system (\ref{eq3.1.2}) if the Jacobian matrix
	\begin{equation*}
		\frac{ \partial (\mathfrak F, \mathfrak D)}{ \partial (\eta_1, \ldots, \eta_s)}
	\end{equation*}
has full rank. We say that the system (\ref{eq3.1.2}) satisfies the \textit{local solubility condition} if the system (\ref{eq3.1.3}) possesses a non-singular real solution and the congruence $  \mathfrak D (\tuplex, \tuplez) \equiv 0 \mmod{p^\nu} $ possesses a non-singular solution for all prime powers $ p^\nu.$
We write $ \tupleeta = (\tuplex^\star, \tupley^\star, \tuplez^\star) \in \mathbb R^s$ to denote  a non-singular solution of the system (\ref{eq3.1.3}). Using the implicit function theorem in a standard fashion, one can deduce the existence of a real solution $\tupleeta$ to the system (\ref{eq3.1.3}) with $\eta_i \neq 0 $ for all $i.$ For the sake of completeness we include a proof of this in Lemma \ref{lem3.6.4}. Suppose that $ \tupleeta$ is such a solution. Using the homogeneity of the system (\ref{eq3.1.3}) one may suppose that $0 < |\eta_i| < 1/2$ for all $i.$ By changing signs if necessarily to the coefficients, we can assume that $ 0 < \eta_i < 1/2 $ for all $i.$ From now on we suppose that $\tupleeta = (\tuplex^\star, \tupley^\star, \tuplez^\star) \in \mathbb R^s$ is such a non-singular real solution of the system  (\ref{eq3.1.3}). 

Let $P$ be a sufficiently large positive real parameter. We write $ \mathcal N (P)$ to denote the number of positive integer solutions $(\tuplex, \tupley, \tuplez)$ of the system (\ref{eq3.1.2}) with 
	\begin{equation*} 
		\frac{1}{2} \tuplex^\star P < \tuplex \leq 2 \tuplex^\star P, \hspace{0.1in} \frac{1}{2} \tupley^\star P < \tupley \leq 2 \tupley^\star P, \hspace{0.1in} \frac{1}{2} \tuplez^\star P < \tuplez \leq 2 \tuplez^\star P.
	\end{equation*}
Our aim is to establish an asymptotic formula for the counting function $ \mathcal N (P)$ as $ P \to \infty.$ Throughout the paper we make use of standard notation in the field such as Vinogradov and Landau symbols. We recall this notation at the end of the introduction. For the sake of clarity let us mention here that for $ x \in \mathbb R $ we write $ \lfloor x \rfloor = \max \{ n \in \mathbb Z : n \leq x \}$ and $\lceil x \rceil = \min \{ n \in \mathbb Z: n \geq x \}$ to denote the floor and the ceiling function respectively.

Before we state our result we make a comment about two special cases. Suppose that $ \ell =0.$ It is apparent from \cite[Theorem 1.1]{poulias_ineq_frac} and the seminal work of Davenport and Lewis \cite{davenport_lewis_homo_add_eqts} that in such a case and provided that $ m \geq \left( \lfloor 2\theta  \rfloor +1 \right) \left( \lfloor 2\theta  \rfloor +2 \right)+1$ and $ n \geq d^2 + 1,$ one certainly has $ \mathcal N (P) \gg P^{m+n - (\theta +d)}.$ Suppose now that $ m = n =0.$ Here one would (in principle) be able to obtain an asymptotic formula for the counting function $ \mathcal N (P)$  provided that $ s= \ell \geq \ell_0(\theta) + 1,$ where $ \ell(\theta) $ is any natural number for which one has the estimate
	\begin{equation*}
		\int_0^1 \int_0^1 \left| \sum_{1 \leq x \leq P} e(\alpha_d x^d + \alpha_\theta x^\theta) \right|^{\ell_0 (\theta)} \text d \tuplealpha \ll P^{ \ell_0 (\theta) - \theta + \epsilon}. 
	\end{equation*}
Here $ \text d \tuplealpha $ stands for $ \text d \alpha_d \text d \alpha_\theta.$ Our first result establishes this observation.

	\begin{theorem} \label{thm3.1.1}
		Suppose that $ d \geq 2$ is an integer and suppose further that $ \theta > d+1  $ is real and non-integral. Let $ \tau $ be a fixed positive real number. Consider the system
		\begin{equation}  \label{eq3.1.4}
			\left| \mathfrak F( \tuplex, \tupley) \right| < \tau \hspace{0.1in} \text{and} \hspace{0.1in} \mathfrak D (\tuplex, \tuplez) =0,
		\end{equation}
		where $ \mathfrak F$ is an indefinite generalised polynomial and $ \mathfrak D$ is an indefinite integral polynomial defined in (\ref{eq3.1.1}). Suppose that $m=n=0$ and suppose further that the system (\ref{eq3.1.4}) satisfies the local solubility condition, namely the system (\ref{eq3.1.3}) possesses a non-singular real solution and the congruence $  \mathfrak D (\tuplex, \tuplez) \equiv 0 \mmod{p^\nu} $ possesses a non-singular solution for all prime powers $ p^\nu.$ Then, provided that $ s\geq \left( \lfloor 2\theta  \rfloor +1 \right) \left( \lfloor 2\theta  \rfloor +2 \right) + 1,$ one has that there exists a positive real number $ C = C(\tuplelambda,\tuplea, \theta, d,s)$ such that 
			\begin{equation} \label{eq3.1.5}
				\mathcal N (P) = 2 \tau C P^{s - (\theta+d)} + o \left( P^{s - (\theta+d)} \right),
			\end{equation}
		as $ P \to \infty.$ In particular, the number of positive integer solutions $ \tuplex \in [1,P]^s$ of the system (\ref{eq3.1.4}) is $ \gg  P^{s- (\theta +d)},$ where the implicit constant is a positive real number, which depends on $s, \lambda_i,  a_i, \theta, d$ and $ \tau.$ 
	\end{theorem}

Certainly more interesting is the case where in (\ref{eq3.1.1}) one has $ m + n \neq  0.$ Our next result examines this case when the total number of variables $s$ is in an intermediate range compared to the number of variables needed in the scenarios where $\ell  = 0$ or $m=n =0.$
	
	\begin{theorem}  \label{thm3.1.2}
		Suppose that $ d \geq 2$ is an integer and suppose further that $ \theta > d+1 $ is real and non-integral.  Let $ \tau $ be a fixed positive real number. Consider the system
			\begin{equation}  \label{eq3.1.6}
				\left| \mathfrak F( \tuplex, \tupley) \right| < \tau \hspace{0.1in} \text{and} \hspace{0.1in} \mathfrak D (\tuplex, \tuplez) =0,
			\end{equation}
		where $ \mathfrak F$ is an indefinite generalised polynomial and $ \mathfrak D$ is an indefinite integral polynomial defined in (\ref{eq3.1.1}). We write
			\begin{equation} \label{eq3.1.7} 
				A_\theta =  \left( \lfloor 2\theta  \rfloor +1 \right) \left( \lfloor 2\theta  \rfloor +2 \right) \hspace{0.1in} \text{and} \hspace{0.1in} A_d = d^2.
			\end{equation}
		Moreover, we set
			\begin{equation*} 
				s_{ \min} = \left \lceil  \max \left\{ A_\theta + n, \hspace{0.04in}  \frac{A_d}{A_\theta} m + A_\theta \right\} \right \rceil + 1
			\end{equation*}
		and 
			\begin{equation*} 
				s_{\normalfont \max} = \left \lfloor \min  \left\{ A_\theta + A_d, \hspace{0.04in} A_\theta + \frac{A_d}{A_\theta}m + n \right\} \right \rfloor + 1.
			\end{equation*}
		Suppose that the system (\ref{eq3.1.6}) satisfies the following conditions.
			\begin{itemize}
				\item[(a)] The system (\ref{eq3.1.6}) satisfies the local solubility condition,  namely the system (\ref{eq3.1.3}) possesses a non-singular real solution and the congruence $  \mathfrak D (\tuplex, \tuplez) \equiv 0 \mmod{p^\nu} $ possesses a non-singular solution for all prime powers $ p^\nu.$
				\item[(b)] One has $ \ell \geq \max \{ \lceil2\theta (1 -n/d) \rceil, \hspace{0.05in} 1\}, \hspace{0.05in}  0 \leq m \leq A_\theta$ and $ 0 \leq n \leq A_d.$   
				\item[(c)] One has $ \ell + m \geq A_\theta +1 $ and $ \ell + n \geq A_d +1 .$ 
				\item[(d)] For the total number of variables $ s = \ell + m +n $ one has $ s_{\normalfont \min} \leq s \leq s_{\normalfont \max}.$ 
			\end{itemize}
		Then, there exists a positive real number $ C = C(\tuplelambda,\tuplemu, \tuplea, \tupleb, \theta,d,s),$ such that as $ P \to \infty$ one has
			\begin{equation} \label{eq3.1.8} 
				\mathcal N (P) = 2 \tau C P^{s - (\theta+d)} + o \left( P^{s - (\theta+d)} \right).
			\end{equation}
		In particular, the number of positive integer solutions $(\tuplex, \tupley, \tuplez) \in [1,P]^\ell \times [1,P]^m \times [1,P]^n $ of the system (\ref{eq3.1.6}) is $ \gg P^{s- (\theta +d)},$ where the implicit constant is a positive real number, which depends on $s, \lambda_i, \mu_j, a_i, b_k, \theta, d$ and $ \tau.$
	\end{theorem}

Observe that the class of systems for which Theorem \ref{thm3.1.2} applies is non-empty. Let us list a few examples with explicit values for the parameters $\theta, d, \ell, m $ and $n,$ for which Theorem \ref{thm3.1.2} is applicable. In Table \ref{table_parameters_syst} below, for each choice we make for the parameters $m$ and $n$ we record the number $\ell $ of common variables required to apply Theorem \ref{thm3.1.2}. One may choose any real number $\theta$ in the given interval. The shape of the intervals has been chosen merely for convenience in the computations. Certainly, one can apply the theorem when $ \theta$ is the endpoint of the given interval (when non-integral).
  
\begin{table}[ht] \label{table_parameters_syst} 
	\caption{ Some values for the parameters $d, \theta, m, n, \ell$} 	
	\centering 
	\begin{tabular}{c c c c c c } 
		\hline
		$d$ \hspace{0.2in} & $\theta$ \hspace{0.2in} & $m$ \hspace{0.2in} & $n$ \hspace{0.2in} & $\ell$ \hspace{0.2in} & $s$\\[0.5ex]\hline\hline 
		$2$ \hspace{0.2in} & $(3, \hspace{0.05in} 3.5)$ \hspace{0.2in}  & $1$ \hspace{0.2in} & $0$ \hspace{0.2in} & $56$ \hspace{0.2in} & $57$  \\[5pt]
%
		$2$ \hspace{0.2in} & $(3, \hspace{0.05in} 3.5)$ \hspace{0.2in} & $40$ \hspace{0.2in} & $2$ \hspace{0.2in} & $ \{17, 18, 19\}$ \hspace{0.2in} & $\{59, 60, 61\}$  \\[5pt]
		$2$ \hspace{0.2in} & $(3.5, \hspace{0.05in} 4)$ \hspace{0.2in}  & $1$ \hspace{0.2in} & $0$ \hspace{0.2in} & $72$ \hspace{0.2in} & $73$  \\[5pt]
%
%
		$2$ \hspace{0.2in} & $(3.5, \hspace{0.05in} 4)$ \hspace{0.2in}  & $40$ \hspace{0.2in} & $2$ \hspace{0.2in} & $\{33, 34, 35\}$ \hspace{0.2in} & $\{75,76, 77\}$  \\[5pt]
		$3$ \hspace{0.2in} & $(4, \hspace{0.05in} 4.5)$ \hspace{0.2in}  & $1$ \hspace{0.2in} & $1$ \hspace{0.2in} & $90$ \hspace{0.2in} & $92$  \\[5pt]			
		$3$ \hspace{0.2in} & $(4.5, \hspace{0.05in} 5)$ \hspace{0.2in}  & $1$ \hspace{0.2in} & $1$ \hspace{0.2in} & $110$ \hspace{0.2in} & $112$  \\[5pt]		
		$3$ \hspace{0.2in} & $(4.5, \hspace{0.05in} 5)$ \hspace{0.2in}  & $55$ \hspace{0.2in} & $7$ \hspace{0.2in} & $\{56, 57,58\}$ \hspace{0.2in} & $\{118, 119, 120\}$  \\[5pt]
		$4$ \hspace{0.2in} & $(5, \hspace{0.05in} 5.5)$ \hspace{0.2in}  & $1$ \hspace{0.2in} & $1$ \hspace{0.2in} & $132$ \hspace{0.2in} & $134$  \\[5pt]				
		\hline 
	\end{tabular}
\end{table}

We say now a word about the asymptotic formula. The positive real number $ C = C(\tuplelambda,\tuplemu, \tuplea, \tupleb, \theta,d,s)$ appearing in the asymptotic formula (\ref{eq3.1.5}) (and similarly in the case of the asymptotic formula (\ref{eq3.1.8})) turns out to be a product of the shape $ C = \mathfrak J_0 \mathfrak S .$ Here
	\begin{equation*}
		\mathfrak J_0 = \int_{-\infty}^\infty \int_{-\infty}^\infty \left( \int_{\mathcal B} e \left( \beta_\theta \mathfrak F (\tuplex, \tupley) + \beta_d \mathfrak D (\tuplex, \tuplez) \right) \text d \tuplex \text d \tupley \text d \tuplez  \right) \text d \tuplebeta,
	\end{equation*}
where 
	\begin{equation*}
		\mathcal B  = \bigtimes_{i=1}^\ell \left[ \frac{1}{2} x_i^\star, 2 x_i^\star \right] \bigtimes_{j=1}^m \left[ \frac{1}{2} y_j^\star, 2 y_j^\star \right] \bigtimes_{k=1}^n \left[  \frac{1}{2} z_k^\star, 2 z_k^\star \right]
	\end{equation*}
is a box containing in its interior a non-singular solution $\tupleeta = (\tuplex^\star, \tupley^\star, \tuplez^\star)$ of the system (\ref{eq3.1.3}). The singular integral $ \mathfrak J_0$ is essentially Schmidt's singular integral. The singular series $ \mathfrak S,$ which captures the arithmetic behind the equation $ \mathfrak D(\tuplex, \tuplez) = 0,$ is given by
	\begin{equation*}
		\mathfrak S = \sum_{q=1}^\infty \sum_{\substack{a = 1 \\ (a,q)=1}}^q T(q, a), 	
	\end{equation*}	
where	
	\begin{equation*}
	T(q, a) = q^{-(\ell + n)} \prod_{i =1}^\ell S(q, aa_i) \prod_{k=1}^n S(q, ab_k),
	\end{equation*}
and for $a \in \mathbb Z$ and $q \in \mathbb N $ we write
	\begin{equation*}
		S (q, a) = \sum_{z=1}^q e \left(\frac{a z^d}{q} \right).
	\end{equation*}

By the assumptions made in Theorem \ref{thm3.1.2} we see that our conclusion is valid for systems for which the total number of variables $ s = \ell + m + n $ satisfies $ A_\theta + 1  \leq s \leq A_\theta + A_d + 1.$  Note that when $m=n=0$ in Theorem \ref{thm3.1.1} we assume that $ s= \ell \geq A_\theta + 1,$ with $A_\theta $ defined in (\ref{eq3.1.7}). The treatment of the minor arcs in the proof of Theorem \ref{thm3.1.1} follows  by using a Hua's type inequality 
	\begin{equation*}
		\int_{\mathcal B} \left| f(\alpha_d, \alpha_\theta) \right|^s \text d \tuplealpha \ll \left( \sup_{(\alpha_d, \alpha_\theta) \in \mathcal B} \left| f(\alpha_d, \alpha_\theta) \right| \right)^{s-2t} \int_{\mathcal B} \left| f(\alpha_d, \alpha_\theta) \right|^{2t} \text d \tuplealpha,
	\end{equation*}
as in \cite{poulias_ineq_frac}, where for $ (\alpha_d, \alpha_\theta) \in \mathbb R^2$ we write
	\begin{equation*}
		f(\alpha_d, \alpha_\theta) = \sum_{  1 \leq x \leq  P} e(\alpha_d x^d + \alpha_\theta x^\theta),
	\end{equation*}
and where $ \mathcal B $ is a Lebesgue measurable subset of $ \mathbb R^2.$ For the case where $ m + n \neq  0$ one can adopt the methods we use in proving Theorem \ref{thm3.1.2} together with an application of H{\"o}lder's inequality to treat the additional variables, in order to deal with systems where the total number of variables is greater than $ A_\theta + A_d + 1.$ For such cases we obtain the following corollary. 

\begin{corollary} \label{cor3.1.3}
	Suppose that $ d \geq 2$ is an integer and suppose further that $ \theta > d+1  $ is real and non-integral. Let $ \tau $ be a fixed positive real number. Consider the system
		\begin{equation}  \label{eq3.1.9}
			\left| \mathfrak F( \tuplex, \tupley) \right| < \tau \hspace{0.1in} \text{and} \hspace{0.1in} \mathfrak D (\tuplex, \tuplez) =0,
		\end{equation}
	where $ \mathfrak F$ is an indefinite generalised polynomial and $ \mathfrak D$ is an indefinite integral polynomial defined in (\ref{eq3.1.1}). Suppose that the system (\ref{eq3.1.9}) satisfies the following conditions.
	\begin{itemize}	
		\item[(a)] The system satisfies the local solubility condition, namely the system (\ref{eq3.1.3}) possesses a non-singular real solution and the congruence $  \mathfrak D (\tuplex, \tuplez) \equiv 0 \mmod{p^\nu} $ possesses a non-singular solution for all prime powers $ p^\nu.$
		\item[(b)] One has  $ \ell \geq \max \{ \lceil 2 \theta (1 - n/d) \rceil, \hspace{0.05in} 1 \}, \hspace{0.05in} 0 \leq m \leq A_\theta$ and $ 0 \leq n \leq A_d,$ with $A_\theta$ and $ A_d$ as in (\ref{eq3.1.7}).  
		\item[(c)] One has $ \ell + m \geq A_\theta + 1 $ and $ \ell + n \geq A_d +1,$ with $A_\theta$ and $ A_d$ as in (\ref{eq3.1.7}). 
		\item[(d)] One has $ s = \ell + m + n \geq A_\theta + A_d + 2.$
	\end{itemize}
	Then, the number of positive integer solutions $(\tuplex, \tupley, \tuplez) \in [1,P]^\ell \times [1,P]^m \times [1,P]^n $ of the system (\ref{eq3.1.9}) is $ \gg P^{s - (\theta +d)},$ where the implicit constant is a positive real number, which depends on $s, \lambda_i, \mu_j, a_i, b_k, \theta, d$ and $ \tau.$
\end{corollary}

Having stated our results let us make a few comments regarding previous works that are of some relevance to the problem we study. The study of Diophantine inequalities for diagonal real forms begins with the work of Davenport and Heilbronn \cite{davenport-heilbr-ineq}. Many authors have engaged with studying the solubility of systems of diagonal real forms of the same degree. For example, Cook \cite{cook_quad_ineq} studied pairs of quadratic inequalities in $s=9$ variables with real algebraic coefficients. Br{\"u}dern and Cook \cite{brued_cook_pairs_cub} considered pairs of cubic inequalities in $s=15$ variables, making similar assumptions as in \cite{cook_quad_ineq}. This improved a previous result due to Pitman \cite{pitman_quadr_diag_ineq_systm1981}. Moreover,  Br{\"u}dern and Cook in \cite{brued_cook_simult_diag_eq_ineq} considered simultaneous real diagonal forms of odd degree. For systems of diagonal real forms of like odd degree $ k \geq 13$ we have the important work of Nadesalingam and Pitman \cite{nadesalingam_pitman}.  That result contains implicitly the case where the forms are multiplies of rational forms. For the case of unlike degrees we have the important work of Schmidt \cite{schmidt_systm_odd_degree_ineq} who studied systems of real (not necessarily diagonal) forms of differing odd degrees. In this work Schmidt proves the existence (without being explicitly determined) of a finite lower bound for the number of variables needed to ensure solubility. For the first time, such an explicit bound was given by Freeman \cite{freeman_syst_cubic_ineq2004} in the case of a system of cubic forms.

Using ideas from \cite{bentkus_goetze_paper}, Freeman in \cite{freeman_lower_bounds} and \cite{freeman_asymp_formula} introduced a variant of the Davenport--Heilbronn method and established the anticipated lower bound and asymptotic formula for the number of integer solutions of diagonal real forms inside a box. These results of Freeman were afterwards improved by Wooley in \cite{wooley_dioph_ineq} using an amplification method. Building on his variant of the original Davenport--Heilbronn method, Freeman considered systems of diagonal quadratic real forms in \cite{freeman_quadr_ineq} and systems of diagonal real forms of degree $d$ in \cite{freeman_syst_diag_ineq}. The results of the latter paper concern as well systems of inequalities of even degree.  Moreover, the irrationality condition that was used in \cite{freeman_quadr_ineq} is now removed, hence the obtained results concern mixed systems consisting of equations and inequalities.

For the case of additive inequalities of unlike degree we begin with the work of Parsell \cite{parsell_ineqI}. In that paper, motivated by Wooley's work  on simultaneous additive equations \cite{wooley_simult_eqts_91}, \cite{wooley_simult_addt_eqtIV} and using Wooley's methods on exponential sums over smooth numbers \cite{wooley_exp_sum_smooth_numb}, Parsell developed a two dimensional version of the Davenport -- Heilbronn method. Shortly afterwards, in \cite{parsell_ineqII} and  \cite{parsell_ineqIII} Parsell adapted Freeman's method to study the solubility of systems of diagonal real forms of unlike degree. More precisely, in \cite{parsell_ineqII} Parsell considers the case of a pair of quadratic and cubic inequalities, while in \cite{parsell_ineqIII} the focus is on $R$ simultaneous inequalities of unlike degrees $k_1 > k_2 >  \cdots >  k_R \geq 1.$ In both cases it is established the anticipated asymptotic lower bound for the number of integer solutions inside a sufficiently large box. Though it is not directly related to the present work, for some recent developments concerning systems of simultaneous additive equations one may look in the papers of Wooley \cite{wooley_pair_quadr_cub_2015}, Brandes and Parsell \cite{brandes_parsell_simult_addit_eqts} and Brandes  \cite{brandes_hasse_princ_quadr_cubic}.  

Coming now to additive problems with non--integral exponents, let us begin by saying that the first such investigations can be traced back to Segal in the 1930's \cite{segal_1933_I}, \cite{segal_1933_II}, \cite{segal_1934_ineq}. For diagonal inequalities of fractional degree the anticipated asymptotic formula for the number of integer solutions inside a box was established in \cite{poulias_ineq_frac}. Key element of the proof is an essentially optimal mean value estimate for exponential sums involving fractional powers of integers. Such a mean value estimate, which however was $P^{1/2}$ from the near optimal, was first appeared in the important work of Arkhipov and Zhitkov \cite{arkhipov_zitkov_warings_probl_non_integr_exp} concerning Waring's problem with non--integral exponent. In \cite[Theorem 1.2]{poulias_approx_TDI_system} we obtain an essentially optimal mean value estimate for exponential sums associated to Approximately Translation--Dilation invariant systems of Vinogradov type, whereas now the highest degree equation is replaced by an inequality for a generalised polynomial with leading term $x^\theta,$ where $ \theta > 2$ is real and non-integral. A special case of this result is quoted below in Theorem \ref{thm3.3.3}. 

We finish this short exposition with the paper of Chow \cite{chow_birch_shifts} which is an inequality analogue of Birch's celebrated result \cite{birch_forms}. The interested reader may look as well in the recent breakthroughs  due to Myerson \cite{myerson_quadratic} and \cite{myerson_cubic}, who obtained a remarkable improvement compared to Birch's theorem for systems of quadratic and cubic integral forms. 
 
\bigskip

\textit{Notation.} Below we collect a few pieces of notation that we use in the rest of the paper. For  $ x \in \mathbb R $ we write $ e(x) $ to denote $ e^{2 \pi i x}$ with $ i = \sqrt{-1}$ being the imaginary unit. For a complex number $z $ we write $ \overline z $ to denote its complex conjugate. For a function $f : \mathbb Z \to \mathbb C $ and for two real numbers $m, M,$ whenever we write
	\begin{equation*}
		\sum_{ m < x \leq M} f(x)
	\end{equation*}
the summation is to be understood over the integers that belong to the interval $(m, M].$ We make use of the standard symbols of Vinogradov and Landau. Namely, when for two functions $f,g$ there exists a positive real constant $ C $ such that $ | f(x) | \leq C | g(x) | $ for all sufficiently large $x$ we write $ f(x) = O (g(x))$ or $ f(x) \ll g (x).$ We write $ f \asymp g$ to denote the relation $ g \ll f \ll g.$  Furthermore, we write $ f(x) = o (g(x)) $ if $ f(x) / g(x) \to 0 $ as $ x \to \infty $ and we write $ f \sim g $ if $ f(x) / g(x) \to 1 $ as $ x \to \infty .$ Throughout, the letter $ \epsilon$ denotes a sufficiently small positive real number. Unless specified otherwise, the implicit constants in the Vinogradov and Landau symbols are allowed to depend on $\lambda_i, \mu_j, a_i, b_k, s, \theta,d, \tau, \epsilon$ and $ \tupleeta,$ where recall that $ \tupleeta = (\tuplex^\star, \tupley^\star, \tuplez^\star)$ denotes a certain non-singular real solution of the system (\ref{eq3.1.3}). Occasionally, we highlight the dependence on some of these parameters by using subscripts. The implicit constants are not allowed to depend on $P.$ For a given real number $x$ we shall write $ \lfloor x \rfloor = \max \{ n \in \mathbb Z : n \leq x \}$ and $\lceil x \rceil = \min \{ n \in \mathbb Z: n \geq x \}$ to denote the floor and the ceiling function respectively. An expression of the shape $ m < \tuplex \leq M$ where $ m < M$ and $ \tuplex = (x_1, \ldots, x_n)$ is an $n$-tuple, is to be understood as $ m< x_1, \ldots, x_n \leq M.$ In a similar fashion, an expression of the shape $ \tupley < \tuplex \leq \tuplez$ where $\tupley = (y_1, \ldots, y_n)$ and $ \tuplez = (z_1, \ldots,z_n)$ are $n$-tuples, is to be understood componentwise as $y_i < x_i \leq z_i$ for all $ 1 \leq  i\leq n.$

\section{Set up}
	
\subsection{An analytic representation for the counting function $ \mathcal N (P)$} 

Set $ \widetilde \tau = \tau (\log P)^{-1}.$ We put
	\begin{equation} \label{eq3.2.1}
		\displaystyle K_\pm (\alpha) = \frac{ \sin \left(\pi \alpha \widetilde \tau \right) \sin \left( \pi \alpha (2 \tau  \pm \widetilde \tau) \right)}{ \pi^2 \alpha^2 \widetilde \tau }.
	\end{equation}
By \cite[Lemma 1]{freeman_asymp_formula} and its proof we know that
	\begin{equation} \label{eq3.2.2}
		K_{\pm}  (\alpha) \ll_{\tau} \min \{1, | \alpha|^{-1}, (\log P) |\alpha|^{-2} \},
	\end{equation}
and
	\begin{equation} \label{eq3.2.3}
		0 \leq \int_{-\infty}^\infty e( \xi \alpha) K_{-}( \alpha) \text{d} \alpha \leq \chi_\tau (\xi) \leq \int_{-\infty}^\infty  e( \xi \alpha) K_{+}( \alpha) \text{d} \alpha \leq 1,
	\end{equation}	
where we write $\chi_\tau (\xi)$ to denote the indicator function of the interval $(-\tau, \tau),$ namely
	\begin{equation*} 
		\chi_\tau (\xi) = 			
			\begin{cases}
				1, \hspace{0.1in} \text{if} \hspace{0.1in} | \xi | < \tau, \\				
				0, \hspace{0.1in} \text{if} \hspace{0.1in} | \xi| \geq \tau.
			\end{cases}
	\end{equation*}	
Note that the expression 
	\begin{equation*} 
		\left| \int_{-\infty}^\infty  e( \xi \alpha) K_{\pm}( \alpha) \text{d} \alpha - \chi_\tau (\xi) \right|
	\end{equation*}
is zero when $ | | \xi| - \tau | > \widetilde \tau$ and at most $1$ for values of $ \xi$ such that $ | |\xi| - \tau | \leq \widetilde \tau .$ 

One can rewrite the kernel functions $K_\pm (\alpha)$ defined in (\ref{eq3.2.1}) in the shape
	\begin{equation*}
		K_\pm (\alpha) = (2 \tau  \pm \widetilde \tau) \frac{ \sin \left(\pi \alpha \widetilde \tau \right)}{ \pi \alpha \widetilde \tau} \cdot  \frac{\sin \left( \pi \alpha (2 \tau  \pm \widetilde \tau) \right)}{ \pi \alpha (2 \tau  \pm \widetilde \tau) }.
	\end{equation*}
Using a Taylor expansion one has for $ | x| < 1 $ with $ x \neq 0 $  that
	\begin{equation*}
		\frac{ \sin x}{x} = 1 + O (x^2).
	\end{equation*}
Recall that  $ \widetilde \tau = \tau (\log P)^{-1}.$ So for $ | \alpha | < 1$ and $P$ sufficiently large one has that
	\begin{equation} \label{eq3.2.4}
		K_\pm (\alpha) = 2 \tau + O \left( \left( \log P \right)^{-2} \right).
	\end{equation}

In our analysis we use various exponential sums. For $ \tuplealpha = (\alpha_d, \alpha_\theta) \in \mathbb R^2$ we define the exponential sums $ f(\alpha_d, \alpha_\theta) = f(\alpha_d, \alpha_\theta; P), \hspace{0.03in} g(\alpha_\theta )  = g (\alpha_\theta ; P)$ and $ h(\alpha_d) = h(\alpha_d ; P)$ by 
	\begin{equation*}
		\begin{split}	
			& f(\alpha_d, \alpha_\theta ; P) = \sum_{ 1 \leq x \leq P}  e(  \alpha_d x^d +  \alpha_\theta x^\theta), \\[10pt]		
			& g (\alpha_\theta ; P) = \sum_{ 1 \leq x \leq P } e (\alpha_\theta x^\theta), \\[10pt]	 
			& h(\alpha_d ; P) = \sum_{ 1 \leq x \leq P } e (\alpha_d x^d).
		\end{split}
	\end{equation*}
Moreover, we define $ F_i (\tuplealpha) = F_i (\tuplealpha ; P), \hspace{0.03in} G_j (\alpha_\theta) = G_j (\alpha_\theta ;P)$ and $ H_k (\alpha_d) = H(\alpha_d ; P)$ by 
	\begin{equation*}
		\begin{split}	
			& F_i (\alpha_d, \alpha_\theta; P) = \sum_{1 \leq x \leq P} e (a_i \alpha_d x^d + \lambda_i \alpha_\theta x^d )  \hspace{0.7in}  (1 \leq  i \leq \ell), \\[10pt]
			& G_j (\alpha_\theta ; P) =  \sum_{1 \leq x \leq P} e ( \mu_j \alpha_\theta x^\theta )  \hspace{1.45in}  (1 \leq  j \leq m), \\[10pt]
			& H_k (\alpha_d; P ) = \sum_{1 \leq x \leq P} e ( b_k \alpha_dx^d )  \hspace{1.45in} (1 \leq  k \leq n).
		\end{split}
	\end{equation*}	
Recall that $(\tuplex^\star, \tupley^\star, \tuplez^\star)$ is a non-singular real solution of the system (\ref{eq3.1.3}). We put
	\begin{equation*}
		\begin{split}	
			& f_i(\alpha_d, \alpha_\theta) = \sum_{ \frac{1}{2} x_i^\star P < x \leq 2 x_i^\star P } e ( a_i \alpha_d x^d +  \lambda_i \alpha_\theta x^\theta) \hspace{0.5in}  (1 \leq  i\leq \ell), \\[10pt]
			& g_j (\alpha_\theta) = \sum_{ \frac{1}{2} y_j^\star P < y \leq 2 y_j^\star P} e( \mu_j \alpha_\theta y^\theta) \hspace{1.3in} (1 \leq j \leq m), \\[10pt]
			& h_k (\alpha_d) = \sum_{ \frac{1}{2} z_k^\star P < z \leq 2 z_k^\star P } e( b_k \alpha_d z^d) \hspace{1.3in}  ( 1 \leq k \leq n).		
		\end{split}
	\end{equation*}
Occasionally, we may write $ f_i(\tuplealpha)$  to denote the exponential sum $f_i(\alpha_d, \alpha_\theta).$ Similarly, we write $ g_j(\tuplealpha)$ to denote the exponential sum $ g_j(\alpha_\theta)$ and $ h_k(\tuplealpha)$ to denote the exponential sum $h_k(\alpha_d).$ We do the same with the other exponential sums defined above. For future reference we note here the following relations
	\begin{equation} \label{eq3.2.5}
		\begin{split}	
			&f_i (\alpha_d, \alpha_\theta) = F \left(\alpha_d, \alpha_\theta; 2 x_i^\star P \right) -  F \left(\alpha_d, \alpha_\theta; \frac{1}{2} x_i^\star P \right), \\[10pt]	
			& g_j (\alpha_\theta) = G \left( \alpha_\theta; 2 y_j^\star P \right) - G \left(\alpha_\theta ; \frac{1}{2} y_j^\star P \right), \\[10pt]				
			& h_k (\alpha_d) = H \left(\alpha_d; 2 z_k^\star P \right) - H \left(\alpha_d ; \frac{1}{2} z_k^\star P \right).
		\end{split}
	\end{equation}	

We define the generating function
	\begin{equation*}
		\mathcal F( \tuplealpha ) = \prod_{i=1}^\ell f_i(\alpha_d, \alpha_\theta) \prod_{j=1}^m g_j(\alpha_\theta) \prod_{k=1}^n h_k(\alpha_d),  
	\end{equation*}
and set
	\begin{equation} \label{eq3.2.6}
		R_{\pm} (P) = \int_{-\infty}^\infty \int_0^1 \mathcal F( \tuplealpha ) K_{\pm} (\alpha_\theta) \text d \tuplealpha.
	\end{equation}
Using now (\ref{eq3.2.3}), together with the usual orthogonality relation 
	\begin{equation*}
			\int_0^1 e( \alpha n) \text d \alpha =
					\begin{cases}
						1, & \text{when } n = 0, \\[10pt]
						0, & \text{when } n \in \mathbb Z  \setminus \{0\}, 
					\end{cases}
	\end{equation*}
one has that
	\begin{equation*} 
		R_{-}(P) \leq \mathcal N (P) \leq R_{+}(P).
	\end{equation*}
From the above inequality it is clear that in order to establish an asymptotic formula for the counting function $ \mathcal N (P)$ it suffices to obtain asymptotic formulae for the integrals $R_{\pm}(P)$ that are asymptotically equal.

\subsection{A mixed version of the circle method}

In order to study the integrals $ R_\pm(P)$ defined in (\ref{eq3.2.6}) we apply a mixed version of the circle method. We dissect separately $ \mathbb R$ and $[0,1).$

\textit{Dissection of $\hspace{0.03in} \mathbb R.$} Here we apply a Davenport--Heilbronn dissection. Write $ \gamma = \theta - \lfloor \theta \rfloor \in (0,1)$ for the fractional part of $\theta.$ Define the parameters $ \delta_0 = \delta_0 (\theta)$ and $ \omega = \omega(\theta)$ by
	\begin{equation} \label{eq3.2.7}
		\delta_0 (\theta ) = 2^{ 1 - 2\theta}  \hspace{0.3in}  \text{and} \hspace{0.3in} \omega(\theta) =  \min \left\{ \frac{1-\gamma}{12}, \hspace{0.03in} 5^{-100(\theta+d)} \right\}.
	\end{equation} 
Define the set of major, minor, and trivial arcs respectively as follows
	\begin{equation*}
		\begin{split}	
			& \mathfrak M = \left\{ \alpha_\theta \in \mathbb R : | \alpha_\theta | < P^{-\theta + \delta_0} \right\}, \\[10pt]			 
			& \mathfrak m = \left\{ \alpha_\theta \in \mathbb R : P^{-\theta + \delta_0}  \leq | \alpha_\theta | < P^{ \omega } \right\},  \\[10pt]			
			& \mathfrak t = \left\{ \alpha_\theta \in \mathbb R :  | \alpha_\theta | \geq  P^{ \omega } \right\}.
		\end{split}
	\end{equation*}

\textit{ Dissection of $ \hspace{0.03in} [0,1).$} Here we apply a classical Hardy--Littlewood dissection into major and minor arcs. Pick a parameter $ \xi$ satisfying 
	\begin{equation} \label{eq3.2.8}
		0 < \xi \leq \frac{\delta_0}{8}.
	\end{equation} 
For integers $ a,q $ such that $ 0 \leq a < q \leq P^\xi $ and $(a,q)=1,$ we define a major arc around the rational fraction $a/q$ to be the set
	\begin{equation*}
		\mathfrak N_\xi (q,a) = \{ \alpha_d \in [0,1): \hspace{0.05in} \left| \alpha_d - a/q \right| < P^{-d + \xi} \}.
	\end{equation*}
We now form the union 
	\begin{equation*}
		\mathfrak N_\xi = \bigcup_{ \substack{ 0 \leq a <  q \leq P^\xi \\ (a,q) =1 }} 	\mathfrak N_\xi (q,a),
	\end{equation*}
and call this the set of major arcs. Note that $ \mathfrak N_\xi$ is a union of disjoint sets. Indeed, suppose that there exists $ \alpha_d \in [0,1) $ which belongs to two distinct major arcs $ \mathfrak N_\xi (q_1, a_1), \mathfrak N_\xi (q_2, a_2) \subset \mathfrak N_\xi.$ Since $ a_1/q_1 \neq a_2 / q_2 $ one has
	\begin{equation*}
		\displaystyle \frac{1}{q_1 q_2} \leq  \left| \frac{a_1 q_2 - a_2 q_1}{q_1 q_2} \right| \leq 2 P^{- d + \xi},
	\end{equation*}
which in turn implies that $ 1 \leq 2q_1 q_2 P^{- d + \xi} \leq 2 P^{-d + 3 \xi}.$ This is clearly impossible for large $P,$ since by our choice in (\ref{eq3.2.8}) one has $ \xi < 1/3 .$ The set of minor arcs is defined to be the complement of the set of major arcs. Denote this set by $ \mathfrak n_\xi .$ Namely we have
	\begin{equation*}
		\mathfrak n_\xi = [0,1) \setminus \mathfrak N_\xi.
	\end{equation*}

Using the above dissections one can express $ [0,1) \times \mathbb R $ as a disjoint union of sets of the shape
	\begin{equation*}
		[0,1) \times \mathbb R = \mathfrak P \cup \mathfrak p \cup \mathfrak c, 
	\end{equation*}
where we define the sets $ \mathfrak P, \mathfrak p $ and $ \mathfrak c $ as follows.
	\begin{itemize}
		\item[(1)] The set of major arcs $ \mathfrak P $ given by
			\begin{equation*}
				\mathfrak P  = \mathfrak N_\xi \times \mathfrak M.
			\end{equation*}
		\item [(2)] The set of minor arcs $ \mathfrak p $ given by
			\begin{equation*}
				\mathfrak p = \left( [0,1) \times  \mathfrak m) \right) \cup \left( \mathfrak n_\xi \times  \mathfrak M   \right).
			\end{equation*}
		\item[(3)] The set of trivial arcs $ \mathfrak c $ given by
			\begin{equation*}
				\mathfrak c  = [0,1) \times  \mathfrak t.
			\end{equation*}
	\end{itemize} 

For a Lebesgue measurable set $ \mathcal B \subset [0,1) \times  \mathbb R $ we define
	\begin{equation} \label{eq3.2.9}
		R_\pm ( P ; \mathcal B)  = \int_{\mathcal B} \mathcal F (\tuplealpha) K_\pm ( \alpha_\theta ) \text d \tuplealpha.
	\end{equation}
Recalling (\ref{eq3.2.6}), one has that
	\begin{equation} \label{eq3.2.10}
		R_\pm(P) = R_\pm(P ; \mathfrak P) + R_\pm (P ; \mathfrak p ) + R_\pm (P ; \mathfrak c).
	\end{equation} 

\subsection{An application of H{\"o}lder's inequality}
	
We begin by recalling the well known inequality
	\begin{equation*}
		| z_1 \cdots z_n | \ll | z_1 |^n + \cdots | z_n |^n,
	\end{equation*}
which is valid for all complex numbers $z_i.$ Let $ \mathcal B $ be a Lebesgue measurable set. An application of this inequality reveals that for some indices $i,j$ and $k$ one has
	\begin{equation*}
		| \mathcal F (\tuplealpha) | \ll |f_i(\alpha_d, \alpha_\theta)|^\ell |g_j(\alpha_\theta)|^m |h_k(\alpha_d)|^n.
	\end{equation*}
Let $ \delta \in [0,  1/3) $ be a real number at our disposal to be chosen at a later stage. We write  	
	\begin{equation} \label{eq3.2.11}
		\ell^\prime = \ell - \delta \hspace{0.1in} \text{and} \hspace{0.1in} s^\prime = \ell^\prime + m + n = s - \delta.
	\end{equation}
Note here that $ \ell^\prime, s^\prime \notin \mathbb N.$ The previous estimate yields
	\begin{equation} \label{eq3.2.12}
		\begin{split}	
			\int_{ \mathcal B} | \mathcal F (\tuplealpha)  K_\pm(\alpha_\theta)| \text d \tuplealpha \ll & \left(  \sup_{ (\alpha_d, \alpha_\theta) \in  \mathcal B} | f_i(\alpha_d, \alpha_\theta) | \right)^{ \delta}  \times  \\ 
			& \times \int_{\mathcal B} |f_i(\alpha_d, \alpha_\theta)|^{\ell^\prime} |g_j(\alpha_\theta)|^m |h_k(\alpha_d)|^n |K_\pm(\alpha_\theta)| \text d \tuplealpha.
		\end{split}
	\end{equation}
	
We define the following auxiliary mean values, 	
	\begin{equation*} 
		\begin{split}
			& \Xi_{f_i}( \mathcal B) = \int_{\mathcal B}  | f_i(\alpha_d, \alpha_\theta)|^{A_\theta } |K_\pm(\alpha_\theta)| \text d \tuplealpha, \\[15pt]
			&\Xi_{f_i,g_j} (\mathcal B) = \int_{\mathcal B}  |f_i(\alpha_d, \alpha_\theta)|^{A_d} |g_j(\alpha_\theta)|^{A_\theta} |K_\pm(\alpha_\theta)| \text d \tuplealpha,  \\[15pt]
			& \Xi_{f_i,h_k} (\mathcal B) = \int_{\mathcal B} |f_i(\alpha_d, \alpha_\theta)|^{A_\theta} |h_k(\alpha_d)|^{A_d} |K_\pm(\alpha_\theta)| \text d \tuplealpha, \\[15pt]
			& \Xi_{g_j,h_k} (\mathcal B) = \int_{\mathcal B} |g_j(\alpha_\theta)|^{A_\theta} |h_k(\alpha_d)|^{A_d} |K_\pm(\alpha_\theta)| \text d \tuplealpha. 
		\end{split}			 
	\end{equation*}		
For $ \omega_i \in (0,1)$ with $ \omega_1 + \cdots + \omega_4 =1 $ a formal application of H\"{o}lder's inequality reveals 	
	\begin{equation} \label{eq3.2.13}
		\begin{split}	
			\int_{\mathcal B} & |f_i(\alpha_d, \alpha_\theta)|^{\ell^\prime} |g_j(\alpha_\theta)|^{m} |h_k(\alpha_d)|^{n} |K_\pm(\alpha_\theta)| \text d \tuplealpha \ll \\[10pt]
			& \ll \left( \Xi_{f_i}  (\mathcal B) \right)^{\omega_1}	\left( \Xi_{f_i,g_j}  (\mathcal B) \right)^{\omega_2} \left( \Xi_{f_i,h_k}  (\mathcal B) \right)^{\omega_3} \left( \Xi_{g_j,h_k}  (\mathcal B) \right)^{\omega_4}. 		
		\end{split}
	\end{equation}
Combining (\ref{eq3.2.13}) and (\ref{eq3.2.12}) yields
	\begin{equation} \label{eq3.2.14}
		\begin{split}	
			\int_{\mathcal B}  | \mathcal F (\tuplealpha) K_\pm(\alpha_\theta)| \text d \tuplealpha & \ll \left( \sup_{ (\alpha_d, \alpha_\theta) \in  \mathcal B} | f_i(\alpha_d, \alpha_\theta) | \right)^{\delta} \times  \\[10pt] 
			& \times \left( \Xi_{f_i}  (\mathcal B) \right)^{\omega_1}	\left( \Xi_{f_i,g_j}   (\mathcal B)\right)^{\omega_2} \left( \Xi_{f_i,h_k}  (\mathcal B) \right)^{\omega_3} \left( \Xi_{g_j,h_k}  (\mathcal B) \right)^{\omega_4}. 	
		\end{split}
	\end{equation}

The task now is to prove that there exist admissible values $ \omega_i$ such that the inequality (\ref{eq3.2.13}) is valid. The $ \omega_i \in (0,1)$ must satisfy the simultaneous linear equations
	\begin{equation*} 
		\begin{cases}
			A_\theta  \omega_1 + A_d \omega_2 + A_\theta \omega_3 = \ell^\prime \\[10pt]
			A_\theta \omega_2 + A_\theta \omega_4 = m \\[10pt]
			A_d \omega_3 + A_d \omega_4 = n \\[10pt]
			\omega_1 + \omega_2 + \omega_3 + \omega_4 = 1.
		\end{cases}
	\end{equation*} 
By the two equations in the middle we infer that
	\begin{equation*}
		\omega_2 = \omega_3 + \frac{m}{A_\theta} - \frac{n}{A_d}.
	\end{equation*}
Substituting $ \omega_2 + \omega_4  = m/ A_\theta$ into the last equation of the system yields
	\begin{equation*}
		\omega_1 = - \omega_3 + 1  - \frac{m}{A_\theta}.
	\end{equation*}
One can substitute into the first equation of the system the above values for $ \omega_2 $ and $ \omega_1.$ Hence
	\begin{equation*}
		\omega_3 = \frac{ s^\prime - A_\theta}{A_d} - \frac{m}{A_\theta}.
	\end{equation*}
Having determined a value for $ \omega_3$ one can solve for $ \omega_1, \omega_2$ and $ \omega_4$ to obtain
	\begin{equation} \label{eq3.2.15}
		\omega_1 = 1 - \frac{ s^\prime  - A_\theta}{A_d}, \hspace{0.3in} \omega_2 = \frac{ s^\prime - A_\theta}{A_d}  - \frac{n}{A_d}, \hspace{0.3in}  \omega_4 = \frac{m}{A_\theta} + \frac{n}{A_d} - \frac{ s^\prime - A_\theta}{A_d}.
	\end{equation}

We now have to ensure that $ \omega_i \in (0,1).$ Since $ \omega_1 + \omega_2 + \omega_3 + \omega_4 = 1 $ it suffices to ensure that $ \omega_i > 0.$ Solving the simultaneous inequalities $ \omega_i > 0 \hspace{0.05in} (1 \leq i \leq 4)$ yields
	\begin{equation*}
		\max \left\{ A_\theta + n, \hspace{0.04in}  \frac{A_d}{A_\theta} m + A_\theta \right\} \leq s^\prime \leq \min \left\{ A_\theta + A_d, \hspace{0.04in} A_\theta + \frac{A_d}{A_\theta}m + n \right\}.
	\end{equation*} 
Note that this is a legitimate constraint since we assume that $ 0 \leq m \leq A_\theta$ and $ 0 \leq n \leq A_d.$

Next, we deduce a constraint for $s.$ Recall from (\ref{eq3.2.11}) that $ s^\prime = s - \delta.$  Since we consider $ s$ to be a natural number, the preceding inequality about the range of $ s^\prime $ now delivers
	\begin{equation*}
		\left \lceil \delta + \max \left\{ A_\theta + n, \hspace{0.04in}  \frac{A_d}{A_\theta} m + A_\theta \right\} \right \rceil \leq s \leq \left \lfloor \delta + \min \left\{ A_\theta + A_d, \hspace{0.04in} A_\theta + \frac{A_d}{A_\theta}m + n \right\} \right \rfloor.
	\end{equation*}
For any $x, y \in \mathbb R$ one has
	\begin{equation*}
		\begin{split}
			&\lfloor x \rfloor + \lfloor y \rfloor \leq \lfloor x + y \rfloor \leq \lfloor x \rfloor + \lfloor y \rfloor + 1 \\[10pt]
			& \lceil x \rceil + \lceil y \rceil -1 \leq \lceil x+ y \rceil \leq \lceil x \rceil + \lceil y \rceil + 1.
		\end{split}
	\end{equation*}
Since $ 0 \leq \delta < 1/3 < 1$ one has
	\begin{equation*}
		\left \lfloor \delta + \min \left\{ A_\theta + A_d, \hspace{0.04in} A_\theta + \frac{A_d}{A_\theta}m + n \right\} \right \rfloor \geq \left \lfloor  \min \left\{ A_\theta + A_d, \hspace{0.04in} A_\theta + \frac{A_d}{A_\theta}m + n \right\} \right \rfloor + 1,
	\end{equation*}
and 
	\begin{equation*}
		\left \lceil \delta + \max \left\{ A_\theta + n, \hspace{0.04in}  \frac{A_d}{A_\theta} m + A_\theta \right\} \right \rceil  \leq \left \lceil  \max \left\{ A_\theta + n, \hspace{0.04in}  \frac{A_d}{A_\theta} m + A_\theta \right\} \right \rceil + 1.
	\end{equation*}
Hence one has
	\begin{equation*}
		\left \lceil  \max \left\{ A_\theta + n, \hspace{0.04in}  \frac{A_d}{A_\theta} m + A_\theta \right\} \right \rceil + 1 \leq s \leq  \left \lfloor \min  \left\{ A_\theta + A_d, \hspace{0.04in} A_\theta + \frac{A_d}{A_\theta}m + n \right\} \right \rfloor + 1,
	\end{equation*}	
which is precisely the range prescribed by the condition $(d)$ in the statement of Theorem \ref{thm3.1.2}. It is therefore clear that for such $s$ the inequality (\ref{eq3.2.13}) is valid. 

\section{Auxiliary mean value estimates} \label{auxiliary_estimates_section}

The aim of this section is to collect the necessary auxiliary estimates that we employ in the following sections. From now on, and for ease of notation, for each $ j \in \{1, \ldots, n, \theta\}$ we put
	\begin{equation} \label{eq3.3.1}
		\sigma_{t, j} (\tuplex) = \sum_{ i=1}^t (x_i^j - x_{t+i}^j).
	\end{equation}

	\begin{lemma} \label{lem3.3.1}
		Suppose that $ I \subset (0,\infty)$ is a finite interval. Let $ \delta $ be a given positive real number and define the number $ \Delta$ by the relation $ 2 \delta \Delta =1.$ We write $ V_t( I ; \delta)$ to denote the number of positive integer solutions $ x_i \in I$ of the inequality
			\begin{equation*}
				| \sigma_{t, \theta} (\tuplex) | < \delta.
			\end{equation*}
		Then one has
			\begin{equation*}
				\delta \int_{ - \Delta}^{ \Delta} \left| \sum_{x \in I}e(\alpha x^\theta) \right|^{2t} \normalfont \text d \alpha \ll  V_t(I; \delta ) \ll \delta \int_{ - \Delta}^{ \Delta} \left| \sum_{x \in I} e(\alpha x^\theta) \right|^{2t}  \normalfont \text d \alpha,
			\end{equation*}		
		with the implicit constants in the above estimate being independent from $I,\theta,$ and $ \delta.$
	\end{lemma}	

	\begin{proof}
		This is a special case of \cite[Lemma 3.2]{poulias_ineq_frac} with $I_1 = I_2 = I$ in their notation. 
	\end{proof}

Next, we need a variant of the above lemma that allows one to bound from above the mixed mean values $ \Xi_{f_i, g_j} (\mathcal B) $ and $ \Xi_{f_i,h_k} (\mathcal B),$ by the number of solutions of the corresponding underlying system. Let $\kappa$ be a positive real number. We write $ Z_1(P)$ to denote the number of integer solutions of the system 
	\begin{equation*}
		\begin{cases}	
			\displaystyle  \left| \lambda_i \sigma_{\frac{A_d}{2}, \theta} (\tuplex) + \mu_j \sigma_{\frac{A_\theta}{2}, \theta} (\tupley) \right| < \frac{1}{2 \kappa} \\[15pt]
			\displaystyle  a_i \sigma_{\frac{A_d}{2}, d} (\tuplex) =0,
		\end{cases}
	\end{equation*}
with $ \frac{1}{2} x_i^\star P < \tuplex \leq 2 x_i^\star P $ and $ \frac{1}{2} y_i^\star P <\tupley \leq 2 y_i^\star P.$ Similarly, we write $Z_2 (P)$ to denote the number of integer solutions of the system
	\begin{equation*}
		\begin{cases}	
			\displaystyle  \left|  \lambda_i \sigma_{\frac{A_\theta}{2}, \theta} (\tuplex) \right| < \frac{1}{2 \kappa} \\[15pt]
			\displaystyle a_i \sigma_{\frac{A_\theta}{2}, d} (\tuplex) + b_k \sigma_{\frac{A_d}{2}, d} (\tuplez) =0,
		\end{cases}
	\end{equation*}
with $ \frac{1}{2} x_i^\star P < \tuplex \leq 2 x_i^\star P $ and $ \frac{1}{2} z_i^\star P <\tuplez \leq 2 z_i^\star P.$

\begin{lemma} \label{lem3.3.2}
	Let $ \kappa $ be a positive real number and write $ \mathcal B = [-1,1] \times  [- \kappa, \kappa] .$ Then, for each index $i, j$ and $k$ one has
	\begin{itemize}
		\item[(i)] $ \displaystyle \Xi_{f_i, g_j}  (\mathcal B) \ll \kappa Z_1(P) $ ;
		\item[(ii)] $ \displaystyle \Xi_{f_i, h_k}  (\mathcal B) \ll \kappa Z_2(P) .$
	\end{itemize}
	The implicit constants do not depend on $ \kappa.$
\end{lemma}

\begin{proof}
	We give the proof only of estimate $(i).$ One can establish estimate $(ii)$ in a similar fashion.  As in \cite[Lemma 3.2]{poulias_ineq_frac}, the proof is inspired by \cite[Lemma 2.1]{watt-exp-sums-riemann-zeta-fnct2}.  
	
	Fix indices $i$ and $j.$ For ease of notation we put 
		\begin{equation*} 
			p (\tuplex, \tupley) = \lambda_i \sigma_{\frac{A_d}{2}, \theta} (\tuplex) + \mu_j \sigma_{\frac{A_\theta}{2}, \theta} (\tupley) \hspace{0.5in} \text{and} \hspace{0.5in} q (\tuplex) =  a_i \sigma_{\frac{A_d}{2}, d} (\tuplex).
		\end{equation*} 
	Then, $Z_1(P)$ is equivalently given by the number of integer solutions of the system
		\begin{equation*}
			\begin{cases}
				\displaystyle | p (\tuplex, \tupley) |  < \frac{1}{2 \kappa} \\[10pt]
				\displaystyle | q (\tuplex) | < \frac{1}{2}
			\end{cases}
		\end{equation*}
	with  $ \frac{1}{2} x_i^\star P < \tuplex \leq 2 x_i^\star P $ and $ \frac{1}{2} y_i^\star P <\tupley \leq 2 y_i^\star P.$
	
	Define the function
		\begin{equation*}
			 \text{sinc} (x) = 
				\begin{cases}
				\displaystyle  \frac{\sin (\pi x) }{ \pi x}, & \hspace{0.1in} \text{when} \hspace{0.1in} x \neq 0, \\[10pt]
				1, & \hspace{0.1in}  \text{when} \hspace{0.1in} x=0.
				\end{cases}
		\end{equation*}
	By \cite{davenport-heilbr-ineq} we know that for each $x, \xi \in \mathbb R$ one has
		\begin{equation*} \label{chapter4_sinc_triangl_fourier_pair}
			\Lambda (x) = \int_{- \infty}^\infty e (x \xi) \text{sinc}^2 (\xi ) \text d \xi,
		\end{equation*}							
	where for $x \in \mathbb R$  we write $ \Lambda (x) = \max \{0, 1 - |x| \}.$ Note that one has $ 0 \leq \Lambda(x) \leq 1.$ So, for each solution counted by $ Z_1 (P)$ one has
	$ 0 < \Lambda (2\kappa p(\tuplex, \tupley)) < 1 $ and  $ 0 < \Lambda (2q(\tuplex)) < 1.$  
	
	By the above considerations and taking the sum over the tuples $\tuplex, \tupley$ with  $ \frac{1}{2} x_i^\star P < \tuplex \leq 2 x_i^\star P $ and $ \frac{1}{2} y_i^\star P <\tupley \leq 2 y_i^\star P,$ we infer that
		\begin{equation*}
			\begin{split}
				\displaystyle Z_1 (P) & \geq \sum_{ \tuplex, \tupley} \Lambda (2 \kappa p (\tuplex, \tupley)) \Lambda ( 2 q (\tuplex)) \\[10pt]
				& = \sum_{ \tuplex, \tupley} \int_{- \infty}^{\infty} \int_{- \infty}^{\infty} e \left( u_1  \kappa p (\tuplex, \tupley) + u_2 2 q(\tuplex) \right) \sinc^2 (u_1) \sinc^2 (u_2) \text d \tupleu \\[10pt]
				& = \frac{1}{4 \kappa}  \sum_{ \tuplex, \tupley} \int_{- \infty}^{\infty} \int_{- \infty}^{\infty} e \left( \alpha_\theta p (\tuplex, \tupley) + \alpha_d q(\tuplex) \right) \sinc^2 \left( \frac{1}{2 \kappa} \alpha_\theta \right) \sinc^2 \left( \frac{1}{2} \alpha_d \right) \text d \tuplealpha,
			\end{split}
		\end{equation*}
	where in the last step we applied a change of variables under the transformation
		\begin{equation*}
			\begin{pmatrix}
				u_1 \\
				u_2
			\end{pmatrix}
			=
			\begin{pmatrix}
				\frac{1}{2 \kappa} & 0 \\
				0 & \frac{1}{2} 
				\end{pmatrix}
				\begin{pmatrix}
				\alpha_\theta \\
				\alpha_d
			\end{pmatrix}.
		\end{equation*}
	Because we have a finite sum and since the integral is absolutely convergent, one can change the order. Thus, by the above inequality we obtain
		\begin{equation} \label{eq3.3.2}
			Z_1 (P) \geq  \frac{1}{4 \kappa}  \int_{- \infty}^{\infty} \int_{- \infty}^{\infty} |f_i(\alpha_d, \alpha_\theta)|^{A_d} |g_j(\alpha_d)|^{A_\theta}  \sinc^2 \left( \frac{1}{2 \kappa} \alpha_\theta \right) \sinc^2 \left( \frac{1}{2} \alpha_d \right) \text d \tuplealpha.
		\end{equation}
	
	Next, we use Jordan's inequality, which states that for $ 0 < x \leq \frac{\pi}{2}$ one has
		\begin{equation*}
			\frac{2}{ \pi} \leq \frac{\sin x}{x} < 1.
		\end{equation*}
	For a proof of this inequality see \cite[p. 33]{mitrinovic_book}. One then has $ \sinc^2 (x) > 4/ \pi^2$ for $ | x| < \frac{1}{2}.$ Thus, for $ | \alpha_\theta| <  \kappa$ and $ | \alpha_d| <  1 $ one has 
		\begin{equation*} 
			\sinc^2 \left( \frac{1}{2 \kappa} \alpha_\theta \right), \hspace{0.1in} \sinc^2 \left( \frac{1}{2} \alpha_d \right) > 4/ \pi^2 .
		\end{equation*} 
	Hence, the inequality (\ref{eq3.3.2}) now delivers 
		\begin{equation*}
			Z_1 (P) \gg \frac{1}{\kappa} \int_{- \kappa }^{ \kappa} \int_{-1}^1  |f_i(\alpha_d, \alpha_\theta)|^{A_d} |g_j(\alpha_\theta)|^{A_\theta} \text d \tuplealpha,
		\end{equation*}
	which completes the proof.
\end{proof}

Temporarily we put $ n = \lfloor \theta \rfloor$ for the integer part of $ \theta.$ For a tuple $ \tuplealpha = (\alpha_1, \ldots, \alpha_n, \alpha_\theta) \in \mathbb R^{n+1}$ we put $ T(\tuplealpha) = T( \tuplealpha ; P ),$ where
	\begin{equation} \label{4_Texp_def}
		T (\tuplealpha ; P) = \sum_{ 1 \leq x \leq P} e (\alpha_1 x + \cdots + \alpha_n x^n + \alpha_\theta x^\theta).
	\end{equation}
We need the following mean value estimate.
	
	\begin{theorem} \label{thm3.3.3}
		Let $ \kappa \geq 1 $ be a real number and suppose that $ t \geq A_\theta / 2 $ is a natural number. Then, for any fixed $ \epsilon > 0$ one has
			\begin{equation*}
				\int_{ - \kappa}^{\kappa} \int_{[0,1)^n} \left| T(\tuplealpha) \right|^{2t} \normalfont \text d \tuplealpha \ll_{t, \theta, \epsilon} \kappa P^{2t - \frac{1}{2} n(n+1) - \theta + \epsilon}.
			\end{equation*}
		The implicit constant does not depend on $\kappa.$ Furthermore, for $t > A_\theta /2 $ one can take $ \epsilon  =0 .$	
	\end{theorem}

	\begin{proof}
		This is a special case of \cite[Theorem 1.2]{poulias_approx_TDI_system} with $ \phi(x) = x^\theta.$
	\end{proof}

Next, we obtain essentially optimal mean value estimate for the exponential sums $f,g$ and $h.$

	\begin{lemma} \label{lem3.3.4}
		Let $ \kappa \geq 1$ be a real number. Then the following are valid.
			\begin{itemize}
				\item [(i)] Suppose that $ t \geq A_\theta / 2 $ is a natural number. Then, for any fixed $ \epsilon > 0$ one has
					\begin{equation*}
						\int_{ - \kappa}^{\kappa} \int_0^1 \left| f(\alpha_d, \alpha_\theta) \right|^{2t} \normalfont \text d \tuplealpha \ll_{t,\theta, \epsilon} \kappa P^{2t - (\theta+d) + \epsilon}.
					\end{equation*}	
				\item [(ii)]  Suppose that $ t \geq A_\theta /2 $ is a natural number. Then, for any fixed $ \epsilon > 0$ one has
					\begin{equation*}
						\int_{- \kappa}^{ \kappa} \left| g(\alpha_\theta) \right|^{2t} \normalfont \text d \alpha_\theta \ll_{t,\theta, \epsilon} \kappa P^{2t -  \theta + \epsilon}.
					\end{equation*}
				\item [(iii)] Suppose that $ t \geq A_d / 2 $ is a natural number. Then, for any fixed $ \epsilon > 0$ one has
					\begin{equation*}
						\int_0^1 \left| h(\alpha_d) \right|^{2t} \normalfont \text d \alpha_d \ll_{t,d, \epsilon} P^{2t - d + \epsilon}.
					\end{equation*}
			\end{itemize}
		The implicit constants do not depend on $\kappa.$ Furthermore, for $ t > A_\theta /2$ in $ (i)$ and $(ii),$ and for $ t > A_d /2$ in $(iii),$ one can take $ \epsilon =0.$ 
	\end{lemma}

	\begin{proof}
		We begin with the estimate in $(iii).$ This follows from \cite[Corollary 14.7]{wooley_NEC} since 
		\\$ A_d \geq s_0(d),$ where $s_0(d)$ is defined as
			\begin{equation*}
				s_0(d) = d (d-1) + \min_{ 0 \leq m < d} \frac{ 2d + m (m-1)}{ m+1}.
			\end{equation*} 
		 
		The proof of the estimate in $(ii)$ can be found in \cite[Theorem 1.4]{poulias_ineq_frac}. Alternative, one can apply an argument similar to the one we present below for proving $(i).$
		
		We now come to the estimate in $(i).$ Temporarily we write $ n = \lfloor \theta \rfloor.$ Keep in mind that we suppose that $ \theta > d+ 1$ and so one has $ d < n.$ In order to prove the estimate in $(i)$ we apply an average process as in \cite[Theorem 2.1]{wooley_IMRN}. For each $ 1 \leq j \leq n$ with $ j \neq d $ and for a tuple $ \tupleh = (h_1, \ldots, h_{d-1}, h_{d+1}, \ldots, h_n) \in \mathbb Z^{n-1}$ we put
			\begin{equation*}
				\delta(\tuplex, \tupleh) = \prod_{ \substack {j=1 \\ j \neq d}}^n \int_0^1 e \left( \beta_j \left( \sigma_{t,j} (\tuplex) - h_j \right) \right) \text d \beta_j,
			\end{equation*}
		where recall from (\ref{eq3.3.1}) the definition of $ \sigma_{t,j} (\tuplex).$ Let us rewrite the exponential sum $T(\tuplealpha )$ defined in (\ref{4_Texp_def}) as
			\begin{equation*}
				T( \tuplebeta , \alpha_d, \alpha_\theta) = \sum_{1 \leq x\leq P} e (\beta_1 x + \cdots + \beta_{d-1} x^{d-1}  + \alpha_d x^d + \beta_{d+1} x^{d+1} + \cdots + \alpha_\theta x^\theta).
			\end{equation*}
		Note that
			\begin{equation} \label{eq3.3.4}
				\begin{split}	
					\int_{-\kappa}^\kappa \int_{[0,1)^n} & \left| T(\tuplebeta, \alpha_d, \alpha_\theta) \right|^{2t}  e \left( - \sum_{ \substack {j=1 \\ j \neq d}}^n \beta_j h_j \right) \text d \tuplebeta = \\[10pt]
					& = \sum_{ 1 \leq \tuplex \leq P} \delta(\tuplex, \tupleh)  \int_{- \kappa}^\kappa \int_0^1 e \left(  \alpha_d \sigma_{t,d} (\tuplex) + \alpha_\theta \sigma_{t,\theta} (\tuplex) \right) \text d \alpha_d \text d \alpha_\theta.
				\end{split}
			\end{equation}
			
		By orthogonality one has
			\begin{equation*}
				\int_0^1 e \left( \beta_j \left( \sigma_{t,j} (\tuplex) - h_j \right) \right) \text d \beta_j =
					\begin{cases}
						1, & \hspace{0.1in} \text{when} \hspace{0.1in}  \sigma_{t,j} (\tuplex) = h_j,  \\[10pt]
						0, & \hspace{0.1in}  \text{when} \hspace{0.1in} \sigma_{t,j} (\tuplex) \neq  h_j.
					\end{cases}
			\end{equation*}
		It is apparent that for each fixed choice of $ 1 \leq \tuplex \leq P$ there is precisely one possible value for the tuple $ \tupleh \in \mathbb Z^{n-1}.$ Moreover, for each $ j$ and for $ 1 \leq \tuplex \leq P$ one has $ \left| \sigma_{t,j} (\tuplex)  \right| \leq t P^j.$ Hence
			\begin{equation} \label{eq3.3.5}
				\sum_{ | h_1| \leq t P} \cdots \sum_{ | h_{d-1}| \leq tP^{d-1} } \sum_{ | h_{d+1}| \leq tP^{d+1} } \cdots \sum_{ | h_n| \leq tP^n } \delta(\tuplex, \tupleh) =1.
			\end{equation}
		One may return to (\ref{eq3.3.4}) and sum over tuples $ \tupleh$ satisfying $ |h_j| \leq tP^j$ for each $1 \leq j \leq n $ with $ j \neq d.$ Thus we obtain
			\begin{equation*}
				\begin{split}
					\sum_{ \tupleh} \int_{-\kappa}^\kappa \int_{[0,1)^n} & \left| T(\tuplebeta, \alpha_d, \alpha_\theta) \right|^{2t} e \left( - \sum_{ \substack {j=1 \\ j \neq d}}^n \beta_j h_j \right) \text d \tuplebeta  = \\[10pt]
					& = \sum_{ 1 \leq \tuplex \leq P} \left(\sum_{ \tupleh} \delta(\tuplex, \tupleh) \right)  \int_{- \kappa}^\kappa \int_0^1 e \left(  \alpha_d \sigma_{t,d} (\tuplex) + \alpha_\theta \sigma_{t,\theta} (\tuplex) \right) \text d \alpha_d \text d \alpha_\theta.
				\end{split}
			\end{equation*}
		Applying the triangle inequality and taking into account (\ref{eq3.3.5}) one has
			\begin{equation*}
				\begin{split}	
					P^{ \frac{1}{2} n(n+1) -d} \int_{-\kappa}^\kappa \int_{[0,1)^n} & \left| T(\tuplebeta, \alpha_d, \alpha_\theta) \right|^{2t} \text d \tuplebeta \geq \\[10pt]
					& \geq \sum_{1 \leq \tuplex \leq P} \int_{-\kappa}^{\kappa} \int_0^1 e \left(  \alpha_d \sigma_{t,d} (\tuplex) + \alpha_\theta \sigma_{t,\theta} (\tuplex) \right) \text d \alpha_d \text d \alpha_\theta.
				\end{split}
			\end{equation*}
		
		Note now that
			\begin{equation*}
				\sum_{1 \leq \tuplex \leq P} \int_{-\kappa}^{\kappa} \int_0^1 e \left(  \alpha_d \sigma_{t,d} (\tuplex) + \alpha_\theta \sigma_{t,\theta} (\tuplex) \right) \text d \alpha_d \text d \alpha_\theta = \int_{-\kappa}^{\kappa} \int_0^1 \left| f(\alpha_d, \alpha_\theta) \right|^{2t} \text d \tuplealpha.
			\end{equation*}
		Invoking Theorem \ref{thm3.3.3}, we deduce that for any fixed $ \epsilon > 0$ one has
			\begin{equation*}
				\begin{split}	
					\int_{-\kappa}^{\kappa} \int_0^1 \left| f(\alpha_d, \alpha_\theta) \right|^{2t} \text d \tuplealpha & \ll P^{ \frac{1}{2} n(n+1) -d} \cdot P^{2t - \frac{1}{2} n (n+1) - \theta + \epsilon} \\[10pt]
					& \ll P^{2t - (\theta +d) + \epsilon},
				\end{split}
			\end{equation*}
		which completes the proof.
	\end{proof}
	
Below we obtain mean value estimates for the exponential sums $f_i, g_j$ and $h_k.$

	\begin{lemma} \label{lem3.3.5}
		For each index $i,j$ and $k$ the following are valid.
			\begin{itemize}
				\item[(i)] Suppose that $ \kappa $ is a real number such that  $ \kappa | \lambda_i | \geq 1.$ Suppose further that $ t \geq A_\theta /2 $ is a natural number. Then, for any fixed $\epsilon > 0 $ one has
					\begin{equation*}
						\int_{ - \kappa}^{ \kappa} \int_0^1  |f_i (\alpha_d, \alpha_\theta)|^{2t} \hspace{0.05in} \normalfont \text d \tuplealpha \ll_{t,\theta, \epsilon, \lambda_i, a_i, x_i^\star} \kappa P^{2t - (\theta +d) + \epsilon}.
					\end{equation*}  		
				\item[(ii)]  Suppose that $ \kappa $ is a real number such that  $ \kappa | \mu_j | \geq 1.$ Suppose further that $ t \geq A_\theta / 2 $ is a natural number. Then, for any fixed $ \epsilon > 0$ one has				
					\begin{equation*}
						\int_{ - \kappa}^{\kappa}  |g_j(\alpha_\theta)|^{2t} \hspace{0.05in} \normalfont \text d \alpha_\theta \ll_{t,\theta, \epsilon, \mu_j, y_j^\star} \kappa P^{2t - \theta + \epsilon}.
					\end{equation*}
				\item[(iii)] Suppose that $ t \geq A_d / 2 $ is a natural number. Then, for any fixed $ \epsilon > 0,$ one has
					\begin{equation*}
						\int_0^1  |h_k(\alpha_d)|^{2t} \hspace{0.05in}  \normalfont \text d \alpha_d \ll_{t,d, \epsilon, b_k, z_k^\star} P^{2t- d + \epsilon}.
					\end{equation*}	
			\end{itemize}
		The implicit constants in the above estimates do not depend on $\kappa.$ Furthermore, for $ t > A_\theta /2$ in $ (i)$ and $(ii),$ and for $ t > A_d /2$ in $(iii),$ one can take $ \epsilon =0.$ 		
	\end{lemma}
	
	\begin{proof}
		We give a proof only for the estimate in $(i)$. One can argue in a similar fashion to establish the estimates in $(ii)$ and $(iii).$  
		
		Fix an index $i.$ Recalling (\ref{eq3.2.5}) we see that it suffices to prove the following estimate
			\begin{equation*}
				\int_{- \kappa}^{\kappa} \int_0^1 \left| F_i (\alpha_d, \alpha_\theta) \right|^{2t} \text d \tuplealpha \ll \kappa P^{2t - (\theta +d) + \epsilon}.
			\end{equation*}
		Making a change of variables  by
			\begin{equation*}
				\displaystyle
				\begin{pmatrix}
					\alpha_\theta \\
					\alpha_d
				\end{pmatrix}
				=
				\begin{pmatrix}
					\frac{1}{ | \lambda_i|} & 0 \\
					0 & \frac{1}{ | a_i|} 
				\end{pmatrix}
				\begin{pmatrix}
					\beta_\theta \\
					\beta_d
				\end{pmatrix},
			\end{equation*}
		yields
			\begin{equation*}
					\int_{- \kappa}^{\kappa} \int_0^1 \left| F_i (\alpha_d, \alpha_\theta) \right|^{2t} \text d \tuplealpha  = \frac{1}{ | \lambda_i a_i|} \int_{ - \kappa | \lambda_i|}^{ \kappa | \lambda_i|} \int_0^{ | a_i|} \left| f(\pm \beta_d, \pm \beta_\theta) \right|^{2t} \text d \tuplebeta.
			\end{equation*}
		One can chop the interval $[0, | a_i|]$ into at most $ \lfloor | a_i| \rfloor + 1$ intervals of length at most one. Moreover, because of the $1$-periodicity with respect to $ \beta_d$ one has
			\begin{equation*}
				\begin{split}	
					\int_{ - \kappa | \lambda_i|}^{ \kappa | \lambda_i|} \int_0^{ | a_i|} \left| f(\pm \beta_d, \pm \beta_\theta) \right|^{2t} \text d \tuplebeta & \ll \sum_{n =0}^{ \lfloor | a_i| \rfloor} \int_{ - \kappa | \lambda_i|}^{ \kappa | \lambda_i|}  \int_{n}^{n+1} \left| f(\pm \beta_d, \pm \beta_\theta ) \right|^{2t} \text d \tuplebeta \\[10pt]
					& \ll_{a_i} \int_{ - \kappa | \lambda_i|}^{ \kappa | \lambda_i|}  \int_0^1 \left| f(\pm \beta_d, \pm \beta_\theta ) \right|^{2t} \text d \tuplebeta.
				\end{split}
			\end{equation*}
		Finally, if necessary, one can make one more change of variables. This together with the fact that $ f( - \tuplebeta) = \overline{ f(\tuplebeta )} $ yields
			\begin{equation*}
				\int_{ - \kappa | \lambda_i|}^{ \kappa | \lambda_i|}  \int_0^1 \left| f(\pm \beta_d, \pm \beta_\theta ) \right|^{2t} \text d \tuplebeta = \int_{ - \kappa | \lambda_i|}^{ \kappa | \lambda_i|}  \int_0^1 \left| f( \beta_d,  \beta_\theta ) \right|^{2t} \text d \tuplebeta.
			\end{equation*} 
		The conclusion now follows by applying Lemma \ref{lem3.3.4}. 
	\end{proof}

We now estimate the auxiliary mean values $ \Xi_{f_i}, \Xi_{f_i, g_j}, \Xi_{f_i, h_k}$ and $\Xi_{g_j, h_k}.$ 

	\begin{lemma} \label{lem3.3.6}
		Let $ \kappa $ be a real number such that for each index $i$ and $j$ one has $ \kappa | \lambda_i| \geq 1 $ and $  \kappa | \mu_j| \geq 1.$ Let $ \mathcal B = [0,1] \times  [- \kappa, \kappa] .$ Then, for each index $i,j$ and $k,$ and for any fixed $ \epsilon >0 $ one has 
			\begin{itemize}
				\item[(i)]  $ \displaystyle \Xi_{f_i}  (\mathcal B)  \ll \kappa P^{A_\theta - (\theta +d) + \epsilon} $ ;
				\item[(ii)] $ \displaystyle \Xi_{f_i,g_j}  (\mathcal B)  \ll \kappa P^{A_\theta +A_d - (\theta +d) + \epsilon}$ ;
				\item[(iii)] $ \displaystyle \Xi_{f_i, h_k}  (\mathcal B) \ll\kappa P^{ A_\theta +A_d - (\theta +d) + \epsilon}$ ;
				\item[(iv)]	$ \displaystyle  \Xi_{g_j,h_k}  (\mathcal B) \ll \kappa P^{ A_\theta +A_d - (\theta +d) + \epsilon} .$
			\end{itemize}
		The implicit constants in the above estimates do not depend on $\kappa.$
	\end{lemma}

	\begin{proof} 
		 In the following we make use of the fact that by (\ref{eq3.2.2}) one has $ |K_\pm (\alpha_\theta)| \ll 1 .$ The estimate $(i)$ follows by part $(i)$ of Lemma \ref{lem3.3.5} with $ t = A_\theta /2.$  The proof of the estimate $(iv)$ is straightforward. One can write
			\begin{equation*}
					\Xi_{g_j, h_k}  \ll \left( \int_{- \kappa}^{  \kappa} \left| g_j (\alpha_\theta) \right|^{A_\theta} \text d \alpha_\theta \right) \left( \int_0^1 \left| h_k (\alpha_d) \right|^{A_d} \text d \alpha_d \right),
			\end{equation*} 
		and the conclusion now follows by using $(ii)$ and $(iii)$ of Lemma \ref{lem3.3.5}.
			
		Now we turn our attention to the estimate in $(ii).$ Fix indices $i$ and $j.$ We put 
			\begin{equation*}
				M = \frac{1}{2 \max_{i,j} \left\{ | \lambda_i|^{-1}, | \mu_j|^{-1}  \right\} }> 0,
			\end{equation*}
		which is a fixed real number. By the assumption  $ \kappa \geq \max_{i,j} \{ | \lambda_i|^{-1}, | \mu_j|^{-1} \}$ one has that $ 1/(2 \kappa) \leq M.$ Hence,  by Lemma \ref{lem3.3.2} and extending plainly the range of the inequality, one has (up to constants which are independent of $\kappa$) that
			\begin{equation} \label{Xfigj_bound_kZ1}
				\Xi_{f_i, g_j} \ll \kappa Z_1 (P) \ll \kappa Z_1^\prime (P),
			\end{equation}
		where $Z_1^\prime (P)$ denotes the number of integer solutions of the system
			\begin{equation} \label{eq3.3.7}
				\begin{cases}	
					\displaystyle \left| \lambda_i \sum_{ i =1}^{  \frac{A_d}{2}} \left(x_i^\theta - x_{\frac{A_d}{2} +i}^\theta \right) + \mu_j \sum_{ i= 1}^{  \frac{A_\theta}{2}} \left(y_i^\theta - y_{\frac{A_\theta}{2} +i}^\theta \right) \right| < M \\[15pt]
					\displaystyle a_i \sum_{ i= 1}^{ \frac{A_d}{2}} \left(x_i^d - x_{\frac{A_d}{2} +i}^d \right) =0,
				\end{cases}	
			\end{equation}			
		with  $ \frac{1}{2} x_i^\star P < \tuplex \leq 2 x_i^\star P $ and $ \frac{1}{2} y_i^\star P <\tupley \leq 2 y_i^\star P.$ By orthogonality, the number of integer solutions of the equation in (\ref{eq3.3.7}) is counted by the mean value 
			\begin{equation*}
				\int_0^1 \left| \sum_{  \frac{1}{2} x_i^\star P < x \leq 2 x_i^\star P} e(\alpha x^d) \right|^{A_d} \text d \alpha_d.
			\end{equation*}	
		Note that
			\begin{equation*}
				\left| \sum_{  \frac{1}{2} x_i^\star P < x \leq 2 x_i^\star P} e(\alpha x^d) \right| \ll \left| h \left( \alpha ; 2x_i^\star P \right) \right| + \left| h \left( \alpha ; \frac{1}{2} x_i^\star P \right) \right|.
			\end{equation*}
		So by Lemma \ref{lem3.3.4} one has for any fixed $ \epsilon > 0$ that
			\begin{equation*}
				 \int_0^1 \left| \sum_{  \frac{1}{2} x_i^\star P < x \leq 2 x_i^\star P} e(\alpha x^d) \right|^{A_d} \text d \alpha_d \ll P^{  A_d  - d + \epsilon}.
			\end{equation*} 
		
		Let us fix an integer solution $\tuplex$ for the equation in (\ref{eq3.3.7}). As we proved, this can be done by choosing among $  O \left( P^{  A_d - d + \epsilon} \right)$ possibilities. Substitute now these values into the inequality in (\ref{eq3.3.7}). Then the first block of variables is fixed and so one has to count the number of solutions of the inhomogeneous inequality
			\begin{equation*} \label{auxuliary_xi_figj_ineq_fixed_part}
					\displaystyle \left| \mu_j \sum_{ i= 1}^{  \frac{A_\theta}{2}} \left(y_i^\theta - y_{\frac{A_\theta}{2} + i}^\theta \right) + L \right| < M
			\end{equation*} 
		 with $ \frac{1}{2} y_i^\star P <\tupley \leq 2 y_i^\star P,$  where $L = L(\lambda_i, \theta, d, \epsilon, \tuplex)$ is a fixed real number, determined by the choice we made for the tuple $\tuplex.$ We write $ V_{A_\theta}^{(1)} (P)$ to denote the number of integer solutions of this inhomogeneous inequality. As a consequence of \cite[Theorem 1.2]{poulias_ineq_frac} one has 
			\begin{equation*}
					\displaystyle  V_{A_\theta}^{(1)} (P) \ll  P^{ A_\theta - \theta + \epsilon}.
			\end{equation*}
		Hence, we have showed that $ Z_1^\prime  (P) \ll P^{A_\theta + A_d - (\theta +d) + \epsilon}$ and in view of (\ref{Xfigj_bound_kZ1}) the proof of $(ii)$ is now complete.
		
		Similarly we argue for $(iii).$ Fix indices $i$ and $k.$ As before,  by Lemma \ref{lem3.3.2} one now has (up to constants which are independent of $\kappa$) that
			\begin{equation} \label{eq3.3.8}
			\Xi_{f_i, h_k} \ll \kappa Z_2 (P) \ll \kappa Z_2^\prime (P),
			\end{equation}
		where $Z_2^\prime (P)$ denotes the number of integer solutions of the system
			\begin{equation} \label{eq3.3.9}
				\begin{cases}
					\displaystyle \left| \lambda_i \sum_{ i=1}^{  \frac{A_\theta}{2}} \left(x_i^\theta - x_{\frac{A_\theta}{2} +i}^\theta \right) \right| < M \\[15pt]
					\displaystyle a_i \sum_{ i= 1}^{ \frac{A_\theta}{2}} \left(x_i^d - x_{\frac{A_\theta}{2} + i}^d \right) + b_k  \sum_{ i= 1 }^{ \frac{A_d}{2}} \left(z_i^d - z_{\frac{A_d}{2} +i}^d \right) =0,
				\end{cases}
			\end{equation}
		with $ \frac{1}{2} x_i^\star P < \tuplex \leq 2 x_i^\star P $ and $ \frac{1}{2} z_i^\star P <\tuplez \leq 2 z_i^\star P.$ We write $V_{A_\theta}^{(2)} (P)$ to denote the number of integer solutions of the inequality in (\ref{eq3.3.9}). By Lemma \ref{lem3.3.1} one has
			\begin{equation*}
				V_{A_\theta}^{(2)} (P) \ll  \int_{ \frac{- M | \lambda_i|}{2}}^{ \frac{M | \lambda_i|}{2}} \hspace{0.05in} \left| \sum_{ \frac{1}{2} y_i^\star P < y \leq 2y_i^\star P } e ( \alpha x^\theta) \right|^{A_\theta} \text d \alpha.
			\end{equation*}
		As in $(ii)$ we can show that for any fixed $ \epsilon > 0$ one has
			\begin{equation*}
					V_{A_\theta}^{(2)} (P) \ll P^{A_\theta - \theta + \epsilon}.
			\end{equation*}
		
		Fix a solution $ \tuplex$ counted by $ V_{A_\theta}^{(2)} (P).$ Substitute these values into the equation of system in (\ref{eq3.3.9}). Then the first block of variables becomes a fixed integer, say $C = C(\lambda_i, \theta, \epsilon, \tuplex),$ which depends on the choice we made for the tuple $ \tuplex.$ Hence, this equation takes the shape
			\begin{equation*}
				a_i C + b_k \sum_{ i= 1 }^{ \frac{A_d}{2}} \left(z_i^d - z_{\frac{A_d}{2} +i}^d \right) = 0.
			\end{equation*}
		Note that if $b_k$ does not divide the product $a_i C$, then the above equation is not soluble in integers. In such a case $Z_2(P) =0$ and the claimed estimate holds trivially. Hence, assuming that $b_k \mid (a_i C)$ we can rewrite it as
			\begin{equation*}
				 \sum_{ i= 1 }^{ \frac{A_d}{2}} \left(z_i^d - z_{\frac{A_d}{2} +i}^d \right) = C^\prime,
			\end{equation*}
		where $C^\prime  = C^\prime(\lambda_i, a_i, b_k,\theta, \epsilon, \tuplex)$ is a fixed integer determined by the choice we made for the tuple $ \tuplex.$ The number of integer solutions of this last equation is bounded above by the mean value 
			\begin{equation*}
				\int_0^1 \left| \sum_{  \frac{1}{2} z_i^\star P < z \leq 2 z_i^\star P} e (\alpha z^d) \right|^{A_d} e(- \alpha C^\prime) \text d \alpha.
			\end{equation*}
		
		Again note that
			\begin{equation*}
				\left| \sum_{  \frac{1}{2} z_i^\star P < z \leq 2 z_i^\star P} e (\alpha z^d) \right| \ll \left| h \left(\alpha ; 2z_i^\star P \right) \right| + \left| h \left(\alpha ; \frac{1}{2} z_i^\star P \right) \right|.
			\end{equation*}
		So, by the triangle inequality and invoking Lemma \ref{lem3.3.4} we deduce that
			\begin{equation*}
				\int_0^1 \left| \sum_{  \frac{1}{2} z_i^\star P < z \leq 2 z_i^\star P} e (\alpha z^d) \right|^{A_d} e(- \alpha C^\prime) \text d \alpha \ll P^{A_d - d +\epsilon}.
			\end{equation*}
		Hence, we deduce that $ Z_2^\prime  (P) \ll P^{A_\theta +A_d  (\theta +d) + \epsilon}.$ In view of (\ref{eq3.3.8}) the proof of the estimate $(iii)$ is now complete.
	\end{proof}



\section{Minor arcs analysis}

In this section we deal with the set of minor arcs $ \mathfrak p = \left( [0,1) \times  \mathfrak m \right) \cup \left( \mathfrak n_\xi \times \mathfrak M \right).$ Here we aim to show that for $ s_{ \min} \leq s \leq s_{\max}$ one has
	\begin{equation*}
		\int_{\mathfrak p} \left| \mathcal F (\tuplealpha) K_\pm(\alpha_\theta) \right| \text d \tuplealpha = o \left( P^{s - (\theta +d)} \right).
	\end{equation*}
For a better presentation of our approach we split the analysis into two parts, dealing separately with the sets $ [0,1) \times  \mathfrak m $ and $ \mathfrak n_\xi \times \mathfrak M .$

\subsection{Minor arcs: Part 1}

First we consider the case where $ ( \alpha_d, \alpha_\theta) \in  [0,1) \times \mathfrak m. $ Recall that the set $ \mathfrak m$  is given by
	\begin{equation*}
		\mathfrak m = \{ \alpha_\theta \in \mathbb R : P^{- \theta + \delta_0} \leq | \alpha_\theta| < P^{\omega }  \}.
	\end{equation*}
Define the intervals $ \mathfrak m^{+} = \left[ P^{- \theta +\delta_0}, P^{\omega } \right), \mathfrak m^{-} = \left( -P^{\omega}, - P^{- \theta +\delta_0} \right] $ and note that $ \mathfrak m = \mathfrak m^{+} \cup \mathfrak m^{-}.$ Recall (\ref{eq3.2.9}). Making a change of variables by
	\begin{equation} \label{eq3.4.1}
		\begin{pmatrix}
			\alpha_\theta \\
			\alpha_d
		\end{pmatrix}
		=
		\begin{pmatrix}
			-1 & 0 \\
			0 & -1 
		\end{pmatrix}
		\begin{pmatrix}
			\beta_\theta \\
			\beta_d
		\end{pmatrix}
		+
		\begin{pmatrix}
			0 \\
			1	
		\end{pmatrix},
	\end{equation}
and using the $1$-periodicity of the function $ \mathcal F (\tuplealpha)$ with respect to $ \alpha_d$ yields
	\begin{equation} \label{eq3.4.2}
		R_\pm \left( P ;  [0,1) \times \mathfrak m^{-} \right) = \overline{	R_\pm \left( P ; [0,1) \times  \mathfrak m^{+} \right)},
	\end{equation}
where $ \overline{	R_\pm \left( P ; [0,1) \times  \mathfrak m^{+} \right)}$ is the complex conjugate. Therefore, it suffices to deal with the set $ [0,1) \times \mathfrak m^{+} .$

Let $f$ be a real valued function defined on the natural numbers, and let $ h \in \mathbb N.$ Define the forward difference operator $ \Delta_hf $ via the relation
	\begin{equation*}
		\left( \Delta_h f  \right)(x) = f(x+h) - f(x).
	\end{equation*}
For a tuple $ \tupleh= (h_1, \ldots, h_t) \in \mathbb N^t$ we define the difference operator $  \Delta_{h_1, \ldots, h_t} = \Delta_{\tupleh}^{(t)}$ inductively by
	\begin{equation*}
		\Delta_{\tupleh}^{(t)} f(x) = \Delta_{h_t}\left(\Delta_{h_1, \ldots, h_{t-1}}f(x)\right).
	\end{equation*}			
It is apparent that the operator $ \Delta_h$ is a linear one. Namely, for constants $a,b,$ and two functions $f,g,$ one has
	\begin{equation*}
		\Delta_h \left( af + bg \right) = a \Delta_h f + b \Delta_h g.
	\end{equation*}

For $ d\geq 2$ one can inductively verify that 
	\begin{equation*}
		\Delta_{\tupleh}^{(d)} (x^d) = d!\  h_1 \cdots h_d.
	\end{equation*}
Next, we wish to obtain an analogous result for the $r$--th derivative of a monomial of fractional degree $ \theta.$ 

	\begin{lemma} \label{lem3.4.1}
		 Suppose that $ t \leq \lfloor \theta \rfloor $ is a natural number. Let $ \tupleh = (h_1, \ldots, h_t) \in \left( \mathbb N \cap [1,P] \right)^t$ and suppose that $ P < x \leq 2P.$ Then for each natural number $ r \geq 1 $ one has
			\begin{equation*}
				\left| \frac{\normalfont \text d^r}{ \normalfont  \text d x^r} \Delta_{\tupleh}^{(t)} (x^\theta) \right| \asymp  h_1 \cdots h_t P^{\theta -r- t}. 
			\end{equation*}
	\end{lemma}

	\begin{proof}
		Observe that if $ \phi : I \to \mathbb R $ is an $r$ times differentiable function defined on an interval $I$ and $h$ is a natural number, then one has for $x_0 \in I$ that
			\begin{equation*}
				\frac{ \text d^r}{ \text d x^r} \Delta_h \phi(x) \Bigr|_{x=x_0} = \frac{ \text d^r}{ \text d x^r}  \left( \phi (x+h) - \phi (x) \right)  \Bigr|_{x=x_0}  = \Delta_h \left( \frac{ \text d^r}{ \text d x^r} \phi (x)  \Bigr|_{x=x_0} \right).
			\end{equation*}
		From the inductively definition of the operator $ \Delta_{\tupleh}^{(t)} $ and iterating we obtain from the above observation that
			\begin{equation*}
				\frac{ \text d^r}{ \text d x^r} \left( \Delta_{\tupleh}^{(t)} (x^\theta) \right) \Bigr|_{x=x_0} =  \Delta_{\tupleh}^{(t)} \left( \frac{ \text d^r}{ \text d x^r} (x^\theta) \Bigr|_{x=x_0} \right) = C_r  \Delta_{\tupleh}^{(t)} (x_0^{\theta -r}),
			\end{equation*} 
		where $ C_r = \theta (\theta-1) \cdots (\theta -r +1).$ 
		
		From the above considerations follows that it suffices to show
			\begin{equation} \label{eq3.4.3}
				\left| \Delta_{\tupleh}^{(t)} (x^{\theta -r}) \right| \asymp h_1 \cdots h_t P^{\theta-r-t}.
			\end{equation}
		To this end, we use induction on the number of shifts $t$ and apply successively the mean value theorem of differential calculus. We show that one has
			\begin{equation*}
				\Delta_{\tupleh}^{(t)} (x^{\theta -r}) =  C_{r,t} h_1 \cdots h_t \xi_x^{\theta -r-t}, 
			\end{equation*}
		for some  $ \xi_x = \xi_{x,\tupleh}$  with $ x < \xi_x < x + h_1 + \cdots + h_t,$ where $C_{r,t} = (\theta-r)(\theta-r-1) \cdots (\theta-r-t+1).$ 
		
		Indeed, for $t=1 $ one has  
			\begin{equation*}
				\Delta_{h_1} (x^{\theta-r}) =  \left( (x+h_1)^{\theta-r} - x^{\theta-r} \right) = (\theta-r) h_1 \xi_x^{\theta-r-1},
			\end{equation*}
		for some $ \xi_x = \xi_{x, h_1}$ with $ x < \xi_x < x + h_1.$ Assume that the statement of the lemma holds for $t-1.$ We prove that it does hold for $t.$ By the definition of the forward difference operator one has
			\begin{equation*}
				\Delta_{\tupleh}^{(t)} (x^{\theta-r}) = \Delta_{h_t} \left(\Delta_{\tupleh^\prime}^{(t-1)} (x^{\theta-r}) \right),
			\end{equation*}
		where $\tupleh^\prime = (h_1, \ldots, h_{t-1}).$ By the inductive hypothesis one has
			\begin{equation*}
				\Delta_{\tupleh}^{(t-1)} (x^{\theta-r})  =  (\theta-r)  \cdots (\theta -r- t+2) h_1 \cdots h_{t-1} \zeta_x^{\theta -r- t +1},
			\end{equation*}		
		for some $ \zeta_x = \zeta_{x, \tupleh^\prime}$ with $ x < \zeta_x < x + h_1 + \cdots + h_{t-1}.$ We put $ f(\zeta_x) = \zeta_x^{\theta-r-t+1}$ and write  
			\begin{equation*}
				f^\prime (\zeta_x) = \frac{ \text d f(\zeta_x)}{ \text d \zeta_x}.
			\end{equation*}
		Clearly, $ f^\prime (\zeta_x) = (\theta -r- t+1) \zeta_x^{\theta -r-t}.$ One now has
			\begin{equation} \label{eq3.4.4}
				\Delta_{\tupleh}^{(t)} (x^{\theta-r}) =  (\theta-r)  \cdots (\theta -r- t+2)  h_1 \cdots h_{t-1} \left( f (\zeta_x+ h_t) -  f(\zeta_x) \right).
			\end{equation}
		To treat the expression in the parenthesis one can apply the mean value theorem of differential calculus to the function $ f.$ Hence one may write
			\begin{equation} \label{eq3.4.5}
				f (\zeta_x + h_t) -  f(\zeta_x) = (\theta- r-t+1) h_t \xi_x^{\theta-r-t},
			\end{equation}
		for some $ \xi_x = \xi_{x, \tupleh} $ with $ \zeta_x < \xi_x < \zeta_x + h_t.$ By the induction process it is apparent that one has $ x < \xi_x < x + h_1 + \cdots + h_t.$ It is apparent that whenever $ 1 \leq \tupleh \leq P$ and $ P < x \leq 2P$ one has $ \xi_x \asymp x \asymp P.$ Putting together (\ref{eq3.4.4}) and (\ref{eq3.4.5}) confirms (\ref{eq3.4.3}), and thus the proof of the lemma is complete.
	\end{proof}

In the analysis below we make use of Weyl's inequality arising from the differencing process.

	\begin{lemma}[Weyl's inequality] \label{lem3.4.2}
		Let $\phi (x)$ be a real valued function defined over the natural numbers. Let $ d \geq 2 $ be a natural number, and write $ D = 2^{d-1}.$ Then one has
			\begin{equation*}
				\left| \sum_{ 1 \leq x \leq X} e( \phi (x)) \right|^D \ll X^{D-1} + X^{D-d}  \left| \sum_{h_1 =1}^X \cdots \sum_{h_{d-1} =1}^X \sum_{ 1 \leq x < x+ Y_{d-1} \leq X} e \left( \Delta_{\tupleh}^{(d-1)} (\phi(x)) \right) \right|,
			\end{equation*}
		where $Y_j = h_1 + \cdots h_j,$ for each $j.$ The implied constant depends only on $d,$ and an empty sum denotes zero.
	\end{lemma}

	\begin{proof}
		See \cite[Lemma 3.8]{baker_book_dioph_ineq}.
	\end{proof}

 From now one we fix an index $i.$ By Lemma \ref{lem3.4.2}, and using the linearity of the forward difference operator one has
	\begin{equation*} \label{4_min_arc1_fi_diff_proc}
		\begin{split}	
			\left| 	F_i(\alpha_d, \alpha_\theta) \right|^{2^d} & \ll P^{2^d-1} + P^{2^d- (d+1)} \sum_{\tupleh} \left| \sum_{x } e \left( a_i \alpha_d  d !\ h_1 \cdots h_d + \lambda_i \alpha_\theta  \Delta_{\tupleh}^{(d)}( x^\theta ) \right) \right| \\[10pt]
			& \ll  P^{2^d-1} + P^{2^d- (d+1)} \sum_{\tupleh} \left| \sum_{x } e \left( \lambda_i \alpha_\theta  \Delta_{\tupleh}^{(d)}( x^\theta ) \right) \right|,
		\end{split}
	\end{equation*}
where in the second step we used the triangle inequality. In the above summation notation, we sum over tuples $ \tupleh$ satisfying $ 1 \leq \tupleh \leq P$ and $x$ belongs to a subinterval of $[1,P]$ determined by the shifts $h_1, \ldots, h_d.$ For convenience we denote this interval by $ I(\tupleh).$ 
	
We put
	\begin{equation} \label{eq3.4.6}
		S_i(\alpha_\theta, \tupleh ) = \sum_{ x \in I(\tupleh)}  e \left( \lambda_i \alpha_\theta  \Delta_{\tupleh}^{(d)}( x^\theta ) \right).
	\end{equation}
Hence, the above estimate now takes the shape
	\begin{equation} \label{eq3.4.7}
		\left| 	F_i(\alpha_d, \alpha_\theta) \right|^{2^d} \ll P^{2^d-1} + P^{2^d- (d+1)} \sum_{\tupleh}  \left| S_i(\alpha_\theta, \tupleh)  \right|.
	\end{equation}

One can split the summation over $ \tupleh$ based on the size of the product $ H= h_1 \cdots h_d.$ Consider the function $ \psi (P) = (\log P)^{-1}$ which decreases  monotonically to zero as $ P \to \infty$ and furthermore for large $P$ satisfies $ \psi(P) > P^{-\epsilon}$ for any fixed $ \epsilon > 0.$ We form a partition of the shape
	\begin{equation*}
			\displaystyle \left\{ (h_1, \ldots, h_d) : h_i \in [1,P] \cap \mathbb Z  \right\} = A_1 \cup A_2 \cup A_3,
	\end{equation*}
where we define the sets $A_1, A_2$ and $ A_3$ by
	\begin{equation*}
		\begin{split}
			& \displaystyle  A_1 = \left\{ (h_1, \ldots, h_d) : h_i \in [1,P] \cap \mathbb Z, \hspace{0.05in} P^d \psi (P) < H \leq  P^d  \right\}, \\[10pt]	
			& \displaystyle  A_2 = \left\{ (h_1, \ldots, h_d) : h_i \in [1,P] \cap \mathbb Z, \hspace{0.05in} P^{d-5^{-\theta}} < H \leq P^d \psi(P) \right\}, 	\\[10pt]
			& \displaystyle  A_3 = \left\{ (h_1, \ldots, h_d) : h_i \in [1,P] \cap \mathbb Z, \hspace{0.05in}  H \leq P^{d - 5^{-\theta}}  \right\}.
		\end{split}
	\end{equation*}
Moreover, for $ \kappa = 1,2,3 $ we define
	\begin{equation} \label{eq3.4.8}
		T_{\kappa} (\alpha_\theta) = \sum_{ \tupleh \in A_{\kappa} }  \left| S_i(\alpha_\theta, \tupleh)  \right|.
	\end{equation}
To avoid confusion in the following, let us observe that in order to reduce the notation, in the definition of $T_\kappa (\alpha_\theta)$ we omit the dependence on $i.$ One can now write
	\begin{equation*}
		\begin{split}	
			\displaystyle \sum_{\tupleh} \left| S_i(\alpha_\theta, \tupleh ) \right| \ll T_1 (\alpha_\theta) + T_2 (\alpha_\theta) + T_3 (\alpha_\theta).
		\end{split}
	\end{equation*}
Invoking (\ref{eq3.4.7}) we deduce that
	\begin{equation} \label{eq3.4.9}
		\displaystyle  \left| 	F_i(\alpha_d, \alpha_\theta) \right|^{2^d} \ll P^{2^d -1} + P^{2^d -(d+1)} \left( T_1 (\alpha_\theta) + T_2 (\alpha_\theta) + T_3 (\alpha_\theta) \right).
	\end{equation}

Our aim now is to obtain a non-trivial upper bound for the exponential sum $ S_i (\alpha_\theta) $ with $ \alpha_\theta \in \mathfrak m^{+}.$ To do so, we make use of van der Corput's $k$-th derivative test for bounding exponential sums. 
	
	\begin{lemma} \label{lem3.4.3}
		Let $q \geq 0$ be an integer. Suppose that $ f : (X, 2X] \to \mathbb R$ is a function having continuous derivatives up to the $(q+2)$-th order in $(X,2X].$ Suppose also there is some $F > 0,$  such that for all $ x \in (X, 2X]$ we have
			\begin{equation} \label{eq3.4.10}
				F X^{-r} \ll | f^{(r)} (x) | \ll FX^{-r},
			\end{equation}
		for $r =1, 2, \ldots, q+2.$ Then we have
			\begin{equation*}
			\sum_{X < x \leq 2X} e(f(x)) \ll F^{1/ (2^{q+2}-2)} X^{1 - (q+2)/(2^{q+2}-2)} + F^{-1}X,
		\end{equation*}
		with the implied constant depending only upon the implied constants in (\ref{eq3.4.10}).
	\end{lemma}
	
	\begin{proof}
		See \cite[Theorem 2.9]{graham_kolesnik_book}.
	\end{proof}

We now make some observations that set the ground for an application of Lemma \ref{lem3.4.3}. It is convenient to work with an exponential sum over a dyadic interval. Recall from (\ref{eq3.4.6}) that
	\begin{equation*}
			S_i(\alpha_\theta, \tupleh ) = \sum_{ x \in I(\tupleh)}  e \left( \lambda_i \alpha_\theta  \Delta_{\tupleh}^{(d)}( x^\theta ) \right).
	\end{equation*}
One can split the interval $ I(\tupleh)$ into $ O \left( \log P \right) $ dyadic intervals. By making abuse of notation one then has
	\begin{equation*}
		 \left| S_i (\alpha_\theta, \tupleh) \right| \ll \log P \sum_{ P < x \leq 2P} e \left( \lambda_i \alpha_\theta  \Delta_{\tupleh}^{(d)}( x^\theta ) \right).
	\end{equation*}
Put
	\begin{equation*}
		 \widetilde S_i (\alpha_\theta, \tupleh) =  \sum_{ P < x \leq 2P} e \left( \lambda_i \alpha_\theta  \Delta_{\tupleh}^{(d)}( x^\theta ) \right).
	\end{equation*}
Hence for all $ \alpha_\theta$ and for any fixed $ \epsilon > 0$ one has
	\begin{equation} \label{eq3.4.11}
		\left| S_i (\alpha_\theta, \tupleh) \right| \ll P^\epsilon \left| \widetilde S_i (\alpha_\theta, \tupleh) \right|.
	\end{equation}	
It is apparent that an upper bound for the exponential sum $ \widetilde S_i (\alpha_\theta )$ leads to an upper bound for the exponential sum $S_i (\alpha_\theta )$ with an $\epsilon$- loss. This is enough for our purpose. Observe that invoking Lemma \ref{lem3.4.1} with $t=d$ one has for each natural number $ r \geq 1 $ that
	\begin{equation*}
		\left| \frac{ \text d^r}{ \text d x^r} \left( \lambda_i \alpha_\theta  \Delta_{\tupleh}^{(d)}( x^\theta )  \right) \right| \asymp FP^{-r},
	\end{equation*}
where $ F = |\lambda_i C_r C_{r,d}| | \alpha_\theta|  H P^{\theta -d}.$  Recall here that $\mathfrak m^{+} = \left[ P^{-\theta + \delta_0}, P^\omega \right).$ 

	\begin{lemma} \label{lem3.4.4}
		Suppose that $ P^{d - 5^{-\theta}} < H \leq P^d.$ For each index $ i$ and for any $ \alpha_\theta \in \mathfrak m^{+}$ one has that 
			\begin{equation*}
				 \left| S_i (\alpha_\theta, \tupleh) \right| \ll P^{1 - 4^{-\theta}}.
			\end{equation*}
	\end{lemma}

	\begin{proof}
		 Note that it is enough to show that for all $ \alpha_\theta \in \mathfrak m^{+}$ one has
		 	\begin{equation*}
		 		\left| \widetilde S_i (\alpha_\theta, \tupleh) \right| \ll P^{1 - \sigma},
		 	\end{equation*}
		 for some $ \sigma > 4^{-\theta}.$ Then returning in (\ref{eq3.4.11}) and taking $ \epsilon = \sigma - 4^{-\theta}>0$ as we are at liberty to do, yields the desired conclusion. We consider two separate cases depending on the size of $H.$
		 
		 Suppose first that $ P^d \psi (P) < H \leq P^d.$ Then one has
			\begin{equation*}
				P^{\delta_0} \psi (P) \ll F \ll P^{\theta + \omega}.
			\end{equation*}
		We may now apply Lemma \ref{lem3.4.3} with $ q = n,$ where temporarily we write $ n =\lfloor \theta \rfloor .$ This reveals that for any $ \alpha_\theta \in \mathfrak m^{+}$ one has
			\begin{equation*}
				\left| \widetilde S_i (\alpha_\theta, \tupleh)  \right|  \ll P^{ 1 - \sigma} + P^{1 - \delta_0} \psi(P)^{-1},
			\end{equation*} 
		where 
			\begin{equation} \label{eq3.4.12}
				\sigma = \frac{n +2 - \theta - \omega}{2^{n+2}-2}.
			\end{equation}
	 Recalling (\ref{eq3.2.7}) one can verify that for $ \theta > d +1 \geq 3$ one has
	 	\begin{equation*}
	 		\sigma > \frac{1}{3^\theta + 6} > \frac{1}{4^{\theta}}.
	 	\end{equation*}
	 Moreover, recalling that $ \psi (P) = (\log P)^{-1}$ one has $ \psi(P)^{-1} \ll P^{10^{-\theta}},$ which yields
	 	\begin{equation*}
	 		P^{1 - \delta_0} \psi(P)^{-1} \ll P^{ 1 - \delta_0 + 10^{-\theta} }.
	 	\end{equation*}
	Hence, the previous estimate for the exponential sum $ \widetilde S_i (\alpha_\theta )$ delivers
		\begin{equation*}
				\left| \widetilde S_i (\alpha_\theta, \tupleh )  \right| \ll P^{1 - \sigma^\prime},
		\end{equation*}	
	where $ \sigma^\prime = \min \{ \sigma, \hspace{0.02in} \delta_0 - 10^{-\theta} \} > 4^{-\theta}$ and we are done.
	
	Suppose now that $ P^{d- 5^{-\theta}} < H \leq P^d \psi (P). $ In this case one has
			\begin{equation*}
				P^{\delta_0 - 5^{-\theta}} \ll F \ll P^{\theta + \omega} \psi (P).
			\end{equation*}
		Applying again Lemma \ref{lem3.4.3} with $ q = n,$ yields that for any $ \alpha_\theta \in \mathfrak m^{+}$  one has
			\begin{equation*}
				\left| \widetilde S_i (\alpha_\theta, \tupleh )  \right|  \ll P^{1 - \sigma } \left( \psi (P) \right)^{1/ (2^{n+2}-2)} + P^{1 - \delta_0 + 5^{-\theta}},
			\end{equation*}
		with $\sigma$ as in (\ref{eq3.4.12}). For large $P$ one may assume that $ \psi (P) < 1.$ Recalling again from (\ref{eq3.2.7}) that $ \delta_0 = 2^{1- 2\theta}$ the above estimate delivers 
			\begin{equation*}
				\left| \widetilde S_i (\alpha_\theta, \tupleh )  \right|  \ll P^{1 - \sigma^\prime},
			\end{equation*}
		where now we write $ \sigma^\prime = \min  \{ \sigma , \hspace{0.02in} \delta_0 - 5^{-\theta} \} > 4^{-\theta}.$ Thus the proof is now complete.
	\end{proof}

We can now estimate the sums $T_\kappa (\alpha_\theta) \hspace{0.05in} (1 \leq \kappa \leq 3)$ defined in (\ref{eq3.4.8}).

	\begin{lemma} \label{lem3.4.5}
		For each index $i$ and for any $ \alpha_\theta \in \mathfrak m^{+}$ one has that
			\begin{itemize}
				\item[(i)] $ \left| T_1 (\alpha_\theta) \right| \ll P^{d+1 - 5^{-\theta}}$; 
				\item [(ii)] $ \left| T_2(\alpha_\theta) \right| \ll  P^{ d + 1 - 5^{-\theta}} \psi(P)$;
				\item [(iii)] $ \left| T_3(\alpha_\theta) \right| \ll  P^{ d +1- 6^{-\theta}} .$
			\end{itemize}
	\end{lemma}

	\begin{proof}
			For each $\kappa =1,2,3$ we write $ \# A_\kappa$ to denote the cardinality of the set $ A_\kappa.$ We set
				\begin{equation*}
					X_1= P^d, \hspace{0.3in} X_2 = P^d \psi (P), \hspace{0.3in} X_3 = P^{d - 5^{-\theta}}.
				\end{equation*}
			Observe that for each $\kappa =1,2,3$ and for any fixed $ \epsilon> 0$ one has
				\begin{equation*}
					\# A_\kappa \ll \sum_{ H \leq X_\kappa} \tau_d (H) \ll X_\kappa P^\epsilon,
				\end{equation*}
			where recall that $ H = h_1 \cdots h_d$ and $ \tau_d$ is the $d$-fold divisor function.
			
			One can get an upper bound for each $T_\kappa (\alpha_\theta)$ by using the above observation together with the bound supplied by Lemma \ref{lem3.4.4}. Let us demonstrate this by proving estimate $(i).$ Recall here that
				\begin{equation*}
					T_1 (\alpha_\theta) = \sum_{ \tupleh \in A_1} \left| S_i(\alpha_\theta) \right|,
				\end{equation*}
			where 
				\begin{equation*} 
					A_1 = \{ (h_1, \ldots, h_d) : h_i \in [1,P] \cap \mathbb Z, \hspace{0.05in} P^d \psi (P) < H \leq P^d \}.
				\end{equation*}
			Invoking Lemma \ref{lem3.4.4} one has for any $ \alpha_\theta \in \mathfrak m^{+}$ and any fixed $ \epsilon > 0$ that
				\begin{equation*}
				 	\left| T_1(\alpha_\theta) \right| \ll \left( \sup_{\substack{ \alpha_\theta \in \mathfrak m^{+} \\ \tupleh \in A_1 } } \left| S_i (\alpha_\theta, \tupleh) \right| \right) \sum_{ \tupleh \in A_1} 1 \ll P^{1-4^{-\theta}} \left(\#A_1\right) \ll  P^{d+1 - 4^{-\theta} +\epsilon }.
				\end{equation*}
			Pick now a sufficiently small $ 0< \epsilon < 4^{-\theta} - 5^{-\theta}$ to deduce that for any $ \alpha_\theta \in \mathfrak m^{+}$ one has
				\begin{equation*}
					\left| T_1(\alpha_\theta) \right| \ll P^{d+1 - 5^{-\theta}}.
				\end{equation*}	
		
		Similarly we argue to estimate the sums $T_2 (\alpha_\theta)$ and $T_3(\alpha_\theta).$ For the sake of clarity, let us mention that in estimating $T_3 (\alpha_\theta)$ one can use the trivial bound 
			\begin{equation*}
				 \left| S_i (\alpha_\theta, \tupleh)  \right| \ll P,
			\end{equation*}
		which is always valid. With this observation the proof of the lemma is now complete.
	\end{proof}

By Lemma \ref{lem3.4.5} it is apparent that for each index $i$ and for any $ \alpha_\theta \in \mathfrak m^{+}$ one has
	\begin{equation*}
		\left| T_\kappa (\alpha_\theta) \right| \ll P^{d+1 - 6^{-\theta}} \hspace{0.5in} (\kappa =1,2,3).
	\end{equation*}
One can now use the above estimate in order to bound from above the right hand side of (\ref{eq3.4.9}). Hence we deduce that
	\begin{equation*}
		\left| F_i (\alpha_d, \alpha_\theta) \right| \ll P^{1 - 1/2^d} + P^{1 - 1/ (2^d \cdot 6^\theta)} \ll P^{1 - 6^{-\theta-d}}.
	\end{equation*}
Upon recalling (\ref{eq3.2.5}) we have proved the following.  

	\begin{lemma} \label{lem3.4.6}
		For each index $i$ and for any $ (\alpha_d, \alpha_\theta) \in [0,1) \times \mathfrak m^{+} $ one has that
			\begin{equation*}
				\left| f_i (\alpha_d, \alpha_\theta) \right| \ll P^{1 - 6^{-\theta-d} }.
			\end{equation*}
	\end{lemma} 

Equipped with all the necessary auxiliary estimates, we may now finish up the first part of the minor arcs analysis. We now set
	\begin{equation*}
		\eta_1 = 6^{-\theta-d} \hspace{0.5in} \text{and} \hspace{0.5in} \kappa = P^\omega.
	\end{equation*}
Note that for large enough $P$ one has $ \min_{i,j} \{\kappa | \lambda_i|, \hspace{0.05in} \kappa | \mu_j| \} \geq 1.$ Recall from (\ref{eq3.2.11}) that one has $ s^\prime  = s - \delta$ and recall as well from (\ref{eq3.2.15}) that $ s^\prime = A_\theta + (1- \omega_1)A_d.$ One can now use  Lemma \ref{lem3.4.6} and Lemma \ref{lem3.3.6} in order to estimate the right hand side of the inequality (\ref{eq3.2.14}). Hence, we infer that for any fixed $ \epsilon > 0$ one has
	\begin{equation*}
		\int_{\mathfrak m^{+}} \int_0^1 \left| \mathcal F(\tuplealpha) K_\pm (\alpha_\theta) \right| \text d \tuplealpha \ll P^{(1-\eta_1 ) \delta} P^{ s^\prime -  (\theta+ d) + \omega + \epsilon} \ll P^{ s - (\theta +d) - \eta_1 \delta + \omega + \epsilon}.
	\end{equation*}
Recall from( \ref{eq3.2.7}) that $ \omega \leq 5^{-100(\theta+d)}.$ One may choose 
	\begin{equation*}
		\delta = 6^{-\theta} \in (0,1/3) \hspace{0.3in} \text{and} \hspace{0.3in} \epsilon = 5^{-100(\theta+d)} ,
	\end{equation*}
as we are at liberty to do. With these choices for $ \delta$ and $ \epsilon$ it is clear that $ - \eta_1 \delta + \omega +  \epsilon < 0.$ Thus the above estimate delivers
	\begin{equation*}
			\int_{\mathfrak m^{+}} \int_0^1 \left| \mathcal F(\tuplealpha) K_\pm (\alpha_\theta) \right| \text d \tuplealpha  = o \left( P^{s - (\theta + d)} \right).
	\end{equation*}
In the light of (\ref{eq3.4.2}) we have established the following.

	\begin{lemma} \label{lem3.4.7}
		For  $ s_{\normalfont \min} \leq s \leq s_{\normalfont \max}$ one has
			\begin{equation*}
				\int_{\mathfrak m} \int_0^1 \left| \mathcal F (\tuplealpha) K_\pm (\alpha_\theta) \right| \normalfont \text d \tuplealpha = o \left( P^{s - (\theta +d)} \right).
			\end{equation*}
	\end{lemma}



\subsection{Minor arcs: Part 2} 

In this subsection we consider the case where $ (\alpha_d, \alpha_\theta) \in \mathfrak n_\xi \times \mathfrak M .$ Let us recall here that $ \mathfrak n_\xi \subset [0,1)$ is a set of minor arcs in the classical sense and $ \mathfrak M = \left( - P^{-\theta + \delta_0}, P^{-\theta + \delta_0} \right).$ We put $ \mathfrak M^{+} = \left( 0, P^{-\theta + \delta_0} \right)$ and $ \mathfrak M^{-} = \left( - P^{-\theta + \delta_0}, 0 \right).$ Note that $ \mathfrak M = \mathfrak M^{+} \cup \mathfrak M^{-}.$  Making a change of variables as in (\ref{eq3.4.1}) yields
	\begin{equation} \label{eq3.4.13}
		R_\pm \left( P ; \mathfrak n_\xi \times \mathfrak M^{-} \right) = \overline{ R_\pm \left( P ;  \mathfrak n_\xi  \times  \mathfrak M^{+} \right) }.
	\end{equation}
So in the following it suffices to deal with the set $ \mathfrak n_\xi \times  \mathfrak M^{+} .$ The point of departure in our approach is the following version of the Weyl - van der Corput inequality. 

	\begin{lemma}[Weyl--van der Corput inequality] \label{lem3.4.8}
		Suppose that $I$ is a finite subset of $ \mathbb N,$ and suppose that $ (w(n))_{n \in \mathbb N} \subset \mathbb C$ is a complex-valued sequence, such that $ w(n) = 0 \hspace{0.05in} $ for $ n \notin I.$ Let $ H$ be a positive integer. Then one has,
			\begin{equation*}
				\left| \sum_{n \in \mathbb N} w(n) \right|^2 \leq \frac{ \normalfont\text{card}(I)+H}{H} \sum_{|h|<H} \left( 1 - \frac{|h|}{H} \right) \sum_{ n \in \mathbb N} w(n) \overline{w(n-h)}.
			\end{equation*}
	\end{lemma}

	\begin{proof}
		See \cite[Lemma 2.5]{graham_kolesnik_book}.
	\end{proof}

To begin with, let us fix an index $i.$ Apply Lemma \ref{lem3.4.8} to the exponential sum $ F_i (\alpha_d, \alpha_\theta) ,$ with $ I = [1,P] \cap \mathbb N.$  For an integer $ H \in [1,P]$ with $ H = o (P) $ to be chosen at a later stage one has 
	\begin{equation} \label{eq3.4.14}
		\begin{split} 	
			\left| F_i (\alpha_d, \alpha_\theta) \right|^2 \ll \frac{P+H}{H} \sum_{ |h| <H}  \sum_{ 1 \leq x \leq P } e \left( a_i \alpha_d \Delta_h \left(x^d \right) +  \lambda_i \alpha_\theta \Delta_h \left(x^\theta \right)  \right).
		\end{split}	
	\end{equation}
By the mean value theorem of differential calculus one has that
	\begin{equation*}
		| (x+h)^\theta - x^\theta| \asymp | h| P^{\theta -1} \ll HP^{\theta -1}.
	\end{equation*}
For $ \alpha_\theta \in \mathfrak M^{+}$ the above estimate leads to
	\begin{equation*}
		| \alpha_\theta | | (x+h)^\theta - x^\theta| \ll P^{-1 + \delta_0} H . 
	\end{equation*}

Using the elementary inequality $ | e(x) | \leq 2 \pi |x|$ which is valid for all $ x \in \mathbb R,$ we infer that for any $ \alpha_\theta \in \mathfrak M^{+}$ one has
	\begin{equation*}
			\left| e \left(  \lambda_i \alpha_\theta \Delta_h \left(x^\theta \right)  \right) \right| \ll P^{-1 + \delta_0} H .
	\end{equation*} 
One may now use the fact that $ | e(x) | \leq 1$ for all $ x \in \mathbb R,$ together with the above estimate to derive that
	\begin{equation*}
			\sum_{ | h| < H} \sum_{ 1 \leq x \leq  P }  e \left( a_i \alpha_d \Delta_h \left(x^d \right) +  \lambda_i \alpha_\theta \Delta_h \left(x^\theta \right)  \right) = \sum_{ | h| < H} \sum_{ 1 \leq x \leq  P }  e \left( a_i \alpha_d \Delta_h \left(x^d \right)  \right) + O \left( P^{\delta_0} H^2 \right).
	\end{equation*}
Substituting the above conclusion into (\ref{eq3.4.14}) and using the fact that $ H = o (P) $ yields
	\begin{equation} \label{eq3.4.15}
		\left| F_i (\alpha_d, \alpha_\theta) \right|^2 \ll P^{1 + \delta_0} H + \frac{P+H}{H} \sum_{ |h| <H} 	 \left| W_i (\alpha_d ,h ) \right|,	 
	\end{equation}	
where we write
	\begin{equation*}
		 W_i (\alpha_d ,h ) = \sum_{  1 \leq x \leq  P} e \left( a_i \alpha_d \Delta_h \left(x^d \right) \right).
	\end{equation*}
We now examine separately the cases $ d \geq 3 $ and $ d=2.$

First we consider the case $ d \geq 3.$ An application of H\"{o}lder's inequality reveals
	\begin{equation} \label{eq3.4.16}
		\left( \sum_{ |h| <H} 	| W_i (\alpha_d, h)| \right)^{2^{d-2}} \ll H^{2^{d-2}-1}  \sum_{ |h| <H} \left| W_i (\alpha_d , h)  \right|^{2^{d-2}}.
	\end{equation} 
Applying Weyl's differencing process, we infer by Lemma \ref{lem3.4.2} that
	\begin{equation*} 
		\left |W_i ( \alpha_d , h) \right|^{2^{d-2}} \ll P^{2^{d-2}-1} + P^{2^{d-2}-(d-1)} \sum_{\tupleh} \left| \sum_{x \in I(\tupleh)} e \left(  d !\ hh_1 \cdots h_{d-2} a_i \alpha_d x \right) \right|,
	\end{equation*}
where in the above summation notation, we sum over tuples $ \tupleh = (h_1, \ldots, h_{d-2})$ satisfying $ 1 \leq \tupleh \leq P$ and $ I(\tupleh)$ is a subinterval of $ [1,P],$ determined by the shifts $h_1, \ldots, h_{d-2}.$ 

Invoking a classical estimate for the sum of the geometric series we see that
	\begin{equation*}
		\left| \sum_{x \in I(\tupleh)} e \left(  d !\ hh_1 \cdots h_{d-2}  \alpha_d a_i x \right) \right| \ll \min \left\{ P, \hspace{0.02in} \|  d !\ hh_1 \cdots h_{d-2} a_i  \alpha_d  \|^{-1} \right\}.
	\end{equation*}
Hence by the preceding estimate concerning $ W_i(\alpha_d, h)$ we deduce that 
	\begin{equation*} 
		\begin{split}	
			\sum_{|h| < H} \left | W_i ( \alpha_d, h) \right|^{2^{d-2}} \ll & H P^{2^{d-2}-1} +  P^{2^{d-2}-(d-1)} \hspace{0.05in} \times \\[10pt]
			& \times \sum_{h =1 }^H \sum_{h_1 = 1}^P \cdots \sum_{h_{d-2} =1}^P \min \left\{ P, \|  d !\ hh_1 \cdots h_{d-2}  a_i \alpha_d \|^{-1} \right\}.
		\end{split}
	\end{equation*}
We write $ d !\ |a_i|  h h_1 \cdots h_{d-2} = m.$ Note that for $ 1 \leq h \leq H$ and for $\tupleh = (h_1, \ldots, h_{d-2})$ with $ 1 \leq \tupleh \leq P$  one has that $ m \in \mathbb Z \cap [1, \hspace{0.05in} d !\ |a_i| H P^{d-2}].$ Clearly, the number of solutions of the previous equation with respect to $m$ is $ \leq \tau_{d-1} (m) \ll_{d, a_i} m^\epsilon.$  Thus, for any fixed $ 0 < \epsilon< 1$ we  obtain 
	\begin{equation} \label{eq3.4.17}
		\sum_{|h| < H} \left | W_i ( \alpha_d, h) \right|^{2^{d-2}} \ll  H P^{2^{d-2}-1} + P^{2^{d-2} - (d-1) + \epsilon} \sum_{ m=1}^{ d !\ |a_i| HP^{d-2}} \min\left\{ P, \| m  \alpha_d  \|^{-1} \right\}.
	\end{equation}
We bound the sum on the right hand side of the above estimate using the following.

	\begin{lemma} \label{lem3.4.9}
		Suppose that $ \alpha, \beta $ are real numbers and suppose further that  $\left| \alpha - a/q \right| \leq 1/q^2,$ where $(a,q) =1.$ Then 
			\begin{equation*}
				\sum_{ z=1}^R \min \left\{ N, \| \alpha z + \beta \| \right\} \ll \left( N + q \log q \right) \left( \frac{R}{q} + 1 \right).
			\end{equation*}
	\end{lemma}

	\begin{proof}
		See \cite[Lemma 3.2]{baker_book_dioph_ineq}. For the sake of clarity we remark here that in the statement Baker is imposing a strict inequality, namely  $\left| \alpha - a/q \right| < 1/q^2.$ However it is apparent from the proof that this is unnecessary.
	\end{proof}

By Dirichlet's  theorem on Diophantine approximation, there exist $ a \in \mathbb Z$ and $ q \in \mathbb N$ which satisfy $(a,q)=1, \hspace{0.05in} 1 \leq q \leq HP^{d-1 - \xi}$ and
	\begin{equation*} 
		\displaystyle  \left| a_i \alpha_d - \frac{a}{q} \right| \leq \frac{1}{q  HP^{d-1 - \xi}}.
	\end{equation*}
We pause for a moment to reflect on the fact that $ \alpha_d \in \mathfrak n_\xi.$ Recall that we assume $ H= o(P).$ So for large enough $P$ one has $ HP^{d-1 -  \xi} < P^{d - \xi}.$ So if it was $ 1 \leq q \leq P^\xi,$ then $\alpha_d$ would belong to the set of major arcs $ \mathfrak N_\xi.$ Thus, we may suppose that $ q > P^\xi.$ Hence
	\begin{equation} \label{eq3.4.18}
		 P^\xi < q \leq HP^{d-1 -\xi}.
	\end{equation}

One can now apply Lemma \ref{lem3.4.9}. For any fixed $ 0< \epsilon < 1$ one has 
	\begin{equation*}
		\begin{split}	
			\sum_{ m=1}^{ d !\ |a_i| HP^{d-2}} \min\left\{ P, \| m  \alpha_d  \|^{-1} \right\} & \ll \left( P + q \log q \right) \left( \frac{ d !\ | a_i| H P^{d-2}}{q} + 1 \right) \\[10pt]
			& \ll HP^{d-1 + \epsilon} \left( \frac{1}{q} + \frac{1}{P} + \frac{q}{ HP^{d-1}} \right),
		\end{split}
	\end{equation*} 
where in the second step estimate, we used the facts that $ \log q \ll P^\epsilon,$ and that for $ d \geq 3 $ one has $ HP^{d-2} \log q \gg P.$ By (\ref{eq3.4.18}) one has
	\begin{equation*}
		\frac{1}{q} + \frac{1}{P} + \frac{q}{ HP^{d-1}}  \ll P^{-\xi}.
	\end{equation*}
Thus, the previous estimate delivers
	\begin{equation*}
		\sum_{ m=1}^{ d !\ |a_i| HP^{d-2}} \min\left\{ P, \| m \alpha_d  \|^{-1} \right\}  \ll  HP^{d-1 - \xi+ \epsilon}.
	\end{equation*}

Using the above bound, one can now estimate the right hand side of (\ref{eq3.4.17}) to obtain
	\begin{equation*}
			\sum_{|h| < H} \left | W_i ( \alpha_d, h) \right|^{2^{d-2}} \ll  H P^{2^{d-2}-1} + H P^{2^{d-2} - \xi + \epsilon}.
	\end{equation*}
Invoking (\ref{eq3.4.16}) the previous estimate implies
	\begin{equation*}
			\sum_{|h| < H} \left | W_i ( \alpha_d, h) \right|  \ll H P^{1 - \xi/ 2^{d-2} + \epsilon}.
	\end{equation*}
Incorporating the above into (\ref{eq3.4.15}) and using the fact that $ H = o (P),$ yields that for any $  \alpha_d \in \mathfrak n_\xi $ one has
	\begin{equation} \label{eq3.4.19}
		\left| F_i (\alpha_d, \alpha_\theta) \right|^2 \ll P^{1 + \delta_0} H +  P^{2 - \xi/ 2^{d-2} + \epsilon}.
	\end{equation}

We now deal with the case where $ d=2.$ In this case one does not have to apply Weyl's differencing process. Note that for $ d=2$ the difference $ \Delta_h (x^2) = 2xh + h^2$ is already a linear polynomial with respect to $x.$ So one has
	\begin{equation*}	
			\left| W_i (\alpha_d, h) \right|  \leq  \left| \sum_{ 1 \leq x \leq P} e \left( 2ha_i \alpha_d x\right) \right| \ll \min\left\{ P, \hspace{0.01in} \| 2ha_i \alpha_d \|^{-1} \right\}.
	\end{equation*}
Thus, 
	\begin{equation*}
		\sum_{ | h| < H} \left| W_i (\alpha_d, h) \right|  \ll \sum_{m=1}^{2|a_i| H}  \min\left\{ P, \hspace{0.01in} \| m \alpha_d \|^{-1} \right\}.
	\end{equation*}

One can now apply Dirichlet's theorem on Diophantine approximation and argue as in the case $ d \geq 3.$ Here the inequality (\ref{eq3.4.18}) is replaced by $  P^\xi < q \leq HP^{1 - \xi}.$ Applying Lemma \ref{lem3.4.9} one has
	\begin{equation*}
		\begin{split}	
			\sum_{m=1}^{2|a_i| H}  \min\left\{ P, \hspace{0.01in} \| 2ha_i \alpha_d \|^{-1} \right\} & \ll \left( P + q \log q \right) \left( \frac{2 |a_i| H}{q} +1 \right) \\[10pt]
			& \ll HP^{1  - \xi + \epsilon} + P^{1 + \epsilon}, 
		\end{split}
	\end{equation*}
where in the second step estimate we used the facts that $ P \gg H \log q $ and $ H \ll HP.$ Therefore, by (\ref{eq3.4.15}) we infer that
	\begin{equation} \label{eq3.4.20}
		\left| F_i (\alpha_d, \alpha_\theta) \right|^2 \ll P^{1 + \delta_0} H + P^{2 -\xi + \epsilon} + P^{2 + \epsilon} H^{-1}.
	\end{equation}

We can now obtain a non-trivial upper bound for the exponential sum $ F_i (\alpha_d, \alpha_\theta) .$  Recall that $H  \in [1,P]$ is an integer at our disposal which satisfies $ H = o (P).$ Let us now choose a value for $H$ so that $ H \asymp P^{\varpi}$ where $ \varpi = (1 - \delta_0)/2 \in (0,1).$ First we deal with the case $ d \geq 3. $ Recall from (\ref{eq3.2.7}) that $ \delta_0 = 2^{1-2\theta}$ and recall from (\ref{eq3.2.8}) that $ 0 < \xi \leq \delta_0/8.$ By (\ref{eq3.4.19}) we deduce that for any fixed $ 0 < \epsilon < 1$ and any $ \alpha_d \in \mathfrak n_\xi$ one has that
	\begin{equation*}
		\left| F_i (\alpha_d, \alpha_\theta) \right| \ll P^{1- \xi/ 2^{d-3} + \epsilon}.
	\end{equation*} 
Now we come to the case $ d =2.$ With the above choice for the integer parameter $H$ we infer by (\ref{eq3.4.20}) that for any fixed  $ 0 < \epsilon < 1 $ and any $ \alpha_d \in \mathfrak n_\xi$ one has
	\begin{equation*}
		\left| F_i (\alpha_d, \alpha_\theta) \right| \ll P^{1 - \xi/2 + \epsilon}.
	\end{equation*}

By the preceding conclusions and recalling (\ref{eq3.2.5}) we have proved the following.

	\begin{lemma} \label{lem3.4.10}
		For each index $i$ and for any $( \alpha_d, \alpha_\theta) \in \mathfrak n_\xi \times   \mathfrak M,$ one has for any fixed $ 0 < \epsilon < 1$ that 
			\begin{equation*}
				\left| f_i (\alpha_d, \alpha_\theta) \right| \ll
					\begin{cases}
						P^{1 - \xi/2 + \epsilon}, & \hspace{0.1in} \text{when} \hspace{0.1in} d=2, \\[10pt]
						P^{1 - \xi/ 2^{d-3} + \epsilon }, & \hspace{0.1in}  \text{when} \hspace{0.1in} d \geq 3.
					\end{cases}
			\end{equation*}
	\end{lemma}

We may now finish our analysis as in Part 1 of the minor arcs treatment. Below we demonstrate how to deal with the case $ d \geq 3.$ One can argue similarly when $ d =2.$ Put 
	\begin{equation*} 
		\eta_2 = \xi / 2^{d-3} \hspace{0.3in} \text{and} \hspace{0.3in} \kappa = \max_{i,j} \left\{ \left| \lambda_i \right|^{-1}, \hspace{0.05in} \left| \mu_j \right|^{-1} \right\}.
	\end{equation*} 
Note that now $\kappa$ is a fixed real number such that $ \min_{i,j} \{  \kappa | \lambda_i|, \hspace{0.05in} \kappa | \mu_j| \} \geq 1.$ As in Part 1 of the minor arcs analysis, one can use  Lemma \ref{lem3.4.10} and Lemma \ref{lem3.3.6} in order to estimate the right hand side of the inequality (\ref{eq3.2.14}). Hence, we infer that for any fixed $ 0 < \epsilon < 1 $ one has
	\begin{equation*}
		\int_{ \mathfrak M^{+} } \int_{ \mathfrak n_\xi} \left| \mathcal F (\tuplealpha) K_\pm (\alpha_\theta) \right| \text d \tuplealpha \ll P^{ (1- \eta_2 + \epsilon) \delta} \cdot  P^{ s^\prime  -  (\theta +d) + \epsilon} \ll P^{ s - (\theta +d) - \eta_2 \delta + (1 + \delta) \epsilon}.
	\end{equation*}
One may now choose 
	\begin{equation*} 
		\delta = \frac{1}{6} \in (0, 1/3) \hspace{0.3in} \text{and} \hspace{0.3in} \epsilon = \frac{ \xi}{ (1 + \delta) 2^d} \in (0,1),
	\end{equation*}
as we are at liberty to do. With these choices one has $ - \eta_2 \delta + (1 + \delta) \epsilon < 0.$ Hence, the previous estimate delivers
	\begin{equation*}
		\int_{ \mathfrak M^{+} } \int_{ \mathfrak n_\xi} \left| \mathcal F (\tuplealpha) K_\pm (\alpha_\theta) \right| \text d \tuplealpha = o \left( P^{s - (\theta +d)} \right).
	\end{equation*}
In the light of (\ref{eq3.4.13}) we have established the following.

	\begin{lemma} \label{lem3.4.11}
			For  $ s_{\normalfont \min} \leq s \leq s_{\normalfont \max}$ one has 
			\begin{equation*}
				\int_{ \mathfrak M } \int_{ \mathfrak n_\xi} \left| \mathcal F (\tuplealpha) K_\pm (\alpha_\theta) \right|  \normalfont \text d \tuplealpha = o \left( P^{s - (\theta +d)} \right).
			\end{equation*}
	\end{lemma}

Before we close this section, we find it appropriate to record the following lemma which concerns the complete set of minor arcs 
	\begin{equation*}
		\mathfrak p  = \left( [0,1)  \times \mathfrak m \right) \cup \left( \mathfrak n_\xi  \times  \mathfrak M \right).
	\end{equation*}
Combining Lemma \ref{lem3.4.7} and Lemma \ref{lem3.4.11} we have established the following.

	\begin{lemma} \label{lem3.4.12}
		For  $ s_{\normalfont \min} \leq s \leq s_{\normalfont \max}$ one has
			\begin{equation*}
				\int_{\mathfrak p } | \mathcal F (\tuplealpha) K_\pm (\alpha_\theta) |  \normalfont \text d \tuplealpha = o \left( P^{s - (\theta +d)} \right).
			\end{equation*}
	\end{lemma}


\section{Trivial arcs}

In this section we deal with the disposal of the set of trivial arcs $ \mathfrak c = [0,1) \times  \mathfrak t ,$ where recall that $ \mathfrak t = \{ \alpha_\theta \in \mathbb R : |\alpha_\theta| \geq P^\omega \}.$ We put $ \mathfrak t^{+} = [P^\omega, \infty)$ and $ \mathfrak t^{-} = (-\infty,  - P^\omega].$ Note that $ \mathfrak t =  \mathfrak t^{+} \cup  \mathfrak t^{-}.$ We set $ \mathfrak c^{+} = [0,1) \times  [P^\omega, \infty) $ and $ \mathfrak c^{-} =  [0,1) \times  ( - \infty, P^\omega] .$ By a change of variables as in (\ref{eq3.4.1}) one has
	\begin{equation} \label{eq3.5.1}
		R_\pm \left( P ; \mathfrak c^{-} \right) = \overline{	R_\pm \left( P ; \mathfrak c^{+} \right)}.
	\end{equation}
So, it is enough to deal with the set $ \mathfrak c^{+}.$ 

Fix an index $i.$ One has 
	\begin{equation*}
		\mathfrak c^{+ }\subset \bigcup_{ \rho  = \lfloor \omega \log_2 P \rfloor}^\infty \left(  [0,1) \times  \left(2^\rho , 2^{\rho +1} \right] \right).
	\end{equation*}
We take $ \kappa = 2^{\rho +1}.$ Here we consider large enough values of $P$ so that for $ \rho  \geq \lfloor \omega \log_2 P \rfloor$ one has $ \min_{i,j} \{ \kappa | \lambda_i|, \hspace{0.05in} \kappa| \mu_j| \} \geq 1.$ By Lemma \ref{lem3.3.6}  and taking into account (\ref{eq3.2.2}), one has for any fixed $ \epsilon > 0$ that
	\begin{equation*}
		\begin{split}	
			\Xi_{f_i} (\mathfrak c^{+}) & \ll \sum_{ \rho  = \lfloor \omega \log_2 P \rfloor}^\infty \int_{2^\rho }^{2^{\rho +1}} \int_0^1 |f_i(\alpha_d, \alpha_\theta)|^{A_\theta } | K_\pm (\alpha_\theta) | \text d \tuplealpha \\[10pt]
			& \ll P^{A_\theta  - (\theta +d) + \epsilon} \sum_{ \rho  = \lfloor \omega \log_2 P \rfloor}^\infty \frac{ 1}{ 2^\rho }. 
		\end{split}
	\end{equation*}
Clearly,
	\begin{equation*}
		\sum_{ \rho  = \lfloor \omega \log_2 P \rfloor}^\infty \frac{ 1}{ 2^\rho  } \ll P^{- \omega}.
	\end{equation*}
Hence, by choosing $ \epsilon = \frac{\omega}{2} > 0$ the previous estimate now delivers
	\begin{equation*}
			\Xi_{f_i}  (\mathfrak c^{+}) \ll P^{A_\theta  - (\theta +d) - \frac{\omega}{2} }. 
	\end{equation*}

One can deal with the auxiliary mean values $ \Xi_{f_i, g_j}  (\mathfrak c^{+}), \Xi_{f_i, h_k}  (\mathfrak c^{+}), \Xi_{g_j, h_k} (\mathfrak c^{+}) $ similarly. We now put these estimates together. One is at liberty to take $ \delta =0$ in the inequality (\ref{eq3.2.14}). So, in this case by (\ref{eq3.2.11}) one has $ s^\prime = s ,$ and by (\ref{eq3.2.15}) one has $ s = A_\theta + (1-\omega_1)A_d.$ Thus we obtain
	\begin{equation*}
		\int_{\mathfrak t^{+} } \int_0^1 \left| \mathcal F (\tuplealpha) K_\pm (\alpha_\theta) \right| \text d \tuplealpha \ll P^{A_\theta + (1-\omega_1)A_d - (\theta +d) -\frac{\omega}{2}} = o \left( P^{s - (\theta+d)}\right).
	\end{equation*}
In the light of (\ref{eq3.5.1}) we have established the following.

	\begin{lemma} \label{lem3.5.1}
			For $ s_{\normalfont \min} \leq s \leq s_{\normalfont \max}$ one has
			\begin{equation*}
				\int_{ \mathfrak c}  \left| \mathcal F (\tuplealpha) K_\pm (\alpha_\theta) \right| \normalfont{\text d} \tuplealpha = o \left( P^{s - (\theta+d)}\right).
			\end{equation*}
	\end{lemma}



\section{Major arcs analysis}

In this section we deal with the set of major arcs  $\mathfrak P = \mathfrak n_\xi \times \mathfrak M.$ We split the analysis into two subsections, dealing separately with the singular integral and the singular series.

\subsection{Singular integral analysis}

Here we deal with the singular integral. For each index $ i, j$ and $ k,$ and any $ \tuplebeta = (\beta_d, \beta_\theta) \in \mathbb R^2 $ we define the continuous generating functions
	\begin{equation} \label{eq3.6.1}
		\begin{split}
			& \upsilon_{f,i}( \tuplebeta) = \int_{ \frac{1}{2}x_i^\star P}^{2 x_i^\star P}  e ( a_i \beta_d \gamma^d  + \lambda_i \beta_\theta \gamma^\theta )  \text d \gamma, \\[10pt] 
			& \upsilon_{ g,j}(\tuplebeta) = \int_{ \frac{1}{2}y_j^\star P}^{2 y_j^\star P} e ( \mu_j \beta_\theta \gamma^\theta)  \text d \gamma, \\[10pt]		  
			& \upsilon_{  h,k}( \tuplebeta) = \int_{ \frac{1}{2}z_k^\star P}^{2 z_k^\star P} e ( b_k \beta_d \gamma^d) \text d \gamma. 
		\end{split}
	\end{equation} 
Moreover, we write
	\begin{equation*} 
		V(\tuplebeta) = \prod_{i =1}^\ell \upsilon_{  f,i}( \tuplebeta) \prod_{j=1}^m  \upsilon_{  g,j}(\tuplebeta) \prod_{k=1}^n \upsilon_{  h,k}( \tuplebeta ).
	\end{equation*}	
Define the truncated singular integrals 
	\begin{equation} \label{eq3.6.2}
		\begin{split}	
			& \mathfrak J^\pm (P^\xi, P^{\delta_0}) = \int_{-P^{-\theta+\delta_0}}^{P^{-\theta+ \delta_0}} \int_{-P^{-d+\xi}}^{P^{-d+\xi}} V (\tuplebeta) K_\pm(\beta_\theta) \text d \tuplebeta, \\[10pt]
			&\mathfrak J (P^\xi, P^{\delta_0}) = \int_{-P^{-\theta+ \delta_0}}^{P^{-\theta+ \delta_0}} \int_{-P^{-d+\xi}}^{P^{-d+\xi}} V (\tuplebeta) \text d \tuplebeta,
		\end{split}
	\end{equation}
and the complete singular integral
	\begin{equation} \label{eq3.6.3}
		\mathfrak J(\infty) = \int_{-\infty}^\infty \int_{-\infty}^\infty V (\tuplebeta) \text d \tuplebeta. 
	\end{equation}

	\begin{lemma} \label{lem3.6.1}
		For each index $i,j,k$ and for any $ \tuplebeta = (\beta_d, \beta_\theta) \in \mathbb R^2 $ one has
			\begin{itemize}
				\item[(i)] $\upsilon_{f,i}(\tuplebeta) \ll P(1+ P^d | \beta_d| + P^\theta | \beta_\theta|)^{-1/ \theta} ;$
				\item [(ii)] $ \upsilon_{ g,j}(\tuplebeta) \ll P (1 + P^\theta | \beta_\theta|)^{-1 / \theta} ;$
				\item [(iii)] $ \upsilon_{h,k} (\tuplebeta) \ll P( 1 + P^d | \beta_d|)^{-1/d}.$
			\end{itemize}
	\end{lemma}

In the case where $ \theta \in \mathbb N$ one can find a proof of this lemma in \cite[Theorem 7.3]{vaughan_book}. In our case one has $ \theta \notin \mathbb N.$ For this reason we give an alternative proof using van der Corput's estimate for oscillatory integrals, dating back to 1935 in van der Corput's work on the stationary phase method \cite{VdC_osc_int}.

	\begin{lemma} \label{lem3.6.2}
		Let $ \lambda $ be a positive real. Suppose that $ \phi : (a,b) \to \mathbb R $ is a smooth function in $(a,b),$ and suppose that $ \left| \phi^{(k)} (x) \right| \geq 1 $ for all $ x \in (a,b).$ Then,
			\begin{equation*}
				\displaystyle \left|\int_a^b e^{i \lambda \phi (x)} \normalfont \text d x \right| \leq c_k \lambda^{-1/k}
			\end{equation*}
		holds when:
			\begin{itemize}
				\item[(i)] $ k \geq 2,$ or
				\item [(ii)]$k=1$ and $ \phi^\prime (x)$ is monotonic.
			\end{itemize}
		The bound $c_k$ is independent of $ \phi $ and $ \lambda.$
	\end{lemma}

	\begin{proof}
		See \cite[Proposition 2, p.332]{stein_book}.
	\end{proof}

	\begin{proof}[Proof of Lemma \ref{lem3.6.1}]
		 
		The estimates $(ii)$ and $(iii)$ can be easily established by using integration by parts. As an alternative approach, one may use Lemma \ref{lem3.6.2} as below.  Now we come to prove estimate $(i).$
		
		For $ \tuplebeta = (\beta_d, \beta_\theta) \in \mathbb R^2$ we put
			\begin{equation*}
				\upsilon_f ( \tuplebeta)  = \int_{1/2}^2 e (\beta_d \gamma^d + \beta_\theta \gamma^\theta) \text d \gamma.
			\end{equation*}
		It is enough to prove that
				\begin{equation} \label{eq3.6.4}				
					\upsilon_f (\tuplebeta) \ll \frac{1}{ \left(1+ | \beta_d| + | \beta_\theta| \right)^{1/ \theta} }.
				\end{equation}	
		The desired estimate for the function $ \upsilon_{f,i}$ follows by a change of variables replacing $ \gamma$ by $x_i^\star P \gamma.$ Then one can apply (\ref{eq3.6.4}) with  $ a_i (x_i^\star P)^d \beta_d$ in place of $ \beta_d $ and $ \lambda_i (x_i^\star P)^\theta \beta_\theta$ in place of $ \beta_\theta.$ 
		
		It is apparent that $ \left| \upsilon_f (\tuplebeta) \right| \leq 3/2 \ll 1 .$ So, if $ |\beta_d | + | \beta_\theta| < 1$ then  (\ref{eq3.6.4}) trivially holds. Hence, in the rest of the proof we may suppose that $ | \beta_d | + | \beta_\theta | \geq 1.$  For $  \gamma \in [1,2]$ we define the function
			\begin{equation*}
				\phi( \gamma ) =  \beta_d \gamma^d + \beta_\theta  \gamma^\theta .
			\end{equation*}
		We distinguish the following two cases about $ \beta_d$ and $\beta_\theta.$	
		
		\textit{Case} $(1).$ Suppose that $ | \beta_\theta | > | \beta_d|.$ Recall that $ d $ is a positive integer such that $ \theta > d + 1.$ This last condition implies that $ d < \lfloor \theta \rfloor .$ Temporarily we write $ n= \lfloor \theta \rfloor .$  Hence, for $ \gamma \in [1/2, 2]$ one has
			\begin{equation*}
				| \phi^{(n)} (\gamma ) | = C_n | \beta_\theta|  \gamma^{\theta -n} \geq  C_n \left(\frac{1}{2} \right)^{\theta -n} | \beta_\theta|, 
			\end{equation*}
		where we put $  C_n =  \theta (\theta-1) \cdots (\theta -n +1) .$ Put $ C = C_n \left(\frac{1}{2} \right)^{\theta -n}.$ One can now take $ \lambda =C| \beta_\theta|$ and apply Lemma \ref{lem3.6.2} with $k =n$ to the function 
			\begin{equation*}
				\gamma \mapsto \frac{1}{C | \beta_\theta|  } \phi (\gamma).
			\end{equation*}
		Since $ | \beta_\theta| > | \beta_d|$ and $ | \beta_d| +  \left| \beta_\theta \right| \geq 1,$ we deduce that
			\begin{equation*}
				\int_{1/2}^2 e(	\phi( \gamma )) \text d \gamma \leq C^{-1/n} \left| \beta_\theta \right|^{-1/n} \ll \frac{1}{( 1+  |\beta_d| + \left| \beta_\theta \right|)^{1/\theta}},
			\end{equation*}
		which confirms (\ref{eq3.6.4}).
		
		\textit{Case} $(2).$ Suppose that $ | \beta_\theta | \leq | \beta_d|.$ One has
			\begin{equation*} 
				| \phi^{(d)}( \gamma ) | = \left| d ! \beta_d  +  C_d \beta_\theta  \gamma^{\theta-d} \right|,
			\end{equation*}
		where we put $ C_d =  \theta (\theta -1) \cdots (\theta -d+1).$
		In order to give a lower bound for the quantity $ \left| \phi^{(d)} (\gamma) \right|$ we examine separately the following two scenarios.
	 	
	 	Suppose that 
			\begin{equation*}
				 \frac{1}{2} d ! |\beta_d | \geq  C_d 2^{\theta-d} |\beta_\theta| .
			\end{equation*}
		By the triangle triangle inequality one may infer for $ \gamma \in [1/2 ,2]$ that
			\begin{equation*}
				\begin{split}	
					| \phi^{(d)}( \gamma) |  >  d ! |\beta_d | -  C_d \gamma^{\theta -d} \left| \beta_\theta \right|  \geq  d ! |\beta_d | -  C_d 2^{\theta- d} \left| \beta_\theta \right|  \geq \frac{1}{2}  d ! |\beta_d |.
				\end{split}
			\end{equation*}
		One can now take $ \lambda = 2^{-1} d ! |\beta_d |$ and apply Lemma \ref{lem3.6.2} with $k =d$ to the function 
			\begin{equation*} 
				\gamma \mapsto \frac{1}{ 2^{-1} d ! |\beta_d |} \phi (\gamma).
			\end{equation*}
		 Since $ | \beta_d| \geq | \beta_\theta|$ and $ | \beta_d| + | \beta_\theta| \geq 1 ,$ we deduce that	
			\begin{equation*}
				\int_{1/2}^2 e( \phi( \gamma) ) \text d \gamma \leq (2^{-1} d!)^{-1/d} \left| \beta_d \right|^{-1/d} \ll  \frac{1}{( 1+  |\beta_d| + \left| \beta_\theta \right|)^{1/\theta}},
			\end{equation*}
		which again confirms (\ref{eq3.6.4}).
			
 		Next, we suppose that 
			\begin{equation*}
				\frac{1}{2} d ! |\beta_d | <  C_d 2^{\theta-d} |\beta_\theta| .
			\end{equation*}
		Since we assume as well that $ | \beta_d| \geq | \beta_\theta|$ one may now suppose that $ | \beta_d| \asymp | \beta_\theta|.$ In such a situation an application of Lemma \ref{lem3.6.2} with $ k = n$ as in \textit{Case $(1)$} yields 
			\begin{equation*}
				\int_{1/2}^2 e \left( \phi(\gamma) \right) \text d \gamma \ll | \beta_\theta|^{-1/\theta} \ll  \frac{1}{( 1+  |\beta_d| + \left| \beta_\theta \right|)^{1/\theta}},
			\end{equation*}
		and thus the proof is now complete.
	\end{proof}

Define $ \Delta = \Delta(\theta, d , \ell, m,n) > 0$ via
		\begin{equation} \label{eq3.6.5}
			\Delta( \theta, d , \ell, m,n) = \min \left\{ \frac{m}{\theta} + \frac{\ell}{2\theta} -1, \hspace{0.05in} \frac{n}{d} + \frac{\ell}{2\theta}-1 \right\}.
		\end{equation}
Note that the assumptions $ \ell + m \geq A_\theta + 1 $ and $ \ell \geq  \max \{ \lceil 2 \theta (1 - n/d) \rceil, \hspace{0.05in} 1 \}$ ensure that $ \Delta > 0.$
		\begin{lemma} \label{lem3.6.3}
			One has
				\begin{equation*}
					\mathfrak J^\pm (P^\xi, P^{\delta_0})  = 2 \tau \mathfrak J (\infty) +  o \left( P^{s - (\theta +d)} \right).
				\end{equation*}
		\end{lemma}	
	
		\begin{proof}
			For $ | \beta_\theta| < P^{-\theta + \delta_0}$ by (\ref{eq3.2.4}) one has that
				\begin{equation} \label{eq3.6.6}
					\mathfrak J^\pm (P^\xi, P^{\delta_0}) = \left( 2 \tau + O \left( \left( \log P \right)^{-2}\right) \right) \mathfrak J (P^\xi, P^{\delta_0}).
				\end{equation}
			By Lemma \ref{lem3.6.1} and a trivial estimate one has that
				\begin{equation*} 
					V(\tuplebeta)  \ll P^s \left( 1 + | \beta_\theta| P^\theta \right)^{-m / \theta} \left( 1 + | \beta_d| P^d \right)^{-n / d} \left( 1 +  | \beta_d| P^d + |  \beta_\theta| P^\theta \right)^{- \ell / \theta}.
				\end{equation*} 	
			Using the trivial estimate
				\begin{equation*}
					\alpha^{1/2} \beta^{1/2} \leq \max\{\alpha, \beta\} \ll 1 + \alpha + \beta,
				\end{equation*}
			the preceding inequality now yields
				\begin{equation} \label{eq3.6.7}
					\begin{split}				
						V(\tuplebeta) & \ll P^s \left( 1 + | \beta_\theta| P^\theta \right)^{-m / \theta - \ell / 2\theta } \left( 1 + | \beta_d| P^d \right)^{-n / d - \ell / 2 \theta} \\[10pt]
						& \ll P^s \left( 1 + | \beta_\theta| P^\theta \right)^{-(1 + \Delta)} \left( 1 + | \beta_d| P^d \right)^{-(1+ \Delta)}.			
					\end{split} 
				\end{equation}
			Temporarily we write $ \mathcal B $ to denote the box $ [P^{-\theta+ \xi}, P^{-\theta+ \xi}] \times [-P^{-d+ \delta_0}, P^{-d+ \delta_0}].$ If $ \tuplebeta \in \mathbb R^2 \setminus \mathcal B $ then we either have $ |\beta_\theta| P^{\theta} \geq P^{\delta_0} $ or $| \beta_d | P^d \geq P^\xi.$ By the preceding estimate we infer that
				\begin{equation*}
					\begin{split}
						\mathfrak J (P^\xi, P^{\delta_0})  - \mathfrak J (\infty) & \ll P^s \left( \int_{ |\beta_\theta| P^{\theta} \geq P^{\delta_0} } \int_{-\infty}^\infty V(\tuplebeta)  \text d \tuplebeta  + \int_{-\infty}^\infty \int_{ |\beta_d| P^d\geq P^\xi }  V(\tuplebeta)  \text d \tuplebeta	\right) \\[10pt]
						& \ll P^{ s -( \theta+d) - \Delta \delta_0} +  P^{ s -( \theta+d) - \Delta \xi}  \\[10pt]
						& \ll o \left( P^{s - (\theta +d)} \right).	
					\end{split}
				\end{equation*}		
			Therefore, by (\ref{eq3.6.6}) we deduce that
				\begin{equation*}
					\mathfrak J^\pm (P^\xi, P^{\delta_0})  = 2 \tau \mathfrak J (\infty) + o \left( P^{s - (\theta +d)} \right),
				\end{equation*}
			which is what we wanted to prove.
		\end{proof}

After these preliminary results we now come to the heart of the singular integral analysis.   The approach we take for studying the singular integral $ \mathfrak J $ is essentially the treatment of Schmidt as presented in \cite{schmidt_simult_rat_zero_quad_forms}. The validity of the results below should come with no surprise to the experts and to those who are familiar with the paper of Schmidt. For the sake of completeness we have decided to include the proofs that are related to the system under investigation. This is mainly due to the nature of the system (\ref{eq3.1.2}), which consists of an equation and an inequality of fractional degree. 

One can plainly extend the definition of $ \mathfrak F$ and $ \mathfrak D $ given in (\ref{eq3.1.2}) to $s$ tuples by taking the additional coefficients to be equal to zero. Namely, for an $s$ tuple $ \tuplex$ we can rewrite $ \mathfrak F $ and $ \mathfrak D$ equivalently in the shape
	\begin{equation} \label{eq3.6.8}
		\begin{cases}
			\mathfrak F (\tuplex) = \lambda_1 x_1^\theta + \cdots + \lambda_\ell x_\ell^\theta + \mu_1 x_{\ell+1}^\theta + \cdots + \mu_{\ell+m} x_{\ell+m}^\theta + 0 x_{\ell + m +1}^\theta + \cdots + 0 x_s^\theta \\[10pt]
			\mathfrak D (\tuplex) = a_1 x_1^d + \cdots + a_\ell x_\ell^d + 0 x_{\ell+1}^d + \cdots + 0 x_{s- n}^d + b_1 x_{s-n+1}^d  + \cdots + b_n x_s^d.
		\end{cases}
	\end{equation}
So, from now on we take the argument in the expressions $ \mathfrak F$ and $ \mathfrak D $ to be $s$ tuples. For convenience in the following, we write $ \mathcal B $ to denote the box defined by
	\begin{equation*} 
		\mathcal B = \left[ \frac{1}{2} \eta_1, 2 \eta_1 \right] \bigtimes \cdots \bigtimes \left[ \frac{1}{2} \eta_s, 2 \eta_s \right],
	\end{equation*}
where $ \tupleeta = (\tuplex^\star, \tupley^\star, \tuplez^\star)$ with $ 0 < \eta_i < 1/2$  is a non-singular real solution of the system (\ref{eq3.1.3}), with $ \mathfrak F $ and $ \mathfrak  D$ defined as in (\ref{eq3.6.8}). Note that with this notation, we count solutions to the system (\ref{eq3.1.2}) with $(\tuplex, \tupley, \tuplez) \in P \mathcal B.$ 

We define the integral 
	\begin{equation*}
		\mathcal K ( \tuplebeta) = \int_{\mathcal B} e \left( \beta_\theta \mathfrak F(\tuplegamma ) + \beta_d \mathfrak D (\tuplegamma) \right) \text d \tuplegamma.
	\end{equation*}	
For future reference we note here that by (\ref{eq3.6.7}) with $P=1$ and since $ \text{meas}(\mathfrak B) = O(1)$ one has
	\begin{equation} \label{eq3.6.9}
		\mathcal K (\tuplebeta) \ll \left( 1 + | \beta_\theta|  \right)^{-(1 + \Delta)} \left( 1 + | \beta_d|  \right)^{-(1+ \Delta)}.
	\end{equation} 
Moreover, we set			
	\begin{equation} \label{eq3.6.10}
		\mathfrak J_0 = \int_{-\infty}^\infty \int_{-\infty}^\infty \mathcal K(\tuplebeta) \text d \tuplebeta.
	\end{equation}
In the light of (\ref{eq3.6.9}) the integral $ \mathfrak J_0$ is well-defined and absolutely convergent. One may express the complete singular integral $ \mathfrak J (\infty)$ in terms of $ \mathfrak J_0.$ Replace $ \gamma$ by $ \gamma P$ in (\ref{eq3.6.1}). Then make a change of variables in the right hand side of (\ref{eq3.6.3}) by putting
	\begin{equation*} 
		\begin{pmatrix}
			\beta_\theta \\
			\beta_d
		\end{pmatrix}
		=
		\begin{pmatrix}
			P^{-\theta} & 0 \\
			0 &  P^{-d} 
			\end{pmatrix}
		\begin{pmatrix}
			\beta_\theta^\prime  \\
			\beta_d^\prime
		\end{pmatrix}.
	\end{equation*}
This yields
	\begin{equation} \label{eq3.6.11}
		\mathfrak J (\infty) = P^{s- (\theta+d)} \mathfrak J_0.
	\end{equation}	

We now focus in analysing the integral $ \mathfrak J_0.$ To do so, we make use of a family of approximate singular integrals. For $ T \geq 1 $ we put
	\begin{equation} \label{eq3.6.12}
		\mathfrak J(T) = \int_{-\infty}^\infty \int_{-\infty}^\infty \mathcal K (\tuplebeta) k_T (\tuplebeta) \text d \tuplebeta,
	\end{equation}
where
	\begin{equation*}
		k_T (\tuplebeta) = \left(\frac{ \sin ( \pi \beta_\theta / T)}{ \pi \beta_\theta / T} \right)^2  \left(\frac{ \sin ( \pi \beta_d / T)}{ \pi \beta_d / T} \right)^2.
	\end{equation*}
Note again that by  (\ref{eq3.6.9}) the integral $ \mathfrak J (T) $ is well-defined and absolutely convergent. Two are the key properties of the family of integrals $ \mathfrak J (T).$ Firstly that $ \mathfrak J (T) \gg 1 $ and secondly that as $ T \to \infty$ one has $ \mathfrak J (T) \to \mathfrak J_0.$  To begin with, let us rewrite  the integrals $ \mathfrak J (T)$ using a Fourier transform formula. For $ T \geq 1 $ we put
	\begin{equation} \label{eq3.6.13}
		\psi_T (y) = 
			\begin{cases}
				T \left(  1 - T | y | \right), & \text{when } | y | \leq T^{-1}, \\[10pt]
				0, & \text{when } | y | > T^{-1}. 
			\end{cases}
	\end{equation}
A standard calculation as presented for example in \cite[Lemma 20.1]{davenport_book} reveals that
	\begin{equation*}
		\psi_T (y) = \int_{-\infty}^\infty e (\beta y) \left( \frac{ \sin \left( \pi \beta/T\right)}{ \pi \beta / T} \right)^2 \text d \beta,
	\end{equation*}
where clearly the integral is absolutely convergent. One may rewrite the integral $ \mathfrak J (T)$ defined in (\ref{eq3.6.12}) as follows
	\begin{equation*} 
			\mathfrak J (T)  = \int_{- \infty}^\infty \int_{-\infty}^\infty \left( \int_{\mathcal B} e (\beta_\theta \mathfrak F (\tuplegamma) + \beta_d \mathfrak D (\tuplegamma) \text d \tuplegamma \right) k_T (\tuplebeta) \text d \tuplebeta.
	\end{equation*}
Hence, invoking Fubini's theorem and appealing to (\ref{eq3.6.13}) one has
	\begin{equation} \label{eq3.6.14}
		\mathfrak J (T) = \int_{\mathcal B}  \psi_T (\mathfrak F (\tuplegamma)) \psi_T (\mathfrak D (\tuplegamma)) \text d \tuplegamma.
	\end{equation}

At this point we pause for a moment in order to exploit the assumption we have made that the system (\ref{eq3.1.2}) satisfies the local solubility condition. The conclusion we establish below plays an essential role in demonstrating that $ \mathfrak J (T) \gg 1 .$ The proof proceeds as in \cite[Lemma 6.2]{wooley_simult_eqts_91}, namely by using the implicit function theorem. We include a proof for the sake of completeness. In order to avoid confusion, let us observe here that in the statement of the lemma below we use $ \tupleeta$ to denote a non-singular real solution, as it is assumed in the statement of Theorem \ref{thm3.1.2}. It is at this step where we show that one can obtain a non-singular real solution with all of its components non-zero, and (as we explained in the introduction) by homogeneity one can additionally assume that its components lie in the interval $(0, 1/2).$ The non-singular real solution with these additional properties is the one considered in the definition of the box $ \mathcal B.$ 

	\begin{lemma} \label{lem3.6.4}
		Let $\tupleeta$ be a non-singular real solution of the system (\ref{eq3.1.3}), with $ \mathfrak F$ and $ \mathfrak D $ as in (\ref{eq3.6.8}). There exists locally an $(s-2)$-dimensional subspace $ \mathcal U $ of positive $(s-2)$-volume in a neighbourhood of $ \tupleeta ,$ on which one has $  \mathfrak F  = \mathfrak D =0.$ In particular, there exists a real solution $  \tupleeta^\prime = (\eta_1^\prime, \ldots, \eta_s^\prime) = (\zeta_1, \zeta_2, \tuplezeta) $ to the system (\ref{eq3.1.3}), with $ \tuplezeta \in \mathcal U$ and $ \eta_i^\prime \neq 0$ for all $ i.$
	\end{lemma}
	
	\begin{proof}
		By relabelling if necessary the variables, one has
			\begin{equation*}
				\text{det} 
					\begin{pmatrix}
						\displaystyle \frac{\partial \mathfrak F}{ \partial x_1} (\tupleeta) & \displaystyle  \frac{\partial \mathfrak F}{ \partial x_2} (\tupleeta) \\[10pt]
					\displaystyle  \frac{\partial \mathfrak D}{ \partial x_1} (\tupleeta) & \displaystyle  \frac{\partial \mathfrak D}{ \partial x_2} (\tupleeta)
					\end{pmatrix} \neq  0,
			\end{equation*}
		namely	
			\begin{equation*}
				\text{det}	
					\begin{pmatrix}
						\theta \lambda_1 \eta_1^{\theta-1} & \theta \lambda_2 \eta_2^{\theta-1} \\[10pt]
					d a_1 \eta_1^{d-1} & d a_2 \eta_2^{d-1}
					\end{pmatrix} = \theta d  \eta_1 \eta_2 \left( \lambda_1 a_2 \eta_1^{\theta - 2} \eta_2^{d-2} - \lambda_2 a_2 \eta_1^{d-2} \eta_2^{d-2} \right) \neq 0.
			\end{equation*}
		Hence, we deduce that $ \eta_1, \eta_2 \neq 0.$ Consider the generalised polynomial $ \mathfrak F $ and the polynomial $ \mathfrak D $ defined in (\ref{eq3.6.8}) as real valued functions defined in $(0, \infty)^s.$ Since some of the variables $x_i$ correspond only to the variables $z_i$ in (\ref{eq3.1.1}), and since we assume that $ \tupleeta$ is a non-singular solution in $ \mathbb R^s,$ these variables could actually equal to zero. By changing if necessary the sign of the corresponding coefficients we can assume that one has $ \eta_i \geq  0$ for all $i.$  If $ \tupleeta$ is a solution with some $z_i =0$ then we consider a sequence of points $ \tupleeta_n$ such that $ \tupleeta_n \in (0, \infty)^s$ and $ \tupleeta_n \to \tupleeta$ as $ n \to \infty.$ One can extend continuously and uniquely the corresponding derivatives of $ \mathfrak D $ to  $[0, \infty).$ Consider the map
			\begin{equation*}
				\Phi : (0, \infty)^{2 + (s-2)} \to \mathbb R^2, \hspace{0.3in} \tuplex \mapsto \Phi( \tuplex) = \left( \mathfrak F (\tuplex) , \mathfrak D ( \tuplex ) \right). 
			\end{equation*}
		By using the fact that $ \Phi( \tupleeta) = 0 = \lim_{n \to \infty} \Phi (\tupleeta_n)$ we deduce by the implicit function theorem (see for example \cite[Theorem 13.7]{apostol_book}), that there exists an open set $ \mathcal U \subset (0, \infty)^{s-2} $ whose closure contains the point $ (\eta_3, \ldots, \eta_s),$ and a continuous map $ \tupleg$ defined on  $\mathcal U$ and taking values in a neighbourhood of the point $ (\eta_1, \eta_2) \in (0, \infty)^2,$  such that for all $ \tuplezeta = (\zeta_3, \ldots, \zeta_s) \in \mathcal U $ one has
				\begin{equation} \label{eq3.6.15}
					\begin{cases}	
						\mathfrak F ( \tupleg(\tuplezeta), \tuplezeta) = 0 \\[10pt]
						\mathfrak D	( \tupleg(\tuplezeta), \tuplezeta ) = 0.
					\end{cases}
				\end{equation}	
		 Observe that if $ \tupleeta$ is a solution with all the components non-zero then we argue similarly (without the need of considering limit points) to draw the same conclusion. Thus we have showed the existence of an $(s-2)$-dimensional subspace in the neighbourhood of $ (\eta_3, \ldots, \eta_s),$ which has positive $(s-2)$-volume and on which one has $ \mathfrak F = \mathfrak D =0.$ We denote this subspace by $ \mathcal U.$ This establishes the main part in the statement of the lemma. 
		
		For the second assertion we argue as follows. One can choose $ \zeta_i \in \mathcal U $ sufficiently close to $ \eta_i$  for $ 3 \leq i \leq s.$ Namely, choose $ \zeta_i$ such that $ | \zeta_i - \eta_i|$ is sufficiently small. Then, we can solve the system (\ref{eq3.6.15}) with respect to $ \tupleg =: (\zeta_1, \zeta_2).$ Hence, we have found a tuple $ \tupleeta^\prime  = (\zeta_1, \zeta_2, \tuplezeta) $ which satisfies $ \mathfrak F( \tupleeta^\prime  ) = \mathfrak D (\tupleeta^\prime )=0.$ Recall that $\eta_1, \eta_2 \neq 0.$ Hence, by continuity we obtain that $ \zeta_1 , \zeta_2 \neq 0.$ Therefore, we can conclude that  $ \zeta_i \neq 0 $ for $ 3 \leq i \leq s.$ This completes the proof of the second part in the statement of the lemma.
	\end{proof}

We now exploit the conclusion of Lemma \ref{lem3.6.4}, in order to prove that $ \mathfrak J (T) \gg 1.$ Here we follow \cite[Lemma 2]{schmidt_simult_rat_zero_quad_forms}. 

	\begin{lemma} \label{lem3.6.5}
		One has 
			\begin{equation*}
				\mathfrak J (T) \gg 1.
			\end{equation*}
	\end{lemma} 

	\begin{proof}
		With the notation as in Lemma \ref{lem3.6.4}, we write $\tupleeta^\prime = (\zeta_1, \zeta_2, \tuplezeta)$ to denote a real solution to the system (\ref{eq3.6.8}) with $\tuplezeta \in \mathcal U$ and $ \eta_i^\prime \neq  0$ for $ 1 \leq i \leq s.$ We put $\tuplezeta = (\zeta_3, \ldots, \zeta_s).$ Note here that we assume $ \zeta_i \neq 0$ for $ 3 \leq  i\leq s.$ For $ \epsilon > 0$ we define 
		
			\begin{equation*}
				S_\epsilon = \left\{ (\tuplexi, \tuplezeta) : \tuplezeta \in \mathcal U \hspace{0.05in} \text{such that} \hspace{0.05in} \| \tupleg( \tuplezeta) - \tuplexi \|_2 < \epsilon  \right\},
			\end{equation*}
	 	where $ \| \cdot \|_2 $ stands for the usual euclidean norm in $ \mathbb R^{s-2}.$ In the set $ S_\epsilon$ we consider points $ \tuplexi \in \mathbb R^2$ which belong to a neighbourhood of the point $ \tupleg (\tuplezeta).$ Since $ \mathcal U$ is a subset of the interior of the box $ \mathcal B,$ one can now consider sufficiently small $ \epsilon$ so that $ S_\epsilon \subset \mathcal B.$ Moreover, by Lemma \ref{lem3.6.4} we know that $ \tupleg (\tuplezeta) \neq \textbf{0}.$ Hence, it becomes apparent that the set $ S_\epsilon$ has a positive $s$-volume.
	 	
	 	When viewed as real valued functions in $s$ variables, the generalised polynomial $ \mathfrak F $ and the polynomial $ \mathfrak D$ are continuously differentiable in the box $ \mathcal B,$ which is a compact subset of $ \mathbb R^s.$  Hence, $ \mathfrak F $ and $ \mathfrak D$ satisfy the Lipschitz condition with some constants $K_1$ and $K_2$ respectively. Put
	 		\begin{equation*}
	 				c = \frac{1}{2 \max\{K_1, K_2\} } > 0.
	 		\end{equation*}
		From now on we take $ T$ sufficiently large so that $  S_{c T^{-1}} \subset \mathcal B.$ 
		
		For $ (\tuplexi, \tuplezeta) \in S_{cT^{-1}}$ one has
			\begin{equation*}
				\displaystyle \frac{ \left| \mathfrak D ( \tuplexi , \tuplezeta) -  \mathfrak D ( \tupleg (\tuplezeta) , \tuplezeta) \right|}{ \left\| \left(\tuplexi, \tuplezeta \right) - \left(\tupleg (\tuplezeta) , \tuplezeta \right)  \right\|_2 } < K_2.
			\end{equation*}
		By (\ref{eq3.6.15}) one has $ \mathfrak D ( \tupleg (\tuplezeta) , \tuplezeta) =0.$ Moreover, one has $ \| (\tuplexi - \tupleg(\tuplezeta) , \textbf{0}) \|_2 < c/T$ and so the above inequality yields
			\begin{equation*}
				 	\left| \mathfrak D ( \tuplexi , \tuplezeta) \right| < \frac{c}{T} K_2 < \frac{1}{2T}.
			\end{equation*}
		Thus, for $ \tuplegamma = ( \tuplexi , \tuplezeta) \in S_{cT^{-1}}$ we deduce 
			\begin{equation*}
				\psi_T( \mathfrak D (\tuplegamma)) =  \max \left\{  0, T \left( 1- T  \left| \mathfrak D (\tuplegamma) \right| \right)  \right\} \geq \frac{T}{2}.
			\end{equation*}
		
		Similarly, when $ ( \tuplexi , \tuplezeta) \in S_{cT^{-1}}$ one can prove that 
			\begin{equation*}
				\left| \mathfrak F ( \tuplexi , \tuplezeta) \right| < \frac{1}{2T}.
			\end{equation*}
		Thus, we may again deduce that for $ \tuplegamma = ( \tuplexi , \tuplezeta) \in S_{cT^{-1}}$ one has
			\begin{equation*}
				\psi_T( \mathfrak F (\tuplegamma)) =  \max \left\{  0, T \left( 1- T  \left| \mathfrak F (\tuplegamma) \right| \right)  \right\} \geq \frac{T}{2}.
			\end{equation*}	
		
		 Note now that the set $ S_{cT^{-1}}$ has positive $s$-volume which is $ \gg T^{-2}.$ Hence, from the above conclusions and (\ref{eq3.6.14}) one has
			\begin{equation*}
				\mathfrak J (T) = \int_{\mathcal B} \psi_T( \mathfrak D (\tuplegamma)) \psi_T( \mathfrak F (\tuplegamma))  \text d \tuplegamma \gg \int_{S_{cT^{-1}}}   \left( \frac{T}{2} \right)^2 \gg \frac{1}{4},
			\end{equation*}
		which completes the proof of the lemma.
	\end{proof}
		
Next, we establish the second key property of the family of approximate integral $ \mathfrak J (T).$	
	
	\begin{lemma} \label{lem3.6.6}
		One has
			\begin{equation*}
				\mathfrak J(T) = \mathfrak J_0 + O\left(T^{- \Delta}\right),
			\end{equation*}
		where $ \Delta > 0$ is defined in (\ref{eq3.6.5}). In particular, the limit of $\mathfrak J(T) $ as $ T \to \infty $ exists and equals to $ \mathfrak J_0.$ 
	\end{lemma}

	\begin{proof}
		By (\ref{eq3.6.9}) and (\ref{eq3.6.12}) we infer that
			\begin{equation*} 
				\begin{split}
					\mathfrak J_0 - \mathfrak J(T) & = \int_{-\infty}^\infty \int_{-\infty}^\infty \mathcal K(\tuplebeta) \left( 1 - k_T (\tuplebeta) \right) \text d \tuplebeta \\[10pt]
					& \ll \int_0^\infty \int_0^\infty ( 1 + \beta_\theta)^{-(1+ \Delta)} ( 1 +  \beta_d)^{-(1+ \Delta)} \left( 1 - k_T (\tuplebeta) \right) \text d \tuplebeta.		
				\end{split}	 
			\end{equation*} 
		
		Let $ \beta \in \mathbb R $ and let $T$ be large enough so that $ \frac{\pi | \beta| }{T} < 1.$ Then one has
			\begin{equation*}
				\left(\frac{ \sin ( \pi \beta / T)}{ \pi \beta / T} \right)^2 = 1 + O \left( \frac{   | \beta |^2}{T^2} \right),
			\end{equation*}
		which yields that
			\begin{equation*} 
				1 - \left(\frac{ \sin ( \pi \beta/ T)}{ \pi \beta / T} \right)^2 \ll \min \left\{ 1, \frac{|\beta|^2}{T^2} \right\} \ll \frac{ | \beta|^2}{T^2 + | \beta |^2},
			\end{equation*}
		and thus we deduce that	
			\begin{equation*}
				1 - k_T (\tuplebeta) \ll  \frac{ | \beta_\theta|^2}{T^2 + | \beta_\theta |^2} +  \frac{ | \beta_d|^2}{T^2 + | \beta_d |^2}.
			\end{equation*}
		We can now finish the proof easily. By symmetry one has
			\begin{equation*}
				\begin{split}
					\mathfrak J_0 - \mathfrak J(T) & \ll \left( \int_0^\infty ( 1 + \beta_\theta)^{-(1+ \Delta)} \text d \beta_\theta \right) \left( \int_0^\infty ( 1 +  \beta_d)^{-(1+ \Delta)}  \frac{ | \beta_d|^2}{T^2 + | \beta_d |^2} \text d \beta_d \right) \\[10pt]
					& \ll T^{-2} \int_0^T \beta_d^{1 - \Delta} \text d \beta_d + \int_T^\infty \beta_d^{-(1+ \Delta)} \text d \beta_d \\[10pt]
					& \ll T^{- \Delta},										
				\end{split}
			\end{equation*}
		which completes the proof.
	\end{proof}

Below we put together the outcomes of the so far analysis, in order to deduce the desired estimate for the truncated singular integral defined in (\ref{eq3.6.2}).

	\begin{lemma} \label{lem3.6.7}
		One has
			\begin{equation*}
				\mathfrak J^\pm (P^\xi, P^{\delta_0}) =  2 \tau \mathfrak J_0 P^{s - ( \theta +d)} + o \left( P^{s - (\theta +d)} \right),
			\end{equation*}
		where $ \mathfrak J_0 > 0$ is defined in (\ref{eq3.6.10}).
	\end{lemma}

	\begin{proof}
	 	Combining Lemma \ref{lem3.6.3} and relation (\ref{eq3.6.11}) we deduce that
			\begin{equation*}
				\mathfrak J^\pm (P^\xi, P^{\delta_0}) = 2 \tau \mathfrak J_0 P^{s- (\theta+d)} +  o \left( P^{s - (\theta +d)} \right).
			\end{equation*}
		Moreover, by Lemma \ref{lem3.6.5} and Lemma \ref{lem3.6.6} we infer that $ \mathfrak J_0 \gg 1 $ which completes the proof.
	\end{proof}

\subsection{Singular series analysis}
	
We now study the singular series related to the equation $ \mathfrak D (\tuplex, \tuplez) = 0.$ For $ a \in \mathbb Z $ and $ q \in \mathbb N$ we write
	\begin{equation*}
		S(q,a) = \sum_{z=1}^q e \left( \frac{az^d}{q} \right).
	\end{equation*}
Furthermore we put	
	\begin{equation*}
		T(q, a) = q^{-(\ell + n)} \prod_{i =1}^\ell S(q, aa_i) \prod_{k=1}^n S(q, ab_k).
	\end{equation*}
Next, we introduce the truncated singular series and its completed analogue	
	\begin{equation*}
		\mathfrak S (P^\xi ) = \sum_{1 \leq q \leq P^\xi} \sum_{\substack{a=1 \\ (a,q)=1}}^q T(q,a) \hspace{0.4in} \text{and} \hspace{0.4in} \mathfrak S = \sum_{q=1}^\infty \sum_{\substack{a = 1 \\ (a,q)=1}}^q T(q, a). 	
	\end{equation*}
	
	\begin{lemma} \label{lem3.6.8}
	 Suppose that $ a \in \mathbb Z $ and $ q \in \mathbb N$ with $(a,q)=1.$ Then for each index $i$ and $k$ one has
	 	\begin{itemize}
	 		\item[(i)] $ 	S(q,aa_i) \ll q^{1 - 1/d}$;
	 		\item [(ii)] $ S(q, ab_k) \ll q^{1 - 1/d}.$
	 	\end{itemize}
	\end{lemma}
	
	\begin{proof}	 	
		 By \cite[Lemma 6.4]{davenport_book} we know that when $(a,q)=1$ one has
		 	\begin{equation*}
		 		S(q,a) \ll q^{1 - 1/d}.
		 	\end{equation*}
		 
		 Fix an index $i.$ Note that one has
		 	\begin{equation*}
		 		S (q,aa_i) = \sum_{z=1}^q e \left( \frac{a_iaz^d}{q} \right) = (q, a_i) S \left( \frac{q}{(q,a_i)} , \frac{a_i a}{ (q, a_i)} \right).
		 	\end{equation*}
		 Since $(a,q )=1$ one has $\left( \frac{q}{(q,a_i)} , \frac{a_i a}{ (q, a_i)} \right) =1.$ Thus we derive that
		 	\begin{equation*}
		 		 S \left( \frac{q}{(q,a_i)} , \frac{a_i a}{ (q, a_i)} \right) \ll q^{1- 1/d},
		 	\end{equation*}
		 which in turn, and since $a_i$ is a fixed integer, delivers the estimate
		 	\begin{equation*}
		 			S (q,aa_i) \ll q^{1- 1/d}.
		 	\end{equation*}
		 
		 Similarly we argue for $S(q,ab_k).$
	\end{proof}

	\begin{lemma} \label{lem3.6.9}
		Provided that $ \ell + n \geq  A_d +1 $ the singular series is absolutely convergent. Moreover one has $ \mathfrak S > 0$ and 
			\begin{equation*}
			 	\mathfrak S(P^\xi) = \mathfrak S + O \left( P^{-\xi/d} \right).
			\end{equation*}
	\end{lemma}
	
	\begin{proof}
		The first two claims follow from the analysis of Davenport as presented in \cite[Sections 5 \& 6]{davenport_book}. Recall that we write $A_d  = d^2.$ By \cite[Theorem 1]{davenport_lewis_homo_add_eqts} we know that if $ \ell + n \geq A_d + 1$ then the singular series is absolutely convergent and positive. For the last assertion note that by Lemma \ref{lem3.6.8} one has		
			\begin{equation*}
				\left| \mathfrak S - \mathfrak S(P^\xi) \right| \leq \sum_{ q >P^\xi} \sum_{\substack{ a = 1 \\ (a,q)=1}}^q \left| T(q,a) \right| \ll \sum_{ q >P^\xi} q^{1 - (\ell+n)/d} \ll P^{(2- (\ell+n)/d) \xi}.
			\end{equation*}
		For $d \geq 2$ one has $ \ell + n \geq A_d +1 \geq 2d+1,$ where in the second inequality, the equality case holds only when $ d=2.$ Thus, we obtain $ \frac{\xi}{d} \leq  \left( \frac{\ell+n}{d} -2 \right) \xi.$ The previous estimate now delivers
			\begin{equation*}
				\left| \mathfrak S - \mathfrak S(P^\xi) \right| \ll P^{-\xi /d},
			\end{equation*}
		which completes the proof.	
	\end{proof}



\section{The asymptotic formula}

We now combine the results from the previous sections to establish the anticipated asymptotic formula for the counting function $ \mathcal N (P).$

For $ \alpha_d \in \mathfrak N_\xi (q,a)$ we write $ \alpha_d = \beta_d + a/q$ with $ | \beta_d | < P^{-d+ \xi}.$ From now on we take $ \tuplebeta = ( \beta_d, \alpha_\theta),$ with $ \alpha_\theta \in \mathfrak M.$ For each $i,j$ and $k$ we define the approximate generating functions
	\begin{equation*}
		f_i^\star(\tuplebeta) = \frac{1}{q} S_{f,i}(q, a) \upsilon_{f,i}( \tuplebeta), \hspace{0.1in} g_j^\star(\tuplebeta) = \upsilon_{g,j}( \tuplebeta), \hspace{0.1in} h_k^\star(\tuplebeta) = \frac{1}{q} S_{h,k}(q, a) \upsilon_{h,k}(\tuplebeta).
	\end{equation*}
Put
	\begin{equation*}
		\mathcal F^\star (\tuplebeta) = \prod_{i=1}^\ell f_i^\star(\tuplebeta) \prod_{ j=1}^m g_j^\star(\tuplebeta) \prod_{k=1}^n h_k^\star(\tuplebeta).
	\end{equation*}
We wish to compare $ \mathcal F (\tuplealpha) $ with $ \mathcal F^\star (\tuplebeta).$  Below we record a consequence of Poisson's summation formula.
	
	\begin{lemma} \label{lem3.7.1}
		Let $ f : [a,b] \to \mathbb R $ be a function differentiable in $[a,b].$ Suppose that $ f^\prime (x)$ is monotonic, and suppose that $ \left| f^\prime (x) \right| \leq A < 1 \hspace{0.02in} $ for all $ x \in [a,b].$ Then
			\begin{equation*}
				\sum_{ a < x \leq b} e  \left( f(x) \right) = \int_a^b  e  \left( f(x) \right)  \normalfont \text d x + O (1).
			\end{equation*}
	\end{lemma}

	\begin{proof}
		See \cite[Lemma 4.8]{titchmarsh_book}.
	\end{proof}

	\begin{lemma} \label{lem3.7.2}
		For each index $i,j, k$ and for any $ \tuplealpha = (\alpha_d, \alpha_\theta) \in \mathfrak N_\xi (q,a) \times \mathfrak M $ one has	
			\begin{itemize}
				\item[(i)] $  f_i (\tuplealpha) - f_i^\star (\tuplebeta) \ll P^{\delta
				_0 + \xi}$ ;
				\item[(ii)] $  g_j (\tuplealpha) - g_j^\star (\tuplebeta) \ll 1 $; 
				\item [(iii)] $  h_k (\tuplealpha) - h_k^\star (\tuplebeta) \ll P^{2 \xi}.$
			\end{itemize}
	\end{lemma}
	 
	\begin{proof}
	 	For the estimate $(iii)$ one can argue as in \cite[Lemma 4.2]{davenport_book}. 
	 	
	 	For the estimate $(ii)$ we apply Lemma \ref{lem3.7.1}. Fix an index $j.$ Recall from (\ref{eq3.6.1}) the definition of the function $\upsilon_{g,j} (\tuplebeta).$ Then, the claimed estimate reads
	 		\begin{equation*}
	 			\sum_{ \frac{1}{2} y_i^\star P < y \leq 2y_i^\star P} e( \mu_j \alpha_\theta y^\theta) -  \int_{ \frac{1}{2}y_j^\star P}^{2 y_j^\star P} e ( \mu_j \alpha_\theta \gamma^\theta)  \text d \gamma = O (1),
	 		\end{equation*}
	 	for $ \alpha_\theta \in \mathfrak M.$ By taking the complex conjugate it suffices to prove the above estimate when $ \alpha_\theta > 0.$ For a real variable $t$ we define the function 
	 		\begin{equation*}
	 		 	\phi : \left(\frac{1}{2} y_i^\star P, 2y_i^\star P\right] \to \mathbb R, \hspace{0.2in} \phi (t) = \mu_j \alpha_\theta t^\theta.
	 		\end{equation*}
	 	The function $ \phi^{\prime \prime} (t)$ is of fixed sign and so $ \phi^\prime (t)$ is monotonic. Moreover, for $ \alpha_\theta \in \mathfrak M$ and for large enough $P$ one has
	 		\begin{equation*}
	 			| \phi^\prime (t) | = | \mu_j| \theta \alpha_\theta t^{\theta -1} \leq | \mu_j| \theta (2 y_i^\star)^{\theta-1} P^{-1 + \delta_0} < 1,
	 		\end{equation*}
	 	where recall from (\ref{eq3.2.7}) that $ \delta_0 < 1.$ Thus, Lemma \ref{lem3.7.1} is applicable and yields the desired conclusion.
	 	
	 	Now we prove estimate $(i).$ Here we argue as in \cite[Lemma 4.2]{davenport_book}. We fix an index $i.$ Decomposing into residue classes modulo $q$ and writing $ x=qy+z$ with $ 1 \leq z \leq q $ we obtain
			\begin{equation} \label{eq3.7.1}
				\begin{split}	
					f_i(\tuplealpha) &= \sum_{ z=1}^q \sum_{y \in I(z)} e \left(a_i ( \beta_d + a/q) (qy+z)^d + \lambda_i \alpha_\theta (qy+z)^\theta \right) \\[10pt]
					& = \sum_{z=1}^q e \left( a_i a z^d /q \right) \sum_{y \in I}  e \left( a_i  \beta_d (qy+z)^d + \lambda_i \alpha_\theta (qy+z)^\theta \right),
				\end{split}			
			\end{equation}
		where $I = I(z)$ is the interval defined by 
			\begin{equation*} 
				I (z) = \left( \frac{ \frac{1}{2}x_i^\star P -z}{q}, \hspace{0.05in} \frac{2 x_i^\star P - z}{q} \right].
			\end{equation*}
		For ease of notation we denote the endpoints of the interval $I$ by $A $ and $B,$ namely we put
			\begin{equation*}
				A = \frac{ \frac{1}{2}x_i^\star P -z}{q} \hspace{0.3in} \text{and} \hspace{0.3in} B = \frac{2 x_i^\star P - z}{q},
			\end{equation*}
		and we write $I = (A,B].$
				
		For $t \in \mathbb R$ we put 	
			\begin{equation*}
				\phi_i (t)  =  e \left( a_i \beta_d (qt+z)^d + \lambda_i \alpha_\theta (qt+z)^\theta \right).
			\end{equation*}
		The function $ \phi_i$ is a holomorphic complex valued function of the real variable $t.$ Consider an arbitrary interval $ [x,x+1] \subset \mathbb R$ of length equal to $1.$ By the fundamental theorem of calculus one has for any $ t\in [x, x+1]$ that
			\begin{equation*}
				\left| \phi_i(t) - \phi_i(x)  \right| = \left| \int_x^t \phi_i^\prime (u) \text d u \right| \leq \max_{ u \in [x,x+1]} \left| \phi_i^\prime (u) \right|.
			\end{equation*}
		One can break the interval $ I$ into $ \ll B - A = O \left( P q^{-1} \right)$ unit intervals of the shape $ [x, x+1]$ with $ x \in \mathbb Z,$ together with two possible broken intervals in the case where at least one of the endpoints $A$ and $B$ of the interval $I$ is not an integer. Then, we deduce that
			\begin{equation*}
				\begin{split}	
					\left| \sum_{ A < y \leq B } \phi_i (y) - \int_A^B \phi_i (t) \text d t  \right| & \ll \sum_{ A < y \leq B} \int_y^{y+1} \left| \phi_i(y) - \phi_i(t) \right| \text d t + \max_{A < y \leq B} \left| \phi_i (t) \right| \\[10pt]
					& \ll Pq^{-1} \max_{ A < t \leq B} \left| \phi_i^\prime (t) \right| + \max_{A < t \leq B} \left| \phi_i (t) \right|.
				\end{split}
			\end{equation*} 
		Clearly, $ | \phi_i (t) | \leq 1 $ for all $ t .$ One has
			\begin{equation*}
				 \phi_i^\prime (t) = 2 \pi i \left( a_i d q \beta_d (qt + z)^{d-1} + \lambda_i \theta q \alpha_\theta (qt + z)^{\theta -1} \right) \phi_i (t).
			\end{equation*}
		Hence, for any $ t \in I$ one has
			\begin{equation*}
				\left| \phi_i^\prime (t) \right| \ll q | \beta_d| P^{d-1} + q | \alpha_\theta |P^{\theta -1}.
			\end{equation*}
		Therefore for $ (\alpha_d, \alpha_\theta) \in \mathfrak N_\xi(q,a) \times \mathfrak M$ and since $ \xi < \delta_0,$ the preceding estimate now delivers
			\begin{equation} \label{eq3.7.2}
				\left| \sum_{ A < y \leq B } \phi_i (y) - \int_A^B \phi_i (t) \text d t  \right|  \ll  P^{\delta_0 }.
			\end{equation}

		We put $ q t + z = \gamma$ and make a change of variables. Then one has
			\begin{equation*} \label{fi_fistar_phi_integral_upsilon}
				 \int_A^B \phi_i (t) \text d t  = \frac{1}{q} \int_{ \frac{1}{2} x_i^\star P}^{ 2 x_i^\star P} e (a_i \beta_d \gamma^d + \lambda_i \alpha_\theta \gamma^\theta ) \text d \gamma = \frac{1}{q} \upsilon_{f,i} (\tuplebeta),
			\end{equation*}
		where bear in mind that $ \tuplebeta = (\beta_d, \alpha_\theta) = (\alpha_d - a/q, \alpha_\theta).$ Putting together (\ref{eq3.7.1}) and (\ref{eq3.7.2}) yields			
			\begin{equation*}
				\begin{split}
					f_i( \tuplealpha) &  = \sum_{z=1}^q e\left( a_i a z^d /q \right) \left(  \int_A^B \phi (t) \text d t  +  P^{\delta_0} \right) \\[10pt]					
					 & = \frac{1}{q} \sum_{z=1}^q e\left( a_i a z^d /q \right) \upsilon_{f,i} (\tuplebeta) +  O \left(  P^{\delta_0 + \xi} \right),					
				\end{split}
			\end{equation*}	
		 where in the last step we used the fact that $ 1 \leq q \leq P^\xi.$ The proof is now complete.
	 \end{proof}
 
By Lemma \ref{lem3.7.2} and using a standard telescoping identity one has that
 	\begin{equation*} 
 		\begin{split}	
 			\mathcal F( \tuplealpha) - \mathcal F^\star(\tuplebeta) & \ll P^{s-1} \left( | f_i - f_i^\star| + | g_j - g_j^\star | + | h_k - h_k^\star| \right) \\[5pt]
			& \ll P^{s-1 + \delta_0 + \xi}.
 		\end{split}
 	\end{equation*}
 Moreover one has
	\begin{equation*} 
		\text{meas} \left( \mathfrak N_\xi (q,a) \times \mathfrak M\right) \asymp P^{-d + \xi} \cdot P^{-\theta + \delta_0} = P^{-(\theta+d) + \delta_0 + \xi}.
	\end{equation*} 
Next, note that one has $ \mathcal F^\star(\tuplebeta) = V( \tuplebeta) T(q,a).$ Integrating over the set $ \mathfrak N_\xi(q,a) \times \mathfrak M  $ against the measure $ K_\pm (\alpha_\theta) \text d \tuplealpha$ and taking into account the preceding observations reveals
	\begin{equation*}
		\int_{ \mathfrak M} \int_{\mathfrak N_\xi (q,a) } \mathcal F( \tuplealpha) K_\pm(\alpha_\theta) \text d \tuplealpha = T(q,a) \int_{\mathfrak M} \int_{\mathfrak N_\xi (q,a)}  V( \tuplebeta) K_\pm(\alpha_\theta) \text d \tuplebeta + E,
	\end{equation*}	
where 
	\begin{equation*}
		E = O \left( P^{s - (\theta+d) - 1 + 2(\delta_0 + \xi)} \right).
	\end{equation*}
One can now sum over $ 1 \leq q \leq P^\xi$ and $ 1 \leq r \leq q$ to conclude that
	\begin{equation*} \label{4_asym_for_maj_arc}
		\int_{\mathfrak M} \int_{\mathfrak N_\xi} \mathcal F( \tuplealpha) K_\pm (\alpha_\theta) \text d \tuplealpha = \mathfrak S(P^\xi) \mathfrak J^\pm (P^\xi, P^{\delta_0}) + O \left( P^{s - (\theta+d) - 1 + 2\delta_0 + 4 \xi} \right).
	\end{equation*} 	
Recall from (\ref{eq3.2.7}) that $ \delta_0 = 2^{1-2\theta}$ and from (\ref{eq3.2.8}) that $ 0 < \xi \leq \delta_0 /8.$ Recall that we assume $ \theta > d + 1 \geq 3.$ Hence, for the error term in the above asymptotic formula one has
	\begin{equation*}
		 P^{s - (\theta+d) - 1 + 2\delta_0 + 4 \xi} \ll  P^{s - (\theta +d) - 1 + \frac{5}{2} \delta_0} = o \left( P^{s- (\theta +d)} \right).
	\end{equation*}

By Lemma \ref{lem3.6.7} and Lemma \ref{lem3.6.9} one has
	\begin{equation*}
		\mathfrak S(P^\xi) \mathfrak J^\pm (P^\xi, P^{\delta_0})  = 2 \tau \mathfrak J_0 \mathfrak S P^{s - (\theta +d)} + o \left( P^{s- (\theta +d)} \right).
	\end{equation*}
Thus we conclude that
	\begin{equation*}
		\int_{\mathfrak P} \mathcal F( \tuplealpha) K_\pm (\alpha_\theta) \text d \tuplealpha  = 2 \tau \mathfrak J_0 \mathfrak S P^{s - (\theta +d)} + o \left(P^{s- (\theta +d)}\right),
	\end{equation*}
where recall that $ \mathfrak P =  \mathfrak N_\xi \times  \mathfrak M .$ Upon invoking (\ref{eq3.2.10}) and taking into account Lemma \ref{lem3.4.12} and Lemma \ref{lem3.5.1}, the proof of Theorem \ref{thm3.1.2} is complete.

\bigskip 

{\textbf{Acknowledgements.}} This paper is based on work appearing in the author's Ph.D. thesis at the University of Bristol and was supported by a studentship sponsored by a European Research Council Advanced Grant under the European Union's Horizon 2020 research and innovation programme via grant agreement No. 695223. The author wishes to thank Prof. Trevor D. Wooley for suggesting this area of research and for helpful conversations and comments. The author wishes also to thank Prof. Angel V. Kumchev for feedback on the material of this paper.


\begin{thebibliography}{99}
	
	\bibitem[Apo74]{apostol_book} T. M. Apostol, \emph{Mathematical Analysis}, 2nd Edition, Addison-Wesley, Reading, MA, 1974.
	
	\bibitem[AZ84]{arkhipov_zitkov_warings_probl_non_integr_exp} G. I. Arkhipov and A. N. Zhitkov, \emph{Waring's problem with nonintegral exponent}, Izv. Akad. Nauk SSSR Ser. Mat. \textbf{48} (1984), No.6, 1138--1150.
	
	\bibitem[Bak86]{baker_book_dioph_ineq} R. C. Baker, \emph{Diophantine inequalities}, London Math. Soc. Monographs, New Series, vol. 1.  The Clarendon Press, Oxford, 1986.
	
	\bibitem[BG99]{bentkus_goetze_paper} V. Bentkus, F. G\"{o}tze, \emph{Lattice point problems and distribution of values of quadratic forms}, Ann. of Math. (2) \textbf{150} (1999), 977--1027.
	
	\bibitem[Bir61]{birch_forms} B. J. Birch, \emph{Forms in many variables},  Proc. Roy. Soc. London Ser. A  \textbf{265} (1961/1962), 245--263.
	
	\bibitem[Bra17]{brandes_hasse_princ_quadr_cubic} J. Brandes, \emph{The Hasse principle for systems of quadratic and cubic diagonal equations}, Q. J. Math. \textbf{68} (2017), 831--850.
	
	\bibitem[BP17]{brandes_parsell_simult_addit_eqts} J. Brandes, S. T. Parsell, \emph{Simultaneous additive equations: repeated and differing degrees}, Canad. J. Math.  \textbf{69} (2017),  258--283.
	
	\bibitem[BC91]{brued_cook_pairs_cub} J.  Br{\"u}dern,  R. J. Cook \emph{On pairs of cubic {D}iophantine inequalities}, Mathematika \textbf{38} (1991), 250--263.
	
	\bibitem[BC92]{brued_cook_simult_diag_eq_ineq} J.  Br{\"u}dern,  R. J. Cook \emph{On simultaneous diagonal equations and inequalities}, Acta Arith. \textbf{62} (1992), 125--149.
	
	\bibitem[Cho17]{chow_birch_shifts} S. Chow, \emph{Birch's theorem with shifts}, Ann. Sc. Norm. Super. Pisa Cl. Sci. (5) \textbf{17} (2017), 449--483.
	
	\bibitem[Coo74]{cook_quad_ineq} R. J. Cook \emph{Simultaneous quadratic inequalities}, Acta Arith. \textbf{25} (1974), 337--346.
	
	\bibitem[Dav05]{davenport_book} H. Davenport, \emph{Analytical Methods for Diophantine Equations and Inequalities}, 2nd edn., Cambridge University Press (Cambridge 2005).
	
	\bibitem[DL63]{davenport_lewis_homo_add_eqts} H. Davenport, D. J. Lewis, \emph{ Homogeneous additive equations}, Proc. Roy. Soc. London Ser. A \textbf{274} (1963), 443--460.
	
	\bibitem[DH46]{davenport-heilbr-ineq} H. Davenport, H. Heilbronn, \emph {On indefinite quadratic forms in five variables}, J. London Math. Soc. \textbf{21} (1946), 185-193.
	
	\bibitem[Fre00]{freeman_lower_bounds} D. E. Freeman, \emph{Asymptotic lower bounds for {D}iophantine inequalities}, Mathematika \textbf{47} (2000), 127--159.
	
	\bibitem[Fre01]{freeman_quadr_ineq} D. E. Freeman, \emph{Quadratic {D}iophantine inequalities}, J. Number Theory \textbf{89} (2001), 268--307.
	
	\bibitem[Fre02]{freeman_asymp_formula} D. E. Freeman, \emph{Asymptotic lower bounds and formulas for Diophantine inequalities}, Number theory for the millennium (Urbana, IL, 2000 ), (M.A.Bennett et. al., ed.), vol. 2, 2002, pp.57--74.
	
	\bibitem[Fre03]{freeman_syst_diag_ineq} D. E. Freeman, \emph{Systems of diagonal {D}iophantine inequalities}, Trans. Amer. Math. Soc. \textbf{355} (2003), 2675--2713.
	
	\bibitem[Fre04]{freeman_syst_cubic_ineq2004} D. E. Freeman, \emph{Systems of cubic {D}iophantine inequalities}, J. Reine Angew. Math. \textbf{570} (2004), 1--46.
	
	\bibitem[GK91]{graham_kolesnik_book} S. W. Graham,  G. Kolesnik, \emph{van der Corput's method of exponential sums}, London Mathematical Society Lecture Note Series \textbf{126}, Cambridge University Press, Cambridge 1991.
	
	\bibitem[Mit70]{mitrinovic_book}  D. S. Mitrinovi\'{c}, \emph{Analytic inequalities}, Springer-Verlag, New York-Berlin 1970.
	
	\bibitem[NP89]{nadesalingam_pitman} T. Nadesalingam, J. Pitman, \emph{ Simultaneous diagonal inequalities of odd degree}, J. Reine Angew. Math. \textbf{394} (1989), 118--158.
	
	\bibitem[Par99]{parsell_ineqI} S. T. Parsell, \emph{On simultaneous diagonal inequalities }, J. London Math. Soc. (2) \textbf{60} (1999), 659--676.
	
	\bibitem[Par01]{parsell_ineqII} S. T. Parsell, \emph{On simultaneous diagonal inequalities {II} }, Mathematika \textbf{48} (2001),191--202.
	
	\bibitem[Par02]{parsell_ineqIII} S. T. Parsell, \emph{On simultaneous diagonal inequalities {III} }, Q. J. Math. \textbf{53} (2002), 347--363.
	
	\bibitem[Pit81]{pitman_quadr_diag_ineq_systm1981} J. Pitman, \emph{Pairs of diagonal inequalities}, Recent progress in analytic number theory, {V}ol. 2 ({D}urham,
	1979),  Academic Press, London-New York (1981), 183--215.
	
	\bibitem[Pou21a]{poulias_ineq_frac} C. Poulias, \emph{Diophantine inequalities of fractional degree}, Accepted in Mathematika.
	
	\bibitem[Pou21b]{poulias_approx_TDI_system} C. Poulias, \emph{An approximately translation--dilation invariant system}, Submitted.
	
	\bibitem[Ryd18]{myerson_quadratic} S. L. Rydin Myerson, \emph{Quadratic forms and systems of forms in many variables}, Invent. Math. \textbf{213} (2018), 205--235.
	
	\bibitem[Ryd19]{myerson_cubic} S. L. Rydin Myerson, \emph{Systems of cubic forms in many variables}, J. Reine Angew. Math. \textbf{757} (2019), 309--328.
	
	\bibitem[Sch80]{schmidt_systm_odd_degree_ineq} W. M. Schmidt, \emph{Diophantine inequalities for forms of odd degree}, Adv. in Math. \textbf{38} (1980), 128--151.
	
	\bibitem[Sch82]{schmidt_simult_rat_zero_quad_forms} W. M. Schmidt, \emph{Simultaneous rational zeros of quadratic forms}, Seminar on Number Theory, Paris, 1980-81, Progress in Mathematics, 22, pp. 281--307, Birkh{\"a}user, Boston, MA. (1982).
	
	\bibitem[Seg33a]{segal_1933_I} B. I. Segal, \emph{Sur la distribution des valeurs d'une certaine fonction}, Travaux Inst. Physico-Math. Stekloff, Acad. Sci. USSR \textbf{4} (1933), 37--48.
	
	\bibitem[Seg33b]{segal_1933_II} B. I. Segal, \emph{ Sur un th\'eor\`eme g\'en\'erate de la th\'eorie additive des nombres}, Travaux Inst. Physico-Math. Stekloff, Acad. Sci. USSR \textbf{4} (1933), 49--62.
	
	\bibitem[Seg33c]{segal_1934_ineq} B. I. Segal, \emph{Warings theorem for degrees with fractional and irrational exponents}, Travaux Inst. Physico-Math. Stekloff, Acad. Sci. USSR \textbf{5} (1934), 73--86.
	
	\bibitem[Ste93]{stein_book} E. M. Stein, \emph{Harmonic analysis: real-variable methods, orthogonality, and oscillatory integrals}, With the assistance of Timothy S. Murphy, Monographs in Harmonic Analysis, III, Princeton Mathematical Series, vol. 43, Princeton University Press, Princeton NJ, 1993
	
	\bibitem[Tit86]{titchmarsh_book} E.C. Titchmarsh, \emph{The theory of the Riemann zeta-function}, Edition 2, London Math. Soc. Monographs, New Series,  The Clarendon Press, Oxford, 1986.
	
	\bibitem[VdC35]{VdC_osc_int} J. G. van der Corput, \emph{Zur {M}ethode der station{\"a}ren {P}hase. {E}rste {M}itteilung, {E}infache {I}ntegrale}, Compositio Math. \textbf{1} (1935), 15--38.
	
	\bibitem[Vau97]{vaughan_book} R. C. Vaughan, \emph{The Hardy--Littlewood method}, Cambridge Tracts in Mathematics \textbf{125}, Cambridge University Press, 1997.
	
	\bibitem[Wat89]{watt-exp-sums-riemann-zeta-fnct2} N. Watt, \emph {Exponential sums and the {R}iemann zeta-function. {II}}, J. London Math. Soc. \textbf{39} (1989), no.3, 385--404.
	
	\bibitem[Woo91]{wooley_simult_eqts_91} T. D. Wooley, \emph{ On simultaneous additive equations {II}}, J. Reine Angew. Math. \textbf{419} (1991), 141 -- 198.
	
	\bibitem[Woo97]{wooley_exp_sum_smooth_numb} T. D. Wooley, \emph{On exponential sums over smooth numbers}, J. Reine Angew. Math.  \textbf{488} (1997), 79 -- 140.
	
	\bibitem[Woo98]{wooley_simult_addt_eqtIV} T. D. Wooley, \emph{On simultaneous additive equations. {IV}}, Mathematika  \textbf{45} (1998), 319--335. 
	
	\bibitem[Woo03]{wooley_dioph_ineq} T.D. Wooley, On {D}iophantine inequalities: Freeman's asymptotic formulae, Proceedings of the {S}ession in {A}nalytic {N}umber {T}heory and {D}iophantine {E}quations (Bonn January-June 2002) (D.R.Heath-Brown and B.Z.Moroz), no. 360, Bonner Mathematische Schriften, 2003.
	
	\bibitem[Woo12]{wooley_IMRN} T. D. Wooley, \emph{The asymptotic formula in {W}aring's problem}, Int. Math. Res. Not. IMRN  (2012), 1485--1504.
	
	\bibitem[Woo15]{wooley_pair_quadr_cub_2015} T. D. Wooley, \emph{Rational solutions of pairs of diagonal equations, one cubic and one quadratic},Proc. Lond. Math. Soc. (3) \textbf{110} (2015), 325--356. 
	
	\bibitem[Woo19]{wooley_NEC} T. D. Wooley, \emph{Nested efficient congruencing and relatives of {V}inogradov's mean value theorem}, Proc. Lond. Math. Soc. (3) \textbf{118} (2019), 942--1016.
	
\end{thebibliography}
\end{document}